\documentclass{article}
\usepackage{amsmath}
\usepackage{amssymb}
\usepackage{mathrsfs}
\usepackage{amsfonts}
\usepackage{graphicx}
\usepackage{amsthm}
\usepackage[colorlinks,linkcolor=blue,anchorcolor=red,citecolor=blue]{hyperref}
\usepackage{geometry}
\theoremstyle{plain}
\newtheorem{Thm}{Theorem}[section]
\newtheorem{Def}[Thm]{Definition}
\newtheorem{Lem}[Thm]{Lemma}
\newtheorem{Coro}[Thm]{Corollary}
\newtheorem{Rem}[Thm]{Remark}
\newtheorem{Prop}[Thm]{Proposition}

\usepackage{authblk}

\newcommand{\qe}{}

\newcommand{\tr}{\text{tr}}
\newcommand{\adj}{\text{adj}}
\newcommand{\R}{\mathbf{R}}
\newcommand{\M}{\mathbf{M}}
\newcommand{\Rd}{{\mathbf{R}^d}}

\newcommand{\C}{\mathbf{C}}
\renewcommand{\P}{\mathbf{P}}
\renewcommand{\S}{\mathbf{S}}
\renewcommand{\Re}{\text{Re}}
\renewcommand{\Im}{\text{Im}}
\newcommand{\F}{\mathscr{F}}
\newcommand{\N}{\mathbf{N}}

\usepackage{cite}

\DeclareMathOperator*{\weaklim}{weak-lim}
\DeclareMathOperator*{\slim}{s-lim}

\allowdisplaybreaks
\newcommand{\B}{\widehat{B}(y)}
\newcommand{\BB}{\widehat{B_0}(y)}
\newcommand{\A}{\widehat{A}(y)}
\newcommand{\Ker}{\text{Ker}}
\newcommand{\<}{\langle}
\renewcommand{\>}{\rangle}

\begin{document}

\title{On the Cauchy problem of the Boltzmann equation with a very soft potential}
\author{Dingqun Deng\\
Department of Mathematics, City University of Hong Kong\\
e-mail: dingqdeng2-c@my.cityu.edu.hk}


\maketitle

\begin{abstract}
The Cauchy problem for the Boltzmann equation with soft potential, in the framework of small perturbation of an equilibrium state, has been studied in many spaces. 
The method of strongly continuous semigroup has been applied by Caflisch\cite{Caflisch1980a} and Ukai-Asano\cite{Ukai1982} for the case of soft potential, where they obtained the $L^\infty$ solution without requiring any velocity deviation.
By generalizing the estimate on linearized collision operator $L$ to the case of very soft potential, we obtain a similar global existence result for $\gamma\in[0,d)$.
For soft potential, the spectrum structure of the linearized Boltzmann operator couldn't give spectral gap, so we use the method of integration by parts and consider a weighted velocity space in order to obtain algebraic decay in time.

\end{abstract}

\tableofcontents

\section{Introduction}

Consider the Cauchy problem of the Boltzmann equation in $d$-dimension:
\begin{align}\label{I_e1}
  f_t + \xi\cdot\nabla_x f = Q(f,f),\qquad f|_{t=0}=f_0.
\end{align}
Here $f = f(x,\xi,t)$ is the distribution function of particle at position $x\in\Rd$ with velocity $\xi\in\Rd$ at time $t\ge 0$, $Q(f,g)$ is the bilinear collision operator defined by
\begin{align}
  Q(f,g) := \int_{\Rd}\int_{S^{d-1}} (f'_*g' + f'g'_* - f_*g - fg_*) q(\xi-\xi_*,\theta)\, d\omega d\xi_*
\end{align}
where
\begin{align*}
  f'_* = f(x,\xi'_*,t),\ \
  f' = f(x,\xi',t),\ \
  f_* = f(x,\xi_*,t),\ \
  f = f(x,\xi,t),
\end{align*}
and similarly for $g$, and
\begin{align}
  \xi' = \xi - \left((\xi-\xi_*)\cdot\omega \right)\omega,\quad
  \xi'_* = \xi_* + \left((\xi-\xi_*)\cdot\omega \right)\omega,
\end{align}
where $\omega\in S^{d-1}$. Here
$(\xi,\xi_*)$ are the velocities of two gas particles before collision and
$(\xi',\xi'_*)$ are the velocities after collision,
while $\theta\in [0,\pi]$ is the angle between the first variable of $q$ and $\omega$. For example, if $q=q(\xi-\xi_*,\theta)$, then
\begin{align}
  \cos\theta = \frac{(\xi - \xi_*)\cdot\omega}{|\xi - \xi_*|},
\end{align}

The function $q$ is the collision kernel determined by the interaction potential model between two colliding particles.
In this paper, we will only consider the Grad's angular cut-off assumption.
That is,
\begin{align}\label{I_cutoffassumption}
  q(\xi-\xi_*,\theta) = |\xi-\xi_*|^{-\gamma} b(\cos\theta),
\end{align}
with $|b(\cos\theta)|\le q_0|\cos\theta|$, for some $q_0>0$ and also $q$ is almost everywhere positive.
We usually call it hard potential if $\gamma\in[-1,0)$ and soft potential if $\gamma\in(0,1)$ and very soft potential if $\gamma\in(0,d)$.

We are looking for a solution $f$ near the equilibrium, that is the global Maxwellian $\M$. Suppose the solution has the form
\begin{align}\label{I_e2}
  f = \M + \M^{\frac{1}{2}}g.
\end{align}
By a translation and scaling of the velocity variable $\xi$, without loss of generality, we will take
\begin{align}
  \M = \frac{1}{(2\pi)^{d/2}} \exp (-\frac{|\xi|^2}{2}).
\end{align}
Substitute \eqref{I_e2} into \eqref{I_e1}, and notice a global Maxwellian is collision invariant, we have
\begin{align}\label{I_intro1}
  g_t + \xi\cdot\nabla_x g = Lg + \Gamma(g,g),
\end{align}
where $\Gamma(g,h):=\M^{-1/2}Q(\M^{1/2}g,\M^{1/2}h)$ and $L$ is a linear operator defined by
\begin{align*}
  Lg := \M^{-1/2}\left( Q(\M^{1/2}g,\M) + Q(\M,\M^{1/2}g)\right).
\end{align*}

{\bf Notation} We introduce a weighted normed space for finding the solution as follows. Define
$H^l(\R^d_x)$ to be the standard Sobolev space and
\begin{align*}
  \|f\|_{L^p_\beta} :&= \|(1+|\xi|)^\beta f\|_{L^p},\\
  L^p_\beta(\Rd) :&= \{f:\Rd\to\C\ |\ \|f\|_{L^p_\beta}<\infty\}.
\end{align*}
Also we denote some multi-variable space.
\begin{align*}
  L^p_\beta(H^l):&= L^p_\beta(\R^d_\xi;(H^l(\R^d_x))),\\ L^\infty_\alpha(L^p_\beta(H^l)) :&=L^\infty_\alpha(\R_t;L^p_\beta(\R^d_\xi;(H^l(\R^d_x)))).
\end{align*}
In addition, for any linear operator $T$ acting on normed space $X$, we denote its resolvent set, spectrum and point spectrum respectively by
\begin{align*}
  \rho(T) :&=\{\lambda\in\C : \lambda I-T \text{ is bijective and }(\lambda I- T)^{-1} \text{ is continuous on }X\},\\
  \sigma(T) :&= \C\setminus\rho(T),\\
  \sigma_p(T) :&= \{\lambda\in\sigma(T) : \lambda\text{ is an eigenvalue of }T\}.
\end{align*}
Also we define the a half space in $\C$ by
\begin{align*}
  \C_+ := \{\lambda\in\C:\Re\lambda>0\}.
\end{align*}
Let $X$ to be a metric space and $Y$ to be a normed space and define $C(X;Y)$ and $BC(X;Y)$ as the following.
\begin{align*}
C(X;Y)&=\{f:X\to Y \,|\,f \text{ is continuous from }X\text{ to }Y\},\\
BC(X;Y)&=\{f\in C(X;Y) \,|\,\sup_{x\in X}\|f\|_Y<\infty\}.
\end{align*}

The goal of present paper is to find a solution to the Cauchy problem of equation \eqref{I_intro1} by applying semigroup theory to operator
\begin{align*}
  B = -\xi\cdot\nabla_x + L,
\end{align*}where we regard $B$ as an operator acting on $L^p_\beta(H^l)$.
The properties of $L$ has been well studied by Caflisch in \cite{Caflisch1980} and by Ukai-Asano in \cite{Ukai1982}. The linearized collision operator has expression $L=-\nu+K$, where $\nu$ is a function satisfying
$\nu(\xi)\sim (1+|\xi|)^\gamma$ and $K$ is an integral operator with kernel $k$.
About the global existence of Cauchy problem for Boltzmann equation \eqref{I_intro1}, some previous good result are established in different situation. Caflisch\cite{Caflisch1980a} gives a solution with exponential time decay in the case of $\gamma\in(0,1)$ on torus $\mathbf{T}_x$ when the initial data has exponential velocity decay.
Ukai-Asano\cite{Ukai1982} proves the existence in space $L^\infty_\alpha(L^\infty_\beta(H^l))$, the same as this paper, when $\gamma\in(0,1)$ and gives algebraic time decay.
Guo\cite{Guo2003} and Strain-Guo\cite{Strain2007} find global solution on torus $\mathbf{T}_x$ for very soft potential $\gamma\in(0,3)$, but requiring higher order of derivatives in velocity of initial data. They also find the exponential time decay when the initial data has exponential velocity decay.

In the present paper, we will establish the global existence to the Cauchy problem of equation \eqref{I_intro1} in space $L^\infty_\alpha(L^\infty_\beta(H^l))$ without any velocity derivation, and finally we can give a algebraic time decay.

Firstly, we generalize the estimate of $K$, defined in \ref{I_ThmL1}, given by Caflisch\cite{Caflisch1980} and Ukai-Asano\cite{Ukai1982} from $\gamma\in[0,1)$ to $\gamma\in[0,d)$. They are:
\begin{align*}
  &\|Kf\|_{L^p_{\beta+\gamma+2}}\le C_{\gamma,d,q,p}\|f\|_{L^p_\beta},\\
  &\|Kf\|_{L^{q_\theta}_{\beta+\gamma+1}}\le C_{\gamma,d,q,p,\theta}\|f\|_{L^{p_\theta}_\beta},
\end{align*}
  where $\frac{1}{q_\theta} = \frac{\theta}{\infty} + \frac{1-\theta}{1},
    \frac{1}{p_\theta} = \frac{\theta}{\frac{d}{d-2}+\frac{d}{\gamma}} + \frac{1-\theta}{1}.$
  We found that these estimates are still valid for $\gamma\in[0,d)$, by using a slightly different technique. 
With the good properties of $K$, we can establish the estimate on semigroup $e^{tB}$. To do this, we will define $P:L^2\to$ Ker$L$ to be the orthogonal projection from $L^2$ to the kernel of $L$ and
\begin{align*}
  \widehat{B}(y) :&= -2\pi i y\cdot\xi + L,\\
  \BB :&= -2\pi i y\cdot\xi + L -P
\end{align*}
as two linear unbounded operators acting on $L^2_\beta(\R^d_\xi)$.
In order to obtain the estimate of semigroup $e^{t\B}$, we need to analyze the behavior of the resolvent $(\lambda I - \B)^{-1}$ of $\B$. When $|y|$ is large, the resolvent of $\B$ has a good property obtained in \cite{Ukai1982}. While $(\lambda I - \B)^{-1}$ has a singular behavior near $|y|=0$.
Using resolvent identity, we have
\begin{align*}
  &(\lambda I - \B)^{-1}\\ &= (\lambda I - \BB)^{-1}
  + (\lambda I - \BB)^{-1}P(I-P(\lambda I - \BB)^{-1}P)^{-1}P(\lambda I - \BB)^{-1}.
\end{align*}
So the singularity near $y=0$ comes from the operator $(I-P(\lambda I - \BB)^{-1}P)^{-1}$ acting on Ker$L$.
In this paper, we will follow Ukai-Asano's idea in \cite{Ukai1982} but use a different approach to obtain the eigenvalues of $P(\lambda I - \BB)^{-1}P$ and its asymptotic behavior. The method in this paper is similar to \cite{Liu2011}.
Wirte $r=|y|$. We will find in section 4 that the singular points of $(I-P(\lambda I - \BB)^{-1}P)^{-1}$ near $y=0$ are $\lambda_j(r) = \sigma_j(r)+i\tau_j(r)\in C^\infty(B(0,r_2))$ for some small $r_2$ and their asymptotic behavior near $r=0$ are
  \begin{align*}
    \sigma_j(r) &= -\sigma_j^{(2)}r^2 + O(r^3),\\
    \tau_j(r) &= \tau_j^{(1)}r + O(r^3),
  \end{align*}as $r\to 0$,
  where $\tau_j^{(1)}\in\R$, $\sigma_j^{(2)}<0$.
  They have the same asymptotic behaviors as the eigenvalues of $(\lambda I - \B)^{-1}$ given in \cite{Ellis1975,Yang2016}.
  So it turns out that the singular behavior of $(I-P(\lambda I - \BB)^{-1}P)^{-1}$ is exactly the same as $(\lambda I - \B)^{-1}$ near $y=0$.

With the well-studied properties on operator $(\lambda I - \B)^{-1}$, one can get the estimate on semigroup $e^{t\B}$ by using the inversion formula of semigroup. For the soft potential, there's no spectral gap to get a good decay on time $t$ as in the hard potential case. However, we can use the method of integration by parts to construct a algebraic decay on time $t$ with a stronger assumption on initial data $f_0$. So we will use the weighted normed space $L^2_\beta$ on velocity to find our solution. Combining the inverse Fourier transform formula as well as the Duhamel formula, we can get a good boundedness on semigroup $e^{tB}$.

Once we have the estimate on semigroup $e^{tB}$, we can obtain our global existence result.
\begin{Thm}	\label{Main_Theorem} Assume the cross-section $q$ satisfies the angular cut-off assumption \eqref{I_cutoffassumption}.
	Assume $d\ge 3$, $\gamma\in[0,d)$, $l>\frac{d}{2}$, $\beta>\frac{d}{2}$, $p\in[1,2)$ such that $\frac{d}{4}(\frac{2}{p}-1)>1/2$. Let $\alpha\in [\frac{1}{2},\min(\frac{d}{4}(\frac{2}{p}-1),1))$. 
	There exists constants $A_0<1, A_1$ such that if the initial data $f_0\in L^\infty_{\beta+\alpha\gamma}(H^l)\cap L^2_{\alpha\gamma}(L^p)$ satisfies
	\begin{align}
	\|f_0\|_{L^\infty_{\beta+\alpha\gamma}(H^l)} + \|f_0\|_{L^2_{\alpha\gamma}(L^p)}\le A_0.
	\end{align}
	Let $X=\{f\in L^\infty_{\alpha}(L^\infty_{\beta}(H^l)): \|f\|_{L^\infty_{\alpha}(L^\infty_{\beta}(H^l))}\le A_1\}$.
	Then the Cauchy problem to Boltzmann equation
	\begin{equation}\left\{
	\begin{aligned}
	f_t + \xi\cdot\nabla_x f &= Lf + \Gamma(f,f),\\
	f|_{t=0} &= f_0.
	\end{aligned}\right.
	\end{equation}
	posseses a unique solution $f=f(t)\in X \cap BC^0([0,\infty);L^\infty_{\beta}(H^l))\cap BC^1([0,\infty);L^\infty_{\beta-1}(H^{l-1}))$ and
	\begin{align}
	\|f\|_{L^\infty_{\alpha}(L^\infty_{\beta}(H^l))} + \|\partial_tf\|_{L^\infty_{\alpha}(L^\infty_{\beta-1}(H^{l-1}))}
	&\le C_{\nu,\gamma,\alpha,\beta,d}\Big(\|f_0\|_{L^\infty_{\beta+\alpha\gamma}(H^l)} + \|f_0\|_{L^2_{\alpha\gamma}(L^p)}\Big).
	\end{align}The uniqueness is taken in the sense that
$f\in X$.
\end{Thm}

Finally, we present the main strategy of analysis in this paper.
In section 2, we prove the estimate on $K$ when $\gamma \in [0,d)$.
Section 3 presents some boundedness and invertibility result on resolvents of $\B$, $\BB$. In section 4, we will analyze the singular behavior of $(\lambda I - \B)^{-1}$ near $y=0$. Section 4 and 5 give the main estimate on our semigroup $e^{t\B}$ and $e^{tB}$ as well as the global existence result.
In appendix, we list some basic theorem on semigroup theory and linearized collision operator $L$.
Also we extend two useful results on the Hilbert-Schmidt operator and interpolation theory in order to make our arguments valid.

\section{Properties of the Linearized Collision Operator}

In this section, we should firstly list some properties of the linearized collision operator $L$ and derive the new estimate on operator $K$.
Suppose the collision kernel satisfies the Grad's angular cut-off assumption:
\begin{align}
  q(\xi-\xi_*,\theta) = |\xi-\xi_*|^{-\gamma} b(\cos\theta),
\end{align}
where $|b(\cos\theta)|\le q_0|\cos\theta|$, for some constant $q_0>0$.

Denote $\P_{\xi_*-\xi}$ to be the hyperplane in $\Rd$ which is orthogonal to the vector $\xi_*-\xi$ and contains the origin. Denote
\begin{align*}
  a &= \frac{\xi+\xi_*}{2} - \left(\frac{\xi+\xi_*}{2}\cdot\frac{\xi_*-\xi}{|\xi_*-\xi|}\right)\frac{\xi_*-\xi}{|\xi_*-\xi|},\\
  b &= \left(\frac{\xi+\xi_*}{2}\cdot\frac{\xi_*-\xi}{|\xi_*-\xi|}\right)\frac{\xi_*-\xi}{|\xi_*-\xi|},
\end{align*}
where the vector $a$ is the projection of $\dfrac{\xi+\xi_*}{2}$ onto the hyperplane $\P_{\xi_*-\xi}$,
while $b$ is its projection onto the direction $\xi_*-\xi$.

The following theorem shows the basic properties of linearized collision operator $L$.
Here I will only give proof of estimate \eqref{I_eq412}, \eqref{I_eq413} and \eqref{I_eq214}, since the other estimate has been well studied in \cite{Caflisch1980} and I will put them in appendix for the sake of completeness.
\begin{Thm}\label{I_ThmL1}(Properties of $L$). Assume $\gamma\in [0,d)$,
  The linear operator $L$ has expression
  \begin{align*}
    Lf = Kf - \nu f.
  \end{align*}

  (1). Here $\nu(\xi)$ is a real positive function defined by
  \begin{align*}
  	\nu(\xi) = \int_\Rd\int_{\S^{d-1}}\M^{1/2}_*q(\xi-\xi_*,\theta)\,d\omega d\xi_*.
  \end{align*}
  For $\gamma\in[0,d)$, there exist constants $\nu_0,\,\nu_1>0$ depending on $\gamma,q,d$ such that
  \begin{align*}
    \nu_0(1+|\xi|)^{-\gamma} \le \nu(\xi) \le \nu_1(1+|\xi|)^{-\gamma}.
  \end{align*}

  (2). The operator $K$ is linear continuous operator on $L^2_\beta(\Rd)$ defined by 
  \begin{align*}
    Kf(\xi) = \int_{\Rd} k(\xi,\xi_*)f(\xi_*)\,d\xi_*.
  \end{align*}
  The kernel $k$ can be divided as $k(\xi,\xi_*) = k_1(\xi,\xi_*) + k_2(\xi,\xi_*)$, with
  \begin{align*}
    k_1(\xi,\xi_*) &= \frac{1}{|\xi_*-\xi|^{d-1}}\int_{\P_{\xi_*-\xi}}(2\pi)^{-d/2}e^{-\frac{|x|^2}{2}} q(x-a+\xi_*-\xi,\theta)\,dx \exp\left(-\frac{|b|^2}{2}-\frac{|\xi_*-\xi|^2}{8}\right),\\
    k_2(\xi,\xi_*) &= -\int_{S^{d-1}} \M^{1/2}(\xi_*)\M^{1/2}(\xi)q(\xi_*-\xi,\theta)\,d\omega.
  \end{align*}
  Here $k_1,\,k_2$ are symmetric functions. For $0<\varepsilon<1$, they satisfy:
  \begin{align}\label{I_eq411}
    |k_1(\xi,\xi_*)| &\le C_{\gamma,\varepsilon,d,q} \frac{1}{|\xi_*-\xi|^{d-2}(1+|\xi|+|\xi_*|)^{\gamma+1}}\exp\left(-\left(1-\varepsilon\right)
    \left(\dfrac{|b|^2}{2}+\dfrac{|\xi_*-\xi|^2}{8}\right)\right),\\\label{I_eq412}
    |k_2(\xi,\xi_*)| &\le C_{\gamma,\varepsilon,d,q} \frac{1}{|\xi_*-\xi|^\gamma(1+|\xi|+|\xi_*|)^{\gamma+1}}\exp\left(-(1-\varepsilon)
    \frac{|\xi_*|^2+|\xi|^2}{4}\right).
  \end{align}
  Consequently, for $p\in[1,\min\{\frac{d}{d-2}, \frac{d}{\gamma}\})$, $\beta\in\R$, we have
  \begin{align}\label{I_eq413}
    |k(\xi,\xi_*)| \le C_{\gamma,\varepsilon,d,q} \left(\frac{1}{|\xi_*-\xi|^{d-2}}+  \frac{1}{|\xi_*-\xi|^{\gamma}} \right) \frac{\exp \left(-(1-\varepsilon) (\frac{|b|^2}{2}+\frac{|\xi_*-\xi|^2}{8})\right)}{(1+|\xi|+|\xi_*|)^{\gamma+1}}
    ,
    \end{align}\begin{align}\label{I_eq214}
  \int_{\Rd}(1+|\xi_*|)^\beta|k(\xi,\xi_*)|^p\,d\xi_*
  \le C_{\gamma,\varepsilon,d,q}\frac{1}{(1+|\xi|)^{-\beta+p(\gamma+1)+1}}.
\end{align}
\end{Thm}
\begin{proof}

1. The expression of $K$ can be found in appendix.
For any $\varepsilon\in(0,1)$,
\begin{align*}
  |k_2(\xi,\xi_*)| &= \left|\int_{S^{d-1}} \M^{1/2}(\xi_*)\M^{1/2}(\xi)q(\xi_*-\xi,\theta)\,d\omega\right|\\
  &= (2\pi)^{-d/2}\exp(-\frac{|\xi_*|^2+|\xi|^2}{4})|\xi_*-\xi|^{-\gamma} \int_{S^{d-1}}b(\cos\theta)\,d\omega\\
  &= C_{d,q_0,\varepsilon}\exp(-(1-\varepsilon)\frac{|\xi_*|^2+|\xi|^2}{4})|\xi_*-\xi|^{-\gamma}(1+|\xi|+|\xi_*|)^{-\gamma-1}\\
  &\qquad\qquad\times\sup_{\xi,\xi_*\in\Rd}\exp(-\frac{\varepsilon(|\xi_*|^2+|\xi|^2)}{4})(1+|\xi|+|\xi_*|)^{\gamma+1}\\
  &= C_{d,\gamma,q_0,\varepsilon}\exp\left(-(1-\varepsilon)\frac{|\xi_*|^2+|\xi|^2}{4}\right)|\xi_*-\xi|^{-\gamma}(1+|\xi|+|\xi_*|)^{-\gamma-1}.
\end{align*}
This proves \eqref{I_eq412}.

2. Once we get the estimate \eqref{I_eq411} and \eqref{I_eq412}, by using a trivial inequality that
\begin{align*}
  \frac{|\xi_*|^2+|\xi|^2}{4}-\frac{|\xi_*-\xi|^2}{8} = \frac{|\xi_*+\xi|^2}{8} \ge \frac{|b|^2}{2},
\end{align*}we have
\begin{align*}
  |k(\xi,\xi_*)| &\le C_{\gamma,\varepsilon,d,q}\left(\frac{1}{|\xi_*-\xi|^{d-2}}+\frac{1}{|\xi_*-\xi|^\gamma}\right) \frac{\exp\left(-(1-\varepsilon)(\frac{|b|^2}{2}+\frac{|\xi_*-\xi|^2}{8})\right)}{(1+|\xi|+|\xi_*|)^{\gamma+1}}
  .
\end{align*}
Therefore, by \eqref{I_lemma13} in lemma \ref{I_lemma1}, for $p\in[1,\min\{\frac{d}{d-2}, \frac{d}{\gamma}\})$,
\begin{align*}
  \int_{\Rd}(1+|\xi_*|)^\beta|k(\xi,\xi_*)|^p\,d\xi_*
  &\le C_{\gamma,\varepsilon,d,q}\int_{\Rd}\left(\frac{1}{|\xi_*-\xi|^{p(d-2)}}+\frac{1}{|\xi_*-\xi|^{p\gamma}}\right) \frac{(1+|\xi_*|)^{\beta}}{(1+|\xi|+|\xi_*|)^{p(\gamma+1)}}\\
  &\qquad\qquad\qquad\qquad\qquad\times\exp\big(-p(1-\varepsilon)\big(\frac{|b|^2}{2}+\frac{|\xi_*-\xi|^2}{8}\big)\big)\,d\xi_*\\
  &\le C_{\gamma,\varepsilon,d,q}\frac{1}{(1+|\xi|)^{-\beta+p(\gamma+1)+1}}.
\end{align*}This completes the proof.
\qe\end{proof}

\begin{Rem}
  Similar estimates are valid for $\gamma\in[-1,0]$, which is the case of hard potential. The only difference is that for hard potential, the term $\frac{1}{|\xi-\xi_*|^\gamma}$ in \eqref{I_eq412} is not singular any more.
\end{Rem}

\begin{Thm}\label{I_proerties_K}(Properties of $K$).
  Let $d\ge 3$ be the dimension, $\gamma \in [0,d)$, $\beta\in\R$.
  Then the followings are valid.

  (1). For $p\in[1,\infty]$,
  \begin{align}
    \|Kf\|_{L^p_{\beta+\gamma+2}}\le C_{\gamma,d,q,p}\|f\|_{L^p_\beta}.
  \end{align}

  (2). The linear operator $K:L^2_\alpha\to L^2_{\beta+\gamma+2}$ is compact for $\alpha>\beta$.

(3). Let $p>\max(\frac{d}{d-\gamma},\frac{d}{2})$. Then
\begin{align}
  \|Kf\|_{L^\infty_{\beta+\gamma+2-1/p}}\le C_{\gamma,d,q}\|f\|_{L^p_\beta}.
\end{align}

(4). Pick $p_0>\max(\frac{d}{d-\gamma},\frac{d}{2})$. For $\theta\in(0,1)$,
$K$ is a linear bounded operator from $L^{p_\theta}$ to $L^{q_\theta}$ with estimate
\begin{align}\label{I_eq333}
  \|Kf\|_{L^{q_\theta}_{\beta+\gamma+1}}\le C_{\gamma,d,q,p,\theta}\|f\|_{L^{p_\theta}_\beta},
\end{align}
where
\begin{align}\label{I_eqpq}
    \frac{1}{q_\theta} = \frac{\theta}{\infty} + \frac{1-\theta}{1},
    \quad\frac{1}{p_\theta} = \frac{\theta}{p_0} + \frac{1-\theta}{1}.
\end{align}
Consequently, if $f$ lies in the space $L^p_\beta(H^l)$ or $L^{p_\theta}_\beta(H^l)$, then we have the following estimate respectively.
\begin{align*}
  \|Kf\|_{L^\infty_{\beta+\gamma+2-1/p}(H^l)}\le C_{\gamma,d,q}\|f\|_{L^p_\beta(H^l)},\\
  \|Kf\|_{L^{q_\theta}_{\beta+\gamma+1}(H^l)}\le C_{\gamma,d,q,p_0,\theta}\|f\|_{L^{p_\theta}_\beta(H^l)}.
\end{align*}

\end{Thm}
\begin{Rem}
  All the estimates are true on $\int_\Rd|k(\xi,\xi_*)|f(\xi_*)\,d\xi_*$ instead of $\int_\Rd k(\xi,\xi_*)f(\xi_*)\,d\xi_*$ as well. This property is useful in some special situation.
\end{Rem}

\begin{proof}
  1. Let $r\in\R$ to be arbitrary, $\beta\in\R$, $p\in[1,\infty)$. Applying H\"older's inequality, \eqref{I_eq214} and noticing $1/p+1/p'=1$, we have
  \begin{align*}
    |Kf(\xi)| &\le \int_\Rd|k(\xi,\xi_*)|(1+|\xi_*|)^{-r}(1+|\xi_*|)^r|f(\xi_*)|\,d\xi_*\\
    &\le \left(\int_\Rd |k(\xi,\xi_*)|(1+|\xi_*|)^{-rp'}\,d\xi_*\right)^{1/p'}
    \left(\int_\Rd |k(\xi,\xi_*)|(1+|\xi_*|)^{rp}|f(\xi_*)|^p\,d\xi_*\right)^{1/p}\\
    &= \frac{C_{\gamma,d,q,r,p}}{(1+|\xi|)^{r+(\gamma+2)(p-1)/p}}
    \left(\int_\Rd |k(\xi,\xi_*)|(1+|\xi_*|)^{rp}|f(\xi_*)|^p\,d\xi_*\right)^{1/p}.
  \end{align*}
  Thus using \eqref{I_eq214} again,
  \begin{align*}
    \|Kf\|^p_{L^p_\beta(\Rd)}
    &\le C_{\gamma,d,q,r,p}\int_\Rd\left(\int_\Rd(1+|\xi|)^{p\beta-pr-(\gamma+2)(p-1)} |k(\xi,\xi_*)|\,d\xi\right) (1+|\xi_*|)^{rp}|f(\xi_*)|^p\, d\xi_*\\
    &\le C_{\gamma,d,q,r,p}\|f\|^p_{L^p_{\beta-\gamma-2}}.
  \end{align*}

  If $p=\infty$, then
\begin{align*}
  \|Kf\|_{L^\infty_\beta}
  &\le \sup_{\xi\in\Rd}(1+|\xi|)^\beta\int_\Rd (1+|\xi_*|)^{\gamma+2-\beta}|k(\xi,\xi_*)|\,d\xi_*\ \|f\|_{L^\infty_{\beta-\gamma-2}}
  \le C_{d,\gamma,q,p}\|f\|_{L^\infty_{\beta-\gamma-2}}.
\end{align*}

2. Now we prove that $K:L^2_\alpha\to L^2_{\beta}$ is compact for $\alpha>\beta-\gamma-2$. Similar to the estimates in step 1, for  $R>1/2>\varepsilon$, $r\ge 0$, we have
\begin{align*}
  |K(\chi_{|\xi-\,\cdot\,|\le\varepsilon}\chi_{|\cdot|\ge R}f)(\xi)| &\le \frac{C_{\gamma,d,q,r,p}}{(1+|\xi|)^{r+(\gamma+2)(p-1)/p}}
  \Big(
\int_{
  \substack{|\xi_*|\ge R \\ |\xi-\xi_*|\le \varepsilon}
}
|k(\xi,\xi_*)|(1+|\xi_*|)^{rp}|f(\xi_*)|^p\,d\xi_*\Big)^{\frac{1}{p}}.
\end{align*}
We claim that $K$ is the limit of compact operators $K(\chi_{|\xi-\,\cdot\,|\le\varepsilon}\chi_{|\cdot|\ge R})$ under the operator norm $\|\cdot\|_{L^p_\alpha\to L^p_{\beta}}$. The term $\chi_{|\xi-\,\cdot\,|\le\varepsilon}$ is used to eliminate the singularity of $k(\xi,\xi_*)$ near $\xi=\xi_*$.
Pick $r\in\R$ and notice that
$|\xi-\xi_*|\le\varepsilon<1/2$ implies $\frac{1}{2}(1+|\xi_*|)\le (1+|\xi|)\le \frac{3}{2}(1+|\xi_*|)$. Thus by \eqref{I_eq413},
\begin{align*}
  &\|K(\chi_{|\xi-\,\cdot\,|\ge\varepsilon}\chi_{|\cdot|\le R}f)(\xi)-K(\chi_{|\cdot|\le R}f)(\xi)\|^p_{L^p_{\beta}}\\
  &\le
  C_{\gamma,d,q,r,p}
  \int_{|\xi_*|\le R}
  \Big(\int_\Rd(1+|\xi|)^{p\beta-pr-(\gamma+2)(p-1)}|k(\xi,\xi_*)|
  \chi_{|\xi-\xi_*|\le\varepsilon}\,d\xi\Big)
  (1+|\xi_*|)^{rp}|f(\xi_*)|^p\,d\xi_*\\
  &\le C_{\gamma,d,q,r,p,\beta}
  \int_{|\xi_*|\le R}
  \Big(\int_{|\xi-\xi_*|\le\varepsilon}
  \Big(\frac{1}{|\xi-\xi_*|^\gamma}+\frac{1}{|\xi-\xi_*|^{d-2}}\Big)
  \,d\xi\Big)
  (1+|\xi_*|)^{p\beta-(\gamma+2)p+1}|f(\xi_*)|^p\,d\xi_*\\
  &\le C_{\gamma,d,q,r,p,\beta}
  \int_{|\xi|\le\varepsilon}
  \left(\frac{1}{|\xi|^\gamma}+\frac{1}{|\xi|^{d-2}}\right)
  \,d\xi
  \min\{1,\,(1+R)^{p\beta-(\gamma+2)p+1-p\alpha}\}\|f\|^p_{L^p_\alpha}\\
  &\to 0,
\end{align*}
as $\varepsilon\to 0$, for any fixed $R>1/2$. On the other hand,
\begin{align*}
  &\|K(\chi_{|\cdot|\le R}f)(\xi)-Kf(\xi)\|^p_{L^p_{\beta}}\\
  &\le C_{\gamma,d,q,r,p}\int_\Rd\left(\int_\Rd(1+|\xi|)^{p\beta-pr-(\gamma+2)(p-1)}|k(\xi,\xi_*)|
  \,d\xi\right)
  (1+|\xi_*|)^{rp}|f(\xi_*)|^p\chi_{|\xi_*|\ge R}\,d\xi_*\\
  &\le C_{\gamma,d,q,r,p}\int_\Rd(1+|\xi_*|)^{p\beta-(\gamma+2)p}|f(\xi_*)|^p\chi_{|\xi_*|\ge R}\,d\xi_*\\
  &=C_{\gamma,d,q,r,p}(1+R)^{p\beta-(\gamma+2)p-p\alpha}\|f\|^p_{L^p_{\alpha}},
\end{align*}provided $\beta-(\gamma+2)-\alpha\le 0$.
Thus
\begin{align*}
  \|K(\chi_{|\cdot|\ge R}f)(\xi)\|_{L^p_{\beta}}\le
  C_{\gamma,d,q,r,p}(1+R)^{\beta-(\gamma+2)-\alpha}\|f\|_{L^p_{\alpha}}\to 0,
\end{align*}if $\beta-(\gamma+2)-\alpha<0$.
This proves that $K$ can be approximated by $K(\chi_{|\xi-\,\cdot\,|\le\varepsilon}\chi_{|\cdot|\ge R})$ under the operator norm $\|\cdot\|_{L^p_\alpha\to L^p_{\beta}}$.

Let $p=2$, it remians to show that $K\chi_{|\xi-\,\cdot\,|\ge\varepsilon}\chi_{|\cdot|\le R}$ is compact in $L(L^2_{\alpha},L^2_{\beta})$ when $\beta-(\gamma+2)-\alpha<0$.
By \ref{A_Hilbert_Schmidt} in appendix, it suffices to prove that $K\chi_{|\xi-\,\cdot\,|\ge\varepsilon}\chi_{|\cdot|\le R}$ is a Hilbert-Schimidt operator.
That is to show that  $k(\xi,\xi_*)\chi_{|\xi-\xi_*|\ge\varepsilon}\chi_{|\xi_*|\le R}\in L^2(\Rd\times\Rd,(1+|\xi_*|)^{2\alpha}\,d\xi_*\otimes(1+|\xi|)^{2\beta}\,d\xi)$.
\begin{align*}
  &\int_\Rd\int_\Rd|k(\xi,\xi_*)|^2\chi_{|\xi-\xi_*|\ge\varepsilon}\chi_{|\xi_*|\le R}(1+|\xi|)^{2\beta}\,d\xi(1+|\xi_*|)^{2\alpha}\,d\xi_*\\
  &\le
  C_{\gamma,\varepsilon,d,q}
  \int_{|\xi_*|\le R}\int_{|\xi-\xi_*|\ge\varepsilon}
  \left(\frac{1}{|\xi-\xi_*|^\gamma}+\frac{1}{|\xi-\xi_*|^{d-2}}\right)^2
  \frac{(1+|\xi|)^{2\beta}(1+|\xi_*|)^{2\alpha}}{(1+|\xi|+|\xi_*|)^{2\gamma+2}}\\
  &\qquad\qquad\qquad\qquad\qquad\qquad\times\exp\left(-2(1-\varepsilon)\left(|b|^2/2+|\xi_*-\xi|^2/8\right)\right)\,d\xi\,d\xi_*\\
  &\le   C_{\gamma,\varepsilon,d,q}\left(\frac{1}{\varepsilon^\gamma}+\frac{1}{\varepsilon^{d-2}}\right)^2
  \int_{|\xi_*|\le R}(1+|\xi_*|)^{2\alpha+2\beta-2\gamma-3}\,d\xi_*\\
  &<\infty.
\end{align*}
This shows that $K\chi_{|\xi-\,\cdot\,|\ge\varepsilon}\chi_{|\cdot|\le R}:L^2_\alpha\to L^2_{\beta+\gamma+2}$ is compact for $\alpha>\beta$, and then
$K:L^2_\alpha\to L^2_{\beta+\gamma+2}$ is compact.

3. For $f\in L^p_\beta$, $\beta\in\R$, $p\in (1,\infty)$,
\begin{align*}
  |Kf(\xi)|&\le \int_\Rd|k(\xi,\xi_*)|\,|f(\xi_*)|\,d\xi_*\\
  &\le \left(\int_\Rd|k(\xi,\xi_*)|^{p'}(1+|\xi_*|)^{-p'\beta })\,d\xi_* \right)^{1/p'}
  \left(\int_\Rd(1+|\xi_*|)^{p\beta}|f(\xi_*)|^p\,d\xi_*\right)^{1/p}\\
  &\le C_{\gamma,\varepsilon,d,q,p}
  \frac{1}{(1+|\xi|)^{(\beta+\gamma+1)+1/p'}}
  \|f\|_{L^p_\beta},
\end{align*}provided $p'\gamma<d$ and $(d-2)p'<d$, that is $p>\frac{d}{d-\gamma}$ and $p>\frac{d}{2}$. Thus for $p\in(\max(\frac{d}{d-\gamma},\frac{d}{2}),\infty)$, we have
\begin{align}\label{I_eq229}
  \|Kf\|_{L^\infty_{\beta+\gamma+1+1/p'}}\le C_{\gamma,d,q,p}\|f\|_{L^p_\beta}.
\end{align}

4. To prove (4), we only need a weaker result then \eqref{I_eq229}. That is for $p\in(\max(\frac{d}{d-\gamma},\frac{d}{2}),\infty)$, $\beta\in\R$,
\begin{align*}
  \|Kf\|_{L^\infty_{\beta+\gamma+1}}\le C_{\gamma,d,q,p}\|f\|_{L^p_\beta}.
\end{align*}
Also step 1 gives that for $\beta\in\R$, $\|Kf\|_{L^1_{\beta+\gamma+1}}\le C_{\gamma,d,q}\|f\|_{L^1_\beta}$.
Pick $p_0>\max(\frac{d}{d-\gamma},\frac{d}{2})$, $p_1=1$.
Applying Riesz-Thorin interpolation theorem to $p_0$ and $p_1$, we obtain that for $\theta\in(0,1)$,
$K$ is a linear bounded operator from $L^{p_\theta}$ to $L^{q_\theta}$ with
\begin{align*}
  \|Kf\|_{L^{q_\theta}_{\beta+\gamma+1}}\le C_{\gamma,d,q,p,\theta}\|f\|_{L^{p_\theta}_\beta},
\end{align*}
where $\frac{1}{q_\theta} = \frac{\theta}{\infty} + \frac{1-\theta}{1}$,
  $\frac{1}{p_\theta} = \frac{\theta}{p_0} + \frac{1-\theta}{1}$.
For the last assertion, it suffices to notice that if $f\in L^{P_\theta}_\beta(H^l)$,
\begin{align*}
  \|Kf\|_{H^l} \le \int_\Rd|k(\xi,\xi_*)|\,\|f(\cdot,\xi_*)\|_{H^l}\,d\xi_*.
\end{align*}
This completes the theorem.
\qe\end{proof}

The following theorem is well studied in many literature such as \cite{Cercignani1994,Ukai} and I will put the proof in appendix.
\begin{Thm}\label{I_ThmL2}(Properties of $L$). Assume $\gamma\in[0,d)$.
Then $L : L^2(\Rd) \to L^2(\Rd)$ is a linear unbounded operator satisfying the following properties.

  (a). $L$ is a self-adjoint non-positive linear operator on $L^2(\Rd)$.

  (b). $Ker L = Span\{\M^{1/2}, \xi_1\M^{1/2},  \dots,  \xi_d\M^{1/2}, |\xi|^2\M^{1/2}\}$.
\end{Thm}

By orthogonal decomposition, we can decomposite $L^2(\Rd)$ by
    \begin{align*}
      L^2(\Rd) = \text{Ker}L \oplus (\text{Ker}L)^{\perp}.
    \end{align*}
  Denote $\varphi_0 = \M^{1/2}$, $\varphi_i=\xi_i\M^{1/2}$ ($i=1,\dots,d$), $\varphi_{d+1}=|\xi|^2\M^{1/2}$.
  Then we can define projection $P$ from $L^2(\Rd)$ onto $\text{Ker}L$ by
  \begin{align*}
    P f &:= \sum^{d+1}_{i=0}(f,\varphi_i)\varphi_i.
  \end{align*}

  \section{Estimate on the Linearized Boltzmann Operator}
  In this section, we will compute some basic estimate on operator $\B$, $\BB$, $\A$ as well as their resolvents. In order to make the subsequent arguments rigorous, we need to verify the existence of the resolvent in some specific space.

  Suppose $\gamma\in [0,d)$, $\beta\in\R$.
  Define
  \begin{align*}
    \widehat{B}(y) :&= -2\pi i y\cdot\xi + L,\\
      \BB :&= -2\pi i y\cdot\xi + L-P,\\
    \widehat{A}(y) :&= -2\pi i y\cdot\xi - \nu,
  \end{align*}
  with domain depending on $\beta$:
  \begin{align*}
    D_\beta := \{ f\in L^2_\beta(\Rd): y\cdot\xi f\in L^2_\beta(\Rd)\}.
  \end{align*}
  Also we define
  \begin{align*}
    K_0 = K-P.
  \end{align*}
  \begin{Rem}
  It's important to notice that the resolvent set and spectrum of these operators depend on the space, i.e $\beta$.
  \end{Rem}

  The following theorem gives some spectrum structures of operators $\B$ and $\BB$.
  \begin{Thm}\label{II_spectrum}Assume $\gamma\in[0,d)$, $\beta\in\R$, $y\in\Rd$.
  Then the following statements are valid.

    (1). $(\B,D_0)$ generates a contraction semigroup on $L^2(\R^d_\xi)$. Consequently,
    \begin{align}
      \rho(\B)\supset \{\text{Re}\lambda>0\}.\label{II_eq34}
    \end{align}

  (2). $(\B,D_\beta)$ generates a strongly continuous semigroup on $L^2_\beta$ with
  \begin{align}
    \|e^{t\B}\|_{L(L^2_\beta)}\le e^{t\|K\|_{L(L^2_\beta)}}.
  \end{align}

      (3). There are two cases about $\sigma_p(\B)\cap \{\Re\lambda=0\}$, where $\B$ is considered acting on $L^2$.
      \begin{align*}
        \sigma_p(\B)\cap \{\Re\lambda=0\}=\left\{
        \begin{aligned}
          &\emptyset, &&\text{  if  }y\neq 0,\\
          &\{0\}, &&\text{  if  }y= 0.
        \end{aligned}\right.
      \end{align*}

  (4). On $L^2$, for $y\in\Rd$,
        \begin{align*}
          \sigma_p(\B-P)\subset\{\text{Re}\lambda<0\}.
        \end{align*}
        
  \end{Thm}

  \begin{proof}

    1. Since $L$ is self-adjoint non-positive, we know $\B$ is dissipative on $D_0$. We claim that $(2\pi i y\cdot\xi+L,D_\beta)$ is the adjoint of $\B$ by using definition \ref{adjoint}. It suffices to show that $D(\B^*)=D_0$. 
    For any $u,v\in D_0$, we have 
    \begin{align*}
    	\<\B u,v\>_{L^2} = \<-2\pi i y\cdot\xi u + Lu, v\>_{L^2} = \<u,2\pi i y\cdot\xi v +Lv\>_{L^2},
    \end{align*}since $L$ is self-adjoint. Thus $u\mapsto \<\B u,v\>_{L^2}$ is continuous and hence $v\in D(\B^*)$. On the other hand, for any $u\in D_0$, $v\in D(\B^*)$, we have 
    \begin{align*}
    	\<\B u,v\>_{L^2} = \int_\Rd  u\,(\overline{2\pi i v\cdot\xi u+\nu v})\,d\xi + \int_\Rd Ku\,\overline{v}\,d\xi.
    \end{align*}
    Also, $K$ is linear continuous on $L^2$, thus $u\mapsto \int_\Rd  u\,(\overline{2\pi i v\cdot\xi v+\nu v})\,d\xi$ is linear continuous on $D_0$, hence has a unique bounded extension on $L^2$. By Riesz representation, $2\pi i v\cdot\xi v+\nu v\in L^2$ and so $v\in D_0$.
    In a conclusion, $D(\B^*)=D_0$.
    
    Now we have $(\B^*)^*=\B$, so $\B$ is closed. Also $\B^*$ is dissipative on $D_0$. Thus by theorem \ref{semigroup_dissipative}, $\B$ generates a contraction semigroup on $L^2$.

  2. For $\beta\in\R$, the semigroup generated by $\A$ on $L^2_\beta$ is defined by 
  \begin{align*}
    e^{t\A}u := e^{(-2\pi iy\cdot\xi-\nu)t}u,
  \end{align*}for any $u\in L^2_\beta$ and so $\|e^{t\A}\|_{L(L^2_\beta)}\le 1$. Indeed, $
	  \|(e^{t\A}u-u)/t - \A u\|_{L^2_\beta}\to 0,$ by Lebesgue dominated convergence theorem.

  Applying theorem \ref{semigroup_boundedpertur}, $\B=\A+K$ generates a semigroup on $L^2_\beta$ and
  \begin{align*}
    \|e^{t\B}\|_{L(L^2_\beta)}\le e^{t\|K\|_{L(L^2_\beta)}}.
  \end{align*}

  3. Consider $L^2$ to be the whole space.
  Let $\lambda\in\sigma_p(\B)$, then $\Re\lambda\le 0$ by \eqref{II_eq34}.
  Suppose $\Re\lambda=0$ and let $f\neq 0$, $f\in L^2$ be an eigenfunction of $\B$ corresponding to $\lambda$, then
  \begin{align*}
  	&\B f = -2\pi iy\cdot\xi f + Lf = \lambda f,\\
    &\Re(\B f,f)_{L^2} = (Lf,f)_{L^2} = 0.
  \end{align*}
  Then $f\in \text{Ker}L$, $Lf=0$. If $y=0$, then $\lambda =0$. If $y\neq 0$, then $2\pi iy\cdot\xi f =\Im\lambda f$, which contradicts to $f\neq 0$. Thus such $\lambda$ doesn't exist when $y \neq 0$.

  4. Let $\lambda\in\sigma_p(\B-P)$ and $f\neq0$ to be a eigenfunction of $\B-P$ corresponding to $\lambda$ in $L^2$, then
  \begin{align*}
  	&(\B-P) f = -2\pi iy\cdot\xi f + Lf -Pf = \lambda f,\\
    &\Re((\B-P) f,f)_{L^2} = (Lf,f)_{L^2}-\|Pf\|^2_{L^2} =\Re\lambda \|f\|_{L^2}.
  \end{align*}Since $L$ is non-positive, we have $\Re\lambda \le 0$. If $\Re\lambda=0$, then $(Lf,f)_{L^2} = \|Pf\|_{L^2} = 0$,
  and $f\in \text{Ker}L\cap (\text{Ker}L)^{\perp}$, and thus $f=0$. This is a contradiction and we complete the proof.
  \qe
  \end{proof}

  The following theorem gives the boundedness estimate on resolvent of $\A$.
  \begin{Thm}\label{II_main_estimate}
  Assume $\gamma\in[0,d)$, $\Re\lambda\ge 0$, $p\in(1,\infty)$. Fix $R>0$, $y\in\Rd$ and $y_1>0$. Then the followings are valid.

  (1). $\lambda\in\rho(\A)$ and
  $(\lambda I - \A)^{-1}: L^p_{\beta+\gamma}\to L^p_\beta$ is linear continuous with
  \begin{align}\label{II_eq314}
    \|(\lambda I - \A)^{-1}u\|_{L^p_{\beta+\gamma}\to L^p_\beta}
    &\le \frac{1}{\nu_0},\\
    \|(\lambda I - \A)^{-1}\chi_{|\xi|\ge R}u\|_{L^p_\beta}\label{II_eq315}
    &\le C_{\nu}\|u\|_{L^p_{\beta+\gamma}(|\xi|\ge R)}.
  \end{align}

  (2).
  If $|y|\le y_1$, $|\lambda|\ge 4\pi y_1 R$, then
  \begin{align}
    \|(\lambda I - \A)^{-1}\chi_{|\xi|\le R}u\|_{L^p_\beta}
    &\le C\|u\|_{L^p_\beta}\frac{1}{|\lambda|}.
  \end{align}
  Consequently,
  \begin{align}
    \|(\lambda I - \A)^{-1}u\|_{L^p_\beta}
    &\le C_{\nu}\|u\|_{L^p_{\beta+\gamma}(|\xi|\ge R)} + C\|u\|_{L^p_\beta}\frac{1}{|\lambda|},\\
    \|(\lambda I - \A)^{-1}Ku\|_{L^p_\beta}
    &\le C_{\nu}\|u\|_{L^p_{\beta-1}}(1+R)^{-1}
    +C\|u\|_{L^p_{\beta-\gamma-2}}\frac{1}{|\lambda|}.
  \end{align}

  (3).
  If $|y|\ge y_1$, then
  \begin{align}
    \|(\lambda I - \A)^{-1}\chi_{|\xi|\le R}u\|_{L^p_\beta}
    \le C_\nu\|u\|_{L^p_{\beta+\gamma}(|\xi|\le R,\,|\Im\lambda+2\pi y\cdot\xi|\le \frac{|y|}{\sqrt{y_1}})}
    +C_\nu \|u\|_{L^p_{\beta+\gamma}(|\xi|\le R)}\frac{1}{y_1^{\frac{p}{4}}}.
  \end{align}
  Consequently,
  \begin{align}
    \|(\lambda I - \A)^{-1}Ku\|_{L^2_\beta}
    &\le C_{\nu,d,q}\|u\|_{L^2_{\beta-2}(|\xi|\ge R)}
    +C_{\nu,d,q,\gamma}
    \|u\|_{L^{2}_{\beta-1}}
    \Big(\frac{R^{\frac{\delta(d-1)}{2(2+\delta)}}}{y_1^{\frac{\delta}{4(2+\delta)}}}
    +\frac{R^{\frac{d\delta}{2(2+\delta)}}}{y_1^{\frac{p}{4}}}\Big),
  \end{align}where $\delta = \frac{2}{p_0-2}\in(0,1)$ and $p_0:=\frac{d}{d-\gamma}+\frac{d}{2}$.

  \end{Thm}
  \begin{proof}

    1. For $\Re\lambda\ge 0$, $p\in(1,\infty)$, we have for $u\in L^p_{\beta+\gamma}$,
    \begin{align*}
      (\lambda I - \A)^{-1} &= \frac{1}{\lambda+\nu(\xi)+2\pi iy\cdot\xi},\\
      \|(\lambda I - \A)^{-1}u\|^p_{L^p_\beta}
      &\le \int_\Rd\frac{|(1+|\xi|)^{\beta}u(\xi)|^p}{|\nu(\xi)|^p}\,d\xi
      \le \frac{1}{\nu_0^p}\|u\|^p_{L^p_{\beta+\gamma}}.
    \end{align*}
    Thus $\lambda\in\rho(\A)$ and $(\lambda I-\A)^{-1}: L^2_{\beta+\gamma}\to L^2_\beta$ is linear continuous.
    For $R>0$, similarly, we have
    \begin{align*}
      \|(\lambda I - \A)^{-1}\chi_{|\xi|\ge R}u\|_{L^p_\beta}
      &\le C_{\nu}\|u\|_{L^p_{\beta+\gamma}(|\xi|\ge R)}.
    \end{align*}

  2.
  Fix $y_1>0$. If $|y|\le y_1$, $|\xi|\le R$, then $|2\pi y\cdot \xi|\le 2\pi y_1R$.
  Thus for $|\lambda|\ge 4\pi y_1R$,
  \begin{align*}
    \|(\lambda I - \A)^{-1}\chi_{|\xi|\le R}u\|_{L^p_\beta}
    &\le \|u\|_{L^p_\beta}\sup_{|\xi|\le R}\frac{1}{|\lambda+\nu(\xi)+2\pi iy\cdot\xi|}\\
    &\le \sqrt{2}\|u\|_{L^p_\beta}
    \sup_{|\xi|\le R}\frac{1}{\Re\lambda+|\Im\lambda|-2\pi iy_1R}\\
    &\le C\|u\|_{L^p_\beta}\frac{1}{|\lambda|}.
  \end{align*}
  Thus for $\alpha > \beta+\gamma$, using \eqref{II_eq315}, we have
  \begin{align*}
    \|(\lambda I - \A)^{-1}u\|_{L^p_\beta}
    &\le C_{\nu}\|u\|_{L^p_{\beta+\gamma}(|\xi|\ge R)}\notag
    +C\|u\|_{L^p_\beta}\frac{1}{|\lambda|}\\
    &\le C_{\nu}\|u\|_{L^p_{\alpha}}(1+R)^{\beta+\gamma-\alpha}
    +C\|u\|_{L^p_\beta}\frac{1}{|\lambda|}.
  \end{align*}
  Pick $\alpha = \beta+\gamma+1$, then
  \begin{align*}
    \|(\lambda I - \A)^{-1}Ku\|_{L^p_\beta}\notag
    &\le C_{\nu}\|Ku\|_{L^p_{\beta+\gamma+1}}(1+R)^{-1}
    +C\|Ku\|_{L^p_\beta}\frac{1}{|\lambda|}\\
    &\le C_{\nu}\|u\|_{L^p_{\beta-1}}(1+R)^{-1}
    +C\|u\|_{L^p_{\beta-\gamma-2}}\frac{1}{|\lambda|}.
  \end{align*}

  3. If $|y|\ge y_1$, we will use another method.
  For $R>0$,
  \begin{align*}
    \|(\lambda I - \A)^{-1}\chi_{|\xi|\le R}u\|^p_{L^p_\beta}
    &\le C_{\nu}\int_\Rd
    (1+|\xi|)^{p\beta+p\gamma}|u(\xi)|^p
    \Big(\frac{|\nu(\xi)|^2}{|\Re\lambda+\nu|^2+|\Im\lambda+2\pi y\cdot\xi|^2}\Big)^{p/2}\,d\xi.
  \end{align*}
  We divide this integral into two parts:
  \begin{align*}
    \int_{
    \substack{
    |\xi|\le R \\
    |\Im\lambda+2\pi y\cdot\xi|\le\frac{|y|}{\sqrt{y_1}}
    }
    }
    +
    \int_{
    \substack{
    |\xi|\le R \\
    |\Im\lambda+2\pi y\cdot\xi|>\frac{|y|}{\sqrt{y_1}}
    }
    }.
  \end{align*}
  Notice if $\xi\in \{\xi:|\Im\lambda+2\pi y\cdot\xi|\le\frac{|y|}{\sqrt{y_1}}\}$, we have $
    \frac{|\nu(\xi)|^2}{|\Re\lambda+\nu|^2+|\Im\lambda+2\pi y\cdot\xi|^2} \le 1$,
  and if $\xi\in \{\xi:|\Im\lambda+2\pi y\cdot\xi|>\frac{|y|}{\sqrt{y_1}}\}$, $|y|\ge y_1$, we have $
    \frac{|\nu(\xi)|^2}{|\Re\lambda+\nu|^2+|\Im\lambda+2\pi y\cdot\xi|^2}
    \le \frac{\nu_1\sqrt{y_1}}{|y|} \le \frac{C_\nu}{\sqrt{y_1}}$.
  Therefore if $|y|\ge y_1$,
  \begin{align}\label{II_temp2222}
    \|(\lambda I - \A)^{-1}\chi_{|\xi|\le R}u\|_{L^p_\beta}
    \le C_\nu\|u\|_{L^p_{\beta+\gamma}(|\xi|\le R,\,|\Im\lambda+2\pi y\cdot\xi|\le \frac{|y|}{\sqrt{y_1}})}
    +C_{\nu,p} \|u\|_{L^p_{\beta+\gamma}(|\xi|\le R)}\frac{1}{y_1^{\frac{p}{4}}}.
  \end{align}

  4. Suppose $|y|\ge y_1$. To prove the last assertion, we need to use the properties of $K$ from \ref{I_proerties_K}.
  Pick $p_0=\frac{d}{d-\gamma}+\frac{d}{2}$, $\theta=\frac{p_0}{2(p_0-1)}$
  and let $\delta=\frac{1}{1-\theta}-2>0$. Then $p_\theta = 2$, $q_\theta=2+\delta$ satisfy \eqref{I_eqpq} and
  \begin{align}
    \|Kf\|_{L^{2+\delta}_{\beta+\gamma+1}}\le C_{\gamma,d,q,p,\theta}\|f\|_{L^{2}_\beta},
  \end{align}

  Pick $p=2$ in \eqref{II_temp2222}, then using the fact $\left|\{|\xi|\le R, |\Im\lambda+2\pi y\cdot\xi|\le \frac{|y|}{\sqrt{y_1}}\}\right|\le \frac{C_dR^{d-1}}{\sqrt{y_1}}$ and H\"older's inequality, we have
  \begin{align*}
    &\|(\lambda I - \A)^{-1}\chi_{|\xi|\le R}u\|_{L^2_\beta}\\
    &\le C_\nu
    \|u\|_{L^{2+\delta}_{\beta+\gamma}}
    \big(\int_{\substack{|\xi|\le R\\ |\Im\lambda+2\pi y\cdot\xi|\le \frac{|y|}{\sqrt{y_1}}}}\,d\xi\big)^{\frac{\delta}{2(2+\delta)}}
    +C_\nu
    \|u\|_{L^{2+\delta}_{\beta+\gamma}(|\xi|\le R)}
    \big(\int_{|\xi|\le R}\,d\xi\big)^{\frac{\delta}{2(2+\delta)}}\frac{1}{y_1^{\frac{p}{4}}}\\
    &=  C_{\nu,d,\gamma}
    \|u\|_{L^{2+\delta}_{\beta+\gamma}}
    \Big(R^{\frac{\delta(d-1)}{2(2+\delta)}}y_1^{-\frac{\delta}{4(2+\delta)}}
    +R^{\frac{d\delta}{2(2+\delta)}}y_1^{-\frac{p}{4}}\Big).
  \end{align*}
  Thus
  \begin{align*}
    \|(\lambda I - \A)^{-1}\chi_{|\xi|\le R}Ku\|_{L^2_\beta}
    &\le C_{\nu,d,q,\gamma}
    \|u\|_{L^{2}_{\beta-1}}
    \Big(R^{\frac{\delta(d-1)}{2(2+\delta)}}y_1^{-\frac{\delta}{4(2+\delta)}}
    +R^{\frac{d\delta}{2(2+\delta)}}y_1^{-\frac{p}{4}}\Big).
  \end{align*}
  With \eqref{II_eq315}, we can get (3).
  \qe\end{proof}

  The following theorem gives the existence of inverse $(I-(\lambda I - \A)^{-1}K)^{-1}$ and $(I-(\lambda I - \A)^{-1}K_0)^{-1}$ on $L^2_\beta$.

  \begin{Thm}\label{II_inverse_1}Fixed $\beta\in\R$, we consider $L^2_\beta$ to be the whole space.

    (1). For $y\in\Rd$, $\Re\lambda\ge 0$, if $y\neq0$ or $\lambda\neq0$, then
    \begin{align}
      1 \in \rho((\lambda I - \A)^{-1}K).
    \end{align}

    (2). For $y\in\Rd$, $\Re\lambda\ge 0$, we have
    \begin{align}
      1 \in \rho((\lambda I - \A)^{-1}K_0).
    \end{align}
  \end{Thm}
  \begin{proof}

  1. If not, we suppose $1 \in \sigma((\lambda I - \A)^{-1}K)$.
  Since $(\lambda I-\A)^{-1}:L^2_{\beta+\gamma}\to L^2_\beta$ is linear continuous and $K:L^2_{\beta}\to L^2_{\beta+\gamma+1}$ is compact, we know $(\lambda I - \A)^{-1}K:L^2_{\beta+\gamma}\to L^2_{\beta+\gamma}$ is compact, for $\beta\in\R$.
  Thus by Fredholm alternative, on $L^2_{\beta+\gamma}$, we have
  \begin{align*}
    1\in \sigma_p((\lambda I - \A)^{-1}K).
  \end{align*}
  Thus for some $0\neq u\in L^2_{\beta+\gamma}$,
  \begin{align}\label{II_e1}
    u=(\lambda I - \A)^{-1}Ku,
  \end{align}
  and hence $u\in L^2_{\beta+\gamma+2}$. By using \eqref{II_e1} inductively, we have $u\in \cap_{\beta\in\R}L^2_{\beta}\subset D_0$ and
  \begin{align}
    (\lambda I-\B)u=0.\label{II_e2}
  \end{align}
  Thus $\lambda$ is an eigenvalue of $\B$ and $\Re\lambda = 0$ by theorem \ref{II_spectrum} (1).
  But $(y,\lambda)\neq 0$, so equation \eqref{II_e2} contradicts to theorem \ref{II_spectrum} (3).

  2. The proof of the second assertion is similar, but using the theorem \ref{II_spectrum} (4).
  \qe\end{proof}

  The following lemma is used for proving the uniformly boundedness of $(\lambda I-\B)^{-1}$ and $(\lambda I-\BB)^{-1}$.
  Here we consider variable $(\lambda,y)\in \overline{\C_+}\times\Rd$.
  \begin{Lem}\label{II_continuous}Assume $\gamma\in[0,d)$.

    (1). For $\Re\ge 0$, $\alpha > \beta+\gamma$, we have
    \begin{align}
      (\lambda I - \A)^{-1}&\in C\left(\overline{\C_+}\times\Rd; L(L^2_{\alpha},L^2_\beta)\right),\notag\\\label{II_ec1}
      (\lambda I - \A)^{-1}K&\in C\left(\overline{\C_+}\times\Rd; L(L^2_{\alpha},L^2_{\beta+\gamma+2})\right)
    \end{align}

    (2). For $r>0$,
    \begin{align*}
      (I-(\lambda I - \A)^{-1}K)^{-1}\in BC(\overline{\C_+}\times  (\Rd\setminus B_r); L(L^2_\beta))\cap BC((\overline{\C_+}\setminus B_r)\times \Rd; L(L^2_\beta)),
    \end{align*}where $B_r\in \Rd$ is the closed ball in $\Rd$ with center $0$ and radius $r$.

    (3).
    \begin{align*}
      (I-(\lambda I - \A)^{-1}K_0)^{-1}\in BC\left(\overline{\C_+}\times\Rd; L(L^2_\beta)\right).
    \end{align*}
  \end{Lem}

  \begin{proof}
  1. Let $\alpha>\beta+\gamma$.
  For $(\lambda_1,y_1),(\lambda_2,y_2)\in \overline{\C_+}\times\Rd$, we have
  \begin{align*}
    (\lambda_1 I - \widehat{A}(y_1))^{-1}-(\lambda_2 I - \widehat{A}(y_2))^{-1}
    &=\frac{1}{\lambda_1+\nu(\xi)+2\pi iy_1\cdot\xi} - \frac{1}{\lambda_2+\nu(\xi)+2\pi iy_2\cdot\xi}
  \end{align*}
  On one hand, by using \eqref{II_eq314} in \ref{II_main_estimate}, we have
  \begin{align}
    \left\|\big[(\lambda_1 I - \widehat{A}(y_1))^{-1}-(\lambda_2 I - \widehat{A}(y_2))^{-1}\big]\chi_{|\xi|\ge R}\right\|_{L(L^2_{\beta+\gamma},L^2_\beta)}
    &\le C_{\nu}.\label{II_e3}
  \end{align}
  On the other hand,
  \begin{align*}
    \left|\frac{1}{\lambda_1+\nu(\xi)+2\pi iy_1\cdot\xi} - \frac{1}{\lambda_2+\nu(\xi)+2\pi iy_2\cdot\xi}\right|
    &\le C_{\nu}(1+|\xi|)^{2\gamma}(|\lambda_2-\lambda_1|+2\pi |y_2-y_1|\,|\xi|).
  \end{align*}
  Thus
  \begin{align}
    \left\|\big[(\lambda_1 I - \widehat{A}(y_1))^{-1}-(\lambda_2 I - \widehat{A}(y_2))^{-1}\big]\chi_{|\xi|\le R}\right\|_{L(L^2_{\beta},L^2_\beta)}
    \le C_{\nu}
    (1+R)^{2\gamma}(|\lambda_2-\lambda_1|+R|y_2-y_1|).\label{II_e4}
  \end{align}

  Combining equation \eqref{II_e3} and \eqref{II_e4}, we have for $\alpha > \beta+\gamma$,
  \begin{align*}
    \|(\lambda_1 I - &\widehat{A}(y_1))^{-1}-(\lambda_2 I - \widehat{A}(y_2))^{-1}\|_{L(L^2_\alpha,L^2_\beta)}\\
    &\le C_{\nu}(1+R)^{\beta+\gamma-\alpha}
    +C_{\nu} (1+R)^{2\gamma}(|\lambda_2-\lambda_1|+R|y_2-y_1|).
  \end{align*}
  Thus
  \begin{align*}
    \|(\lambda_1 I - \widehat{A}(y_1))^{-1}-(\lambda_2 I - \widehat{A}(y_2))^{-1}\|_{L(L^2_\alpha,L^2_\beta)}
    &\to 0,
  \end{align*}
  as $|(\lambda_1,y_1) - (\lambda_2,y_2)|\to 0$ and $(\lambda_1,y_1),(\lambda_2,y_2)\in \overline{\C_+}\times\Rd$.
  Therefore,
  \begin{align}
    (\lambda I - \A)^{-1}&\in C\left(\overline{\C_+}\times\Rd; L(L^2_\alpha,L^2_\beta)\right),\notag\\
    (\lambda I - \A)^{-1}K&\in C\left(\overline{\C_+}\times\Rd; L(L^2_\alpha,L^2_{\beta+\gamma+2})\right).\label{II_174}
  \end{align}

  2. Pick $\alpha=\beta+\gamma+2$ in \eqref{II_174},
  then we have, for any $\beta\in\R$, \begin{align}\label{II_temp123}
    (\lambda I - \A)^{-1}K\in C(\overline{\C_+}\times\Rd; L(L^2_\beta)).
  \end{align}
  Also  theorem \ref{II_inverse_1} shows that
  $(I-(\lambda I - \A)^{-1}K)^{-1}$ exists on $L^2_\beta$ for $(\lambda, y)\neq (0,0)$.
  Firstly we prove the resolvent $(\eta I-T)^{-1}$ is continuously depending on $T$.
  For any bounded linear operators $T_1$, $T_2$ defined on the same Banach space, if $\eta\in \rho(T_1)\cap\rho(T_2)$, then whenever $\|T_1-T_2\|\le \frac{1}{2\|(\eta I -T_1)^{-1}\|}$,
  \begin{align*}
    (\eta I - T_1)^{-1} - (\eta I - T_2)^{-1} &= (\eta I - T_1)^{-1} (T_1-T_2) (\eta I - T_2)^{-1}\\
    &= \sum^\infty_{n=1}\big((\eta I -T_2)^{-1}(T_1-T_2)\big)^n(\eta I -T_2)^{-1},\\
    \|(\eta I - T_1)^{-1} - (\eta I - T_2)^{-1}\| &\le \frac{1}{2}\|T_1-T_2\|\|(\eta I -T_2)^{-1}\|^2.
  \end{align*}
  And so the resolvent $(\eta I-T)^{-1}$ of $T$ is continuous with respect to $T$.
  Now $(I-(\lambda I - \A)^{-1}K)^{-1}$ exists on $L^2_\beta$ for $(\lambda, y)\neq (0,0)$, thus using \eqref{II_temp123}, we have
  \begin{align*}
    (I-(\lambda I - \A)^{-1}K)^{-1}\in C(\overline{\C_+}\times &(\Rd\setminus\{0\}); L(L^2_\beta))\cap C((\overline{\C_+}\setminus\{0\})\times \Rd; L(L^2_\beta)).
  \end{align*}
  Similarly, $(I-(\lambda I - \A)^{-1}K_0)^{-1}$ exists on $L^2_\beta$ for $(\lambda, y)\in\overline{\C_+}\times\Rd$ and
  \begin{align*}
    (I-(\lambda I - \A)^{-1}K_0)^{-1}\in C(\overline{\C_+}\times\Rd; L(L^2_\beta,L^2_\beta)).
  \end{align*}
  These prove the continuity.

  3. Fix $y_1>0$, $R>0$, we shall use theorem \ref{II_main_estimate} to prove the boundedness.
   If $|y|\le y_1$, $\Re\lambda>0$ with $|\lambda|\ge 4\pi y_1 R$,
   then
  \begin{align}\label{II_eb1}
    \|(\lambda I - \A)^{-1}K\|_{L^2_\beta\to L^2_\beta}\le C_{\nu,d}[(1+R)^{-1}+\frac{1}{|\lambda|}].
  \end{align}
  If $|y|\ge y_1$,
  \begin{align}\label{II_eb2}
    \|(\lambda I - \A)^{-1}K\|_{L^2_\beta\to L^2_\beta}
    &\le C_{\nu,d,q,\gamma}\Big((1+R)^{-1}
    +
    R^{\frac{\delta(d-1)}{2(2+\delta)}}y_1^{-\frac{\delta}{4(2+\delta)}}
    +R^{\frac{d\delta}{2(2+\delta)}}y_1^{-\frac{p}{4}}\Big).
  \end{align}
  Write $C_0 = \max(C_{\nu,d},C_{\nu,d,q,\gamma})$.
  Firstly we pick a sufficiently large $R_1=R_1(\nu,d,q,\gamma)$ s.t.
  \begin{align*}
    C_{0}(1+R_1)^{-1}<1/4.
  \end{align*}
  Then with this $R_1$, we pick $y_1 = y_1(R_1)$ so large that
  \begin{align*}
    C_{0}\Big(R^{\frac{\delta(d-1)}{2(2+\delta)}}y_1^{-\frac{\delta}{4(2+\delta)}}
    +R^{\frac{d\delta}{2(2+\delta)}}y_1^{-\frac{p}{4}}\Big)<1/4.
  \end{align*}
  Let $r_1 = \max(4\pi y_1 R_1, 1+R_1)$ and
  \begin{align*}
    B_0(r,r_1,y_1)&:=\{(\lambda,y) : \Re\lambda\ge 0, |\lambda|\le r_1, |y|\le y_1, |y|\ge r\},\\
    B_1(r,r_1,y_1)&:=\{(\lambda,y) : \Re\lambda\ge 0, |y|\ge r\}\setminus B_0(r,r_1,y_1).
  \end{align*}
  Since $(I-(\lambda I - \A)^{-1}K)^{-1}\in C(\overline{\C_+}\times\Rd\setminus\{0\}; L(L^2_\beta,L^2_\beta))$,
  and $B_0(r,r_1,y_1)$ is a compact set, we have
  \begin{align}\label{II_e2222}
    \sup_{(\lambda,y)\in B_0(r,r_1,y_1)}\|(I-(\lambda I - \A)^{-1}K)^{-1}\|_{L(L^2_\beta)}<\infty.
  \end{align}
  For $(\lambda,y)\in B_1(r,r_1,y_1)$, we have $|\lambda|> r_1$ or $|y|> y_1$.
  If $|y|> y_1$, then by \eqref{II_eb2}, we have
  \begin{align*}
    \|(\lambda I - \A)^{-1}K\|_{L(L^2_{\beta})} \le 1/2.
  \end{align*}
  If $|y|\le y_1$, then $|\lambda|> r_1$ and then by \eqref{II_eb1},
  \begin{align*}
    \|(\lambda I - \A)^{-1}K\|_{L(L^2_{\beta})} \le 1/2.
  \end{align*}
  Thus for $(\lambda,y)\in B_1(r,r_1,y_1)$, we have $(I-(\lambda I - \A)^{-1}K)^{-1}$ exists on $L^2_\beta$ and
  \begin{align}\label{II_e1111}
    \|(I-(\lambda I - \A)^{-1}K)^{-1}\|_{L(L^2_{\beta})} \le 2.
  \end{align}
  Combining \eqref{II_e2222} and \eqref{II_e1111}, we have for $r>0$.
  \begin{align*}
    (I-(\lambda I - \A)^{-1}K)^{-1}\in BC(\overline{\C_+}\times (\Rd\setminus B_r); L(L^2_\beta)).
  \end{align*}

  4. By digging out a ball near $\lambda=0$ instead of $y=0$, we can get the second boundedness for $(I-(\lambda I - \A)^{-1}K)^{-1}$.

  5. The proof of the last assertion is similar to step 3, but in this case we don't need to dig out the ball $B_r$, since $\B-P$ has no zero eigenvalue at $y=0$.
  \qe\end{proof}

  Noticing
  \begin{align*}
    (\lambda I -\B)^{-1} &= (I-(\lambda I-\A)^{-1}K)^{-1}(\lambda I-\A)^{-1},\\
    \lambda I -\BB &= (\lambda I-\A)(I-(\lambda I-\A)^{-1}K_0).
  \end{align*}we have the following corollary.
  \begin{Coro}\label{II_coro}
    For any $r>0$, we have
    \begin{align*}
      (\lambda I -\B)^{-1}\in BC(\overline{\C_+}\times & (\Rd\setminus B_r); L(L^2_{\beta+\gamma},L^2_\beta)\cap BC((\overline{\C_+}\setminus B_r)\times \Rd; L(L^2_{\beta+\gamma},L^2_\beta))
    \end{align*}
    Consequently,
    \begin{align*}
      \Big(\sup_{\Re\lambda\ge 0, y\in\Rd, |y|\ge r}+
      \sup_{\Re\lambda\ge 0, |\lambda|\ge r, y\in\Rd}\Big)\|(\lambda I -\B)^{-1}\|_{L(L^2_{\beta+\gamma},L^2_\beta)}<\infty.
    \end{align*}
    Also,
    \begin{align}\label{III_eq12}
    \sup_{\Re\lambda\ge 0, y\in\Rd}\|(\lambda I -\BB)^{-1}\|_{L^2_{\beta+\gamma}\to L^2_\beta}<\infty.
    \end{align}
  \end{Coro}

  Furthermore, we need the following invertibilities.
  \begin{Lem}\label{II_inverse_2}
    Let $\Re\lambda\ge 0$, $y\in\Rd\setminus\{0\}$, $\beta\in\R$.  \\
    (1). The inverse $(I-(\lambda I-\BB)^{-1}P)^{-1}$ exists on $L^2_\beta$.\\ (2). The inverse $(I-P(\lambda I-\BB)^{-1}P)^{-1}$ exists on Ker$L\subset L^2_\beta$.
  \end{Lem}
  \begin{proof}
    1. If not, we suppose $1 \in \sigma((\lambda I-\BB)^{-1}P)$.
    Since $(\lambda I-\BB)^{-1}:L^2_{\beta+\gamma}\to L^2_\beta$ is linear continuous and $P:L^2_{\beta}\to L^2_{\beta+\gamma}$ is compact,
    we have $(\lambda I-\BB)^{-1}P:L^2_{\beta+\gamma}\to L^2_{\beta+\gamma}$ is compact, for $\beta\in\R$.
    Thus by Fredholm alternative, on $L^2_{\beta+\gamma}$, we have $1\in \sigma_p((\lambda I-\BB)^{-1}P)$.
    Thus for some $0\neq u\in L^2_{\beta+\gamma}$,
    \begin{align*}
      u=(\lambda I-\BB)^{-1}Pu,
    \end{align*}
    and hence $u\in \cap_{\beta\in\R}L^2_{\beta}$ with
    \begin{align}
    (\lambda I-\B)u &= 0.\label{II_eeee2}
    \end{align}
    Thus $\lambda$ is an eigenvalue of $\B$ and $\Re\lambda = 0$ by theorem \ref{II_spectrum} (1).
    But $y\neq 0$, so equation \eqref{II_eeee2} contradicts to theorem \ref{II_spectrum} (3).

  2. The proof existence of $(I-P(\lambda I-\BB)^{-1}P)^{-1}$ is similar to step 1.
  \qe\end{proof}

  \section{Eigenvalue Structure near $y=0$}
  In this section, we will give the proof of the existence of eigenvalues to operator $P(\lambda I -\BB)^{-1}P$ as well as the asymptotic behavior of the singular points of $(I-P(\lambda I -\BB)^{-1}P)^{-1}$ as $y\to 0$. These theorems are necessary for the estimate on semigroup $e^{t\B}$ due to inversion formula of semigroup and \eqref{III_eq44} below.

  Using resolvent identities:
  \begin{align*}
    (\lambda I-\B)^{-1} &= (\lambda I -\BB)^{-1}(I-P(\lambda I - \BB)^{-1})^{-1}\\
    &= (\lambda I -\BB)^{-1} - (\lambda I-\B)^{-1} P(\lambda I - \BB)^{-1},
  \end{align*}
  we have
  \begin{align}
    (\lambda I-\B)^{-1} = (\lambda I -\BB)^{-1} - (\lambda I -\BB)^{-1}(I-P(\lambda I - \BB)^{-1})^{-1} P(\lambda I - \BB)^{-1}.
  \end{align}
  Then applying lemma \ref{III_lem11} below with $A=(\lambda I - \BB)^{-1}$, we have
  \begin{align}\label{III_eq44}
  (\lambda I &- \B)^{-1}\notag\\ &= (\lambda I - \BB)^{-1}
  + (\lambda I - \BB)^{-1}P(I-P(\lambda I - \BB)^{-1}P)^{-1}P(\lambda I - \BB)^{-1}.
  \end{align}

  Here $\|P\|_{L^2_\beta\to L^2_{\beta+\gamma}}<\infty$. Thus
  \begin{align*}
  \sup_{\Re\lambda\ge 0, y\in\Rd}\|(\lambda I - \BB)^{-1}P\|_{L(L^2_{\beta},L^2_\beta)}<\infty,\\
  \sup_{\Re\lambda\ge 0, y\in\Rd}\|P(\lambda I - \BB)^{-1}\|_{L(L^2_{\beta},L^2_\beta)}<\infty.
  \end{align*}
  Then $(\lambda I - \BB)^{-1}P$ and $P(\lambda I - \BB)^{-1}$ are bounded on $L^2_\beta$,
  so the singularity of resolvent $(\lambda I - \B)^{-1}$ near $y=0$ comes from $(I-P(\lambda I - \BB)^{-1}P)^{-1}:\Ker L\to \Ker L$. So in the following subsections we will study the behavior of this operator.

  \begin{Lem}\label{III_lem11} Let $A$ be a linear continuous operator from $L^2_{\beta+\gamma}$ to $L^2_\beta$, for any $\beta\in\R$. If
    the inverse in the following statement exists, then they are valid.

    (1). For $f\in L^2_\beta$, we have $(I-PA)^{-1}Pf \in \Ker L$.
    Consequently,
    \begin{align}\label{III_lemP1}
      (I-PA)^{-1}Pf = P(I-PA)^{-1}Pf.
    \end{align}

    (2). On $L^2_\beta$, for any $\beta\in\R$, we have
    \begin{align}
      (I-PA)^{-1} P = (I-PAP)^{-1} P.
    \end{align}
  \end{Lem}
  \begin{proof}
    1. For $f\in L^2_\beta$, then $Pf\in \cap_{\beta\in\R}L^2_\beta$.
    Let $g = (I-PA)^{-1}Pf$.
    Then
    \begin{align*}
      &(I-PA)g = Pf,\\
      &g = PAg + Pf \in\Ker L.
    \end{align*}

    2. Let $f\in L^2_\beta$, then by \eqref{III_lemP1},
    \begin{align*}
      Pf = (I&-PA)(I-PA)^{-1}Pf = (I-PAP)(I-PA)^{-1}Pf,\\
      &(I- PAP)^{-1}Pf = (I-PA)^{-1}Pf.
    \end{align*}
  \qe\end{proof}

  \begin{Lem}\label{III_lem12} For $\Re\lambda\ge 0$, $\beta\in\R$, for $f\in L^2_\beta$,
    \begin{align}
      P(\lambda I - L+P)^{-1}f = (\lambda I - L+P)^{-1}Pf = \frac{Pf}{\lambda+1}.
    \end{align}
  \end{Lem}
  \begin{proof}
    Let $f\in L^2_\beta$ and $g:= (\lambda I - L+P)^{-1}f$, then
    $(\lambda I - L+P)g = f$, and
    $Pg = \frac{Pf}{\lambda+1}$.
    For the second equality, it suffices to show that $(\lambda I - L+P)^{-1}Pf\in\Ker L$.
    Let $h:= (\lambda I - L+P)^{-1}Pf$, then $(\lambda I - L+P)h = Pf$. Taking inner product with $P^\perp h$, we have
    \begin{align*}
      \lambda\|P^\perp h\|^2_{L^2} + (-Lh,P^\perp h) = 0
    \end{align*}
    But $L$ is a non-positive operator, thus $(-Lh, h) = 0 $ and $h\in\Ker L$.
  \qe\end{proof}

  \subsection{Eigenvalue Structure of $D$}
  Formally by the second resolvent identity, on Ker$L$, we have
  \begin{align*}
  (\lambda I - \BB)^{-1}
  &= (\lambda I - L+P)^{-1}+(\lambda I - \BB)^{-1}(-2\pi iy\cdot\xi)(\lambda I - L+P)^{-1}.
  \end{align*}
  It is valid only on Ker$L$, so here we check this identity carefully.
  Indeed for $f\in\Ker L$, we have
  \begin{align*}
  (\lambda I -\BB) f &= (\lambda I +2\pi i y\cdot\xi-L+P)f
  = (\lambda + 1 +2\pi iy\cdot\xi)f,\\
  f &= (\lambda I -\BB)^{-1}(\lambda + 1 +2\pi iy\cdot\xi)f,\\
  (\lambda I -\BB)^{-1} f &= \frac{1}{\lambda+1}\big(I-(\lambda I -\BB)^{-1}(2\pi iy\cdot\xi) \big)f,
  \end{align*} and then by using lemma \ref{III_lem12}, we complete the checking.
  
  Write $y=r \omega$, with $ \omega\in S^{d-1}$, $r=|y|$ and write $\lambda = \sigma+i\tau$.
  Define
  \begin{align}
  D(\sigma,\tau,r,\omega)=P((\sigma+i\tau +2\pi i r\omega\cdot\xi) I -L +P)^{-1}(\omega\cdot\xi)P.
  \end{align}
  Then on $L^2_\beta$, we have
  \begin{align}
  P(\lambda I - \BB)^{-1}P
  &= \frac{1}{\lambda+1}\big(P-2\pi i r D(\sigma,\tau,r,\omega)\big).\label{III_eq28}
  \end{align}
  Here we can assume $r\in\R$ instead of $r>0$.
  \begin{Rem}
  	When considering operator $D(\sigma,\tau,r,\omega)$, we can assume $r\in\R$, but when we
  	go back to $(I-P(\lambda I - \BB)^{-1}P)^{-1}$, we should assume $y\neq 0$.
  \end{Rem}

  Define an orthonormal basis $\{\psi_j\}^{d+1}_{j=0}$ of $\Ker L$ in $L^2(\Rd)$ as following,
  \begin{equation}
    \left\{
    \begin{aligned}
      \psi_0 &= \M^{1/2},\\
      \psi_j &= \xi_j\M^{1/2},\text{  if  }j=1,\dots,d\\
      \psi_{d+1} &= \frac{1}{\sqrt{2}d}(|\xi|^2-d)\M^{1/2}.
    \end{aligned}
    \right.
  \end{equation}

  Fix $\omega\in\S^{d-1}$. Define rotation $R\in O(d)$ on $\Rd$ s.t. $R\omega = e_1$, where $e_1=(1,0,\dots,0)$.
  Now we investigate the eigenvalues of
  \begin{align*}
    D(\sigma,\tau,r,\omega)=P((\sigma+i\tau +2\pi i r\omega\cdot\xi) I -L +P)^{-1}(\omega\cdot\xi)P,
  \end{align*}where $\sigma\ge 0$, $\tau\in\R$, $r\in\R$, $\omega\in\S^{d-1}$.
  Notice $D$ maps Ker$L$ into Ker$L$, so under the orthonormal basis $\{R^T\psi_j\}$,
  we can obtain its matrix representation.
  \begin{Def}For $j,k=0,\dots,d+1$, define
    \begin{align*}
      D_{jk}(\sigma,\tau,r):=(D(\sigma,\tau,r,e_1)\psi_j,\psi_k)_{L^2}.
    \end{align*}
  \end{Def}

  \begin{Lem}For $\sigma\ge 0$, $\tau\in\R$, $r\in\R$, $\omega\in\S^{d-1}$, $j,k=0,\dots,d+1$, we have
    \begin{align}\label{III_lem151}
      (D(\sigma,\tau,r,\omega)R^T\psi_j,R^T\psi_k)_{L^2} = (D(\sigma,\tau,r,e_1)\psi_j,\psi_k)_{L^2},
    \end{align}
    and
    \begin{align}\label{III_lem152}
      (D(\sigma,\tau,r,e_1)\psi_j,\psi_k)_{L^2} \in C^\infty(\overline{\R_+}\times\R\times\R;\C).
    \end{align}
  \end{Lem}
  With this lemma, we have
  \begin{align*}
    D(\sigma,\tau,r,\omega)(R^T\psi_0,\dots,R^T\psi_{d+1})
    = (R^T\psi_0,\dots,R^T\psi_{d+1})(D_{kj})_{j,k=0}^{d+1},
  \end{align*}
  and $(D_{kj})_{j,k=0}^{d+1}$ is the matrix representation of $D(\sigma,\tau,r,\omega)$ under basis $\{R^T\psi_j\}$.
  \begin{proof}
  1. For any rotation $R\in O(d)$ acting on velocity variable, we know that $R$ commutes with $P$, $I$, $L$, thus for $\sigma\ge 0$,
  \begin{align*}
    RD(\sigma,\tau,r,\omega) &= RP((\sigma+i\tau +2\pi i r\omega\cdot\xi) I -L +P)^{-1}(\omega\cdot\xi)P\\
    &= P((\sigma+i\tau +2\pi i r\omega\cdot R\xi) I -L +P)^{-1}(\omega\cdot R\xi)PR\\
    &= D(\sigma,\tau,r,R^T\omega)R.
  \end{align*}
  Then \eqref{III_lem151} follows from $R^T\omega = e_1$.

  2. Recall thoerem \ref{II_coro} that $((\sigma+i\tau +2\pi i r\xi_1) I -L +P)^{-1}$ is a linear bounded operator form $L^2_{\beta+\gamma}$ to $L^2_\beta$ and
  is continuous with respect to $\sigma\ge0$, $\tau\in\R$, and $r\in\R$.
  By the second resolvent identity, for any $\sigma_1,\sigma_2\ge 0$,
  \begin{align*}
    &D_{jk}(\sigma_1,\tau,r) - D_{jk}(\sigma_2,\tau,r)\\
    &= \big(\big[((\sigma_1+i\tau +2\pi i r\xi_1) I -L +P)^{-1}-((\sigma_2+i\tau +2\pi i r\xi_1) I -L +P)^{-1}\big]\xi_1\psi_j,\psi_k\big)\\
    &= \big(((\sigma_2+i\tau +2\pi i r\xi_1) I -L +P)^{-1}(\sigma_2-\sigma_1)((\sigma_1+i\tau +2\pi i r\xi_1) I -L +P)^{-1}\xi_1\psi_j,\psi_k\big),
  \end{align*}and so whenever $\sigma> 0$,
  \begin{align*}
    \partial_\sigma D_{jk}(\sigma,\tau,r) = -(((\sigma+i\tau +2\pi i r\xi_1) I -L +P)^{-2}\xi_1\psi_j,\psi_k)_{L^2}.
  \end{align*}
  Inductively,
  \begin{align*}
    \partial^n_\sigma D_{jk}(\sigma,\tau,r) = (-1)^nn!(((\sigma+i\tau +2\pi i r\xi_1) I -L +P)^{-n-1}\xi_1\psi_j,\psi_k)_{L^2},
  \end{align*}
  Similarly,
  \begin{align}\label{III_lem14145}
    \partial^n_\tau D_{jk}(\sigma,\tau,r) = (-i)^nn!(((\sigma+i\tau +2\pi i r\xi_1) I -L +P)^{-n-1}\xi_1\psi_j,\psi_k)_{L^2}.
  \end{align}
  For the derivative with resect to $r$, we need to be more careful. For $r_1,r_2\in\R$,
  \begin{align*}
    &D_{jk}(\sigma,\tau,r_1) - D_{jk}(\sigma,\tau,r_2)\\
    &= \big(((\sigma+i\tau +2\pi i r_2\xi_1) I -L +P)^{-1}2\pi i(r_2-r_1)\xi_1((\sigma+i\tau +2\pi i r_1\xi_1) I -L +P)^{-1}\xi_1\psi_j,\psi_k\big).
  \end{align*}
  Use the uniformly boundedness of $((\sigma+i\tau +2\pi i r_2\xi_1) I -L +P)^{-1}$ from $L^2_{\beta+\gamma}$ to $L^2_\beta$ and notice $\psi_j\in\cap_{\beta\in\R}L^2_\beta$,
   we have
  \begin{align}\label{III_lem14146}
    \partial_r D_{jk}(\sigma,\tau,r) = -2\pi i\big((((\sigma+i\tau +2\pi i r\xi_1) I -L +P)^{-1}\xi_1)^2\psi_j,\psi_k\big).
  \end{align}
  Inductively,
  \begin{align*}
    \partial^n_r D_{jk}(\sigma,\tau,r) = (-2\pi i)^nn!
    \big((((\sigma+i\tau +2\pi i r\xi_1) I -L +P)^{-1}\xi_1)^{n+1}\psi_j,\psi_k\big).
  \end{align*}

  All these derivatives are right-continuous at $\sigma=0$ and so
  $D_{jk}(\sigma,\tau,r) \in C^\infty(\overline{\R_+}\times\R\times\R)$.
  \qe\end{proof}

  Here we need a $C^\infty$ extention theorem from \cite{Seeley1964}.
  \begin{Thm}(Seeley).
      Suppose $f(x,\sigma)\in C^\infty(\Rd\times\{\sigma\ge 0\})$. Let $\phi\in C^\infty(\Rd)$ such that
      $\phi = 1 $ on $0\le |t|\le 1$ and $0$ if $|t|\ge 2$.
    There exists $\{a_k\}$, $\{b_k\}$ such that (i). $b_k<0$; (ii). $\sum|a_k||b_k|^n<\infty$, for $n=0,1,\dots$;
    (iii). $\sum a_k(b_k)^n=1$, for $n=0,1,\dots$; (iv). $b_k\to -\infty$. Define for $\sigma<0$,
      \begin{align*}
        f(x,\sigma) := \sum^\infty_{k=0} a_k \phi(b_k\cdot\sigma)f(x,b_k\sigma).
      \end{align*} Then $f(x,\sigma)\in C^\infty(\Rd\times\R)$.
  \end{Thm}
  Applying this $C^\infty$ extention theorem, we can extend $D_{jk}$ to all $\sigma\in\R$ such that
  \begin{align*}
  D_{jk}(\sigma,\tau,r) \in C^\infty(\R\times\R\times\R),
  \end{align*}and for $\sigma<0$,
  \begin{align*}
    D_{jk}(\sigma,\tau,r)=\sum^\infty_{k=0} a_k \phi(b_k\cdot\sigma) D_{jk}(b_k\sigma,\tau,r).
  \end{align*}

  \subsubsection{The Eigenvalue Equation.}

  For $\sigma\ge 0$,
  \begin{align}
    D(\sigma,\tau,r,e_1)=P((\sigma+i\tau +2\pi i r\xi_1) I -L +P)^{-1}(\xi_1)P.
  \end{align}
  Thus for $j=2,\dots,d$, the reflection $r_j:\xi\to (\xi_1,\dots,-\xi_j,\dots,\xi_d)$ commutes with $D(\sigma,\tau,r,e_1)$.
  Also for $j=2,\dots,d$, $\psi_j$ is odd with respect to $\xi_j$, and for $k\neq j$, $\psi_k$ is even with respect to $\xi_j$.
    Thus
    \begin{align*}
      D_{jk}(\sigma,\tau,r) = 0,\qquad  \text{ if } &2\le j\le d, 0\le k\le d+1, k\neq j\\
      \text{ or if } &2\le k\le d, 0\le j\le d+1, k\neq j.
    \end{align*}

  So the eigenvalues equation of operator $D(\sigma,\tau,r,\omega)$ under basis $\{R^T\psi_j\}^{d+1}_{j=0}$ is
  \begin{align*}
    \eta I_{(d+2)\times(d+2)} = (D_{jk})_{j,k=0}^{d+1}.
  \end{align*}
  That is
  \begin{equation}\label{III_eigenequation}\eta I_{(d+2)\times(d+2)} -
    \begin{pmatrix}
      D_{00} & D_{01} & 0 & \dots & 0  & D_{0,d+1} \\
      D_{10} & D_{11} & 0 & \dots & 0  & D_{1,d+1} \\
      0 & 0 & D_{22} & \dots & 0  & 0 \\
      \vdots & \vdots & \vdots & \ddots & \vdots \\
      0 & 0 & 0 & \dots & D_{dd} & 0 \\
      D_{d+1,0} & D_{d+1,1} & 0 & \dots & 0  & D_{d+1,d+1}
  \end{pmatrix}=0,
  \end{equation}
  where the matrix $(D_{jk})_{j,k=0,\dots,d+1}$ is smooth in $(\sigma,\tau,r)\in\R^3$.

  \subsubsection{The Eigenvalues of $D$.}

  Firstly, we can easily get $(d-1)$ eigenvalues. That is, for $j=2,\dots,d$,
  \begin{align}
    \eta_j(\sigma,\tau,r) &= D_{jj}(\sigma,\tau,r)\in C^\infty(\R\times\R\times\R).
  \end{align}
  Also, one can pick the eigenvector corresponding to $\eta_j(\sigma,\tau,r)$ to be the unit vector $e_j\in\R^{d+2}$.

  The remaining part is
  \begin{equation}\label{III_eigenequation2}\eta I_{3\times3} -
    \begin{pmatrix}
      D_{00} & D_{01} & D_{0,d+1} \\
      D_{10} & D_{11} & D_{1,d+1} \\
      D_{d+1,0} & D_{d+1,1} & D_{d+1,d+1}
  \end{pmatrix}=0.
  \end{equation}
  We want to find its eigenvalues and the corresponding eigenvectors. Here we shall use the method in \cite{Evans2010}.
  \begin{Thm}
    There exists $r_1>0$ such that the eigenvalues $\eta_j(\sigma,\tau,r)$ and the corresponding right eigenvectors $z_j(\sigma,\tau,r)$ of $(D_{jk})_{j,k=0,1,d+1}$ exist and are smooth in $B(0,r_1)\subset\R^3$. Futhermore, for $j=0,1,d+1$,
    \begin{align*}
      \eta_j(0,0,0) = \eta_{0,j},
    \end{align*}where
    \begin{align*}
        \eta_{0,0} = \sqrt{\alpha_1^2+\alpha_2^2},\,
        \eta_{0,1} = 0,\,
        \eta_{0,d+1} = -\sqrt{\alpha_1^2+\alpha_2^2}.
    \end{align*}
  \end{Thm}
  \begin{proof}
  1. Denote matrix $F(\sigma,\tau,r) := (D_{jk})_{j,k=0,1,d+1}$ and define
  \begin{align}
    f(\sigma,\tau,r,z,\eta) := ((F-\eta I)z, |z|^2)\in C^\infty(\R^7),
  \end{align}with $z\in\R^3$, $\eta\in\R$,
  We intend to use implicit function theorem near $f(\sigma,\tau,r,z,\eta)=(0,1)\in \R^4$.

  2. If $\sigma=\tau=r=0$, we have
    \begin{align*}
      (D_{jk})_{j,k=0,1,d+1}\Big|_{r=\tau=\sigma=0}=
      \begin{pmatrix}
        0 & \alpha_1 & 0 \\
        \alpha_1 & 0 & \alpha_2 \\
        0 & \alpha_2 & 0
    \end{pmatrix},
  \end{align*}
    where $\alpha_1 = (\xi^2_1\M^{1/2},\M^{1/2})_{L^2}$,
     $\alpha_2 = (\xi^2_1\M^{1/2},\psi_{d+1})_{L^2}$.
    Thus we can obtain three distinct real eigenvalues of $F(0,0,0)$ and their corresponding eigenvectors:
    \begin{align*}
        \eta_{0,0} &= \sqrt{\alpha_1^2+\alpha_2^2},\quad\quad z_{0,0} = (-\alpha_2,0,\alpha_1)^T/\sqrt{\alpha_1^2+\alpha_2^2},\\
        \eta_{0,1} &= 0,\quad\quad\quad\quad\quad\quad z_{0,1} = (\alpha_1,-\sqrt{\alpha_1^2+\alpha_2^2},\alpha_2)/\sqrt{2\alpha_1^2+2\alpha_2^2},\\
        \eta_{0,d+1} &= -\sqrt{\alpha_1^2+\alpha_2^2},\quad z_{0,d+1} = (\alpha_1,\sqrt{\alpha_1^2+\alpha_2^2},\alpha_2)/\sqrt{2\alpha_1^2+2\alpha_2^2}.
    \end{align*}
  Then for $j=0,1,d+1$,
  \begin{align*}
  f(0,0,0,z_{0,j},\eta_{0,j}) = (0,1).
  \end{align*}

  3. In order to use implicit function theorem, we need to verify that
  \begin{align*}
    \det\nabla_{z,\eta}f(0,0,0,z_{0,j},\eta_{0,j})\neq 0.
  \end{align*}
  Here
  \begin{align*}
    \nabla_{z,\eta}f(\sigma,\tau,r,z,\eta) =
    \begin{pmatrix}
      F-\eta I  &  -z   \\
      2z^T & 0
  \end{pmatrix}_{4\times 4}.
  \end{align*}
  Let $F_\varepsilon = F(0,0,0,z_{0,j},\eta_{0,j})-\eta_{0,j} I -\varepsilon I$. Then $F_\varepsilon z_{0,j} = -\varepsilon z_{0,j}$
  and so
  \begin{align*}
    \begin{pmatrix}
      F_\varepsilon & -z_{0,j}\\
      2z_{0,j}^T & 0
  \end{pmatrix}
  \times
  \begin{pmatrix}
    I_{3\times 3} & -\frac{z_{0,j}}{\varepsilon}\\
    0 & 1
  \end{pmatrix}
  =
  \begin{pmatrix}
    F_\varepsilon & 0\\
    2z_{0,j}^T & -\frac{2}{\varepsilon}
  \end{pmatrix}.
  \end{align*}
  Taking the determinant, we have
  \begin{align*}
  \begin{vmatrix}
      F_\varepsilon & -z_{0,j}\\
      2z_{0,j}^T & 0
  \end{vmatrix}
  &= \det(F(0,0,0,z_{0,j},\eta_{0,j})-\eta_{0,j} I -\varepsilon I)(\frac{-2}{\varepsilon})\\
  &= 2\prod_{k=0,1,d+1,k\neq j}(\eta_{0,k}-\eta_{0,j} -\varepsilon).
  \end{align*}
  Letting $\varepsilon \to 0$, we have
  \begin{align*}
    \det\nabla_{z,\eta}f(0,0,0,z_{0,j},\eta_{0,j}) = 2\Pi_{k=0,1,d+1,k\neq j}(\eta_{0,k}-\eta_{0,j})\neq 0.
  \end{align*}
  Then we can apply the implicit function theorem to get the smooth eigenvalues and eigenvectors near $(\sigma,\tau,r)=(0,0,0)$.
  \qe\end{proof}

  Therefore, for $j=0,1,d+1$, we can get the eigenvalues $\eta_j(\sigma,\tau,r)\in C^\infty(B(0,r_1);\R)$
  and eigenvectors $z_j(\sigma,\tau,r)\in C^\infty(B(0,r_1);\R^{d+2})$ to $(D_{jk})_{j,k=0}^{d+1}$, while the eigenvectors is still denoted by $z_j$ by keeping the $0^{th},1^{th},(d+1)^{th}$ component the same and supplementing the $2^{th}$ to $d^{th}$ to be $0$.

  \subsubsection{Asymptotic Behavior of the Eigenvalues to $D$.}
  Here we will investigate the derivatives of $\eta_j$ with respect to $\tau$ and $r$ at $(\sigma,\tau,r)=(0,0,0)$.

  For $j=2,\dots,d$, we know
  \begin{align*}
    \eta_j(\sigma,\tau,r) &= D_{jj}(\sigma,\tau,r).
  \end{align*}
  Thus from \eqref{III_lem14145}, \eqref{III_lem14146} and recall lemma \ref{III_lem12}, we have
  \begin{align*}
    \eta_j(0,0,0) &= D_{jj}(0,0,0) = (\xi_1\psi_j,\psi_j)_{L^2} =0,\\
    \partial_{\tau}\eta_j(0,0,0)
    &= -i((-L+P)^{-2}\xi_1\psi_j,\psi_j)_{L^2}\\
    &= -i( \xi_1\xi_j\M^{1/2},\xi_j\M^{1/2})_{L^2}
    = 0,\\
    \partial_{r}\eta_j(0,0,0)
    &= -2\pi i((-L+P)^{-1}\xi_1(-L+P)^{-1}\xi_1\psi_j,\psi_j)_{L^2}\\
    &= -2\pi i((-L+P)^{-1}\xi_1\psi_j,\xi_1\psi_j)_{L^2}.
  \end{align*}
  For the inner product $((-L+P)^{-1}\xi_1\psi_j,\xi_1\psi_j)_{L^2}$, we shall use the following lemma to deal with it.
  \begin{Lem}\label{III_lem17}
    Let $f\in \cap_{\beta\in\R}L^2_\beta$, then $(-L+P)^{-1}P^\perp f \in(\Ker L)^\perp$ and so
    \begin{align}
    ((-L+P)^{-1}f,f)_{L^2} &= (Pf,Pf)_{L^2} + (-L^{-1}P^\perp f,P^\perp f)_{L^2},
  \end{align}where $(-L^{-1}P^\perp f,P^\perp f)_{L^2}>0$ whenever $P^\perp f\neq 0$.
  \end{Lem}
  \begin{proof}
  Let $g = (- L+P)^{-1}P^\perp f$, then $(-L+P)g = P^\perp f$. Taking inner product with any $\psi\in\Ker L$, we have $(g,\psi) =0$.
  Thus $g\in(\Ker L)^\perp$ and so $-Lg=P^\perp f$, $g=-L^{-1}P^\perp f$. Thus
    \begin{align*}
      ((-L+P)^{-1}f,f)_{L^2}
      &= ((-L+P)^{-1}Pf,f)_{L^2} +  ((-L+P)^{-1}P^\perp f,f)_{L^2}\\
      &= (Pf,Pf)_{L^2} + (-L^{-1}P^\perp f,P^\perp f)_{L^2}.
    \end{align*}
    Also if $P^\perp f \neq 0$, then $(-L^{-1}P^\perp f,P^\perp f)_{L^2} = (h,-Lh)_{L^2}>0$, where $h=L^{-1}P^\perp f$.
  \qe\end{proof}

  With this lemma and noticing $P(\xi_1\psi_j)=0$, we have
  \begin{align*}
    \partial_{r}\eta_j(0,0,0)
    &= 2\pi i (L^{-1}\xi_1\xi_j\M^{1/2},\xi_1\xi_j\M^{1/2})_{L^2}.
  \end{align*}

  \vspace{1em}
  For $j=0,1,d+1$, in order to obtain the asymptotic behavior of $\eta_j$, we shall use the determinant. That is to let
  \begin{align*}
    f(\sigma,\tau,r,\eta) := \det(\eta I_{3\times 3} - (D_{jk})_{j,k=0,1,d+1}).
  \end{align*}
  Then $f(\sigma,\tau,r,\eta_j(\sigma,\tau,r)) = 0$, for $(\sigma,\tau,r)\in B(0,r_1)$,
    since $\eta_j$ is the eigenvalue of $(D_{jk})_{j,k=0,1,d+1}$.
  Taking the derivatives with respect to $\tau,r$, for $|(\sigma,\tau,r)|<r_1$, we have
  \begin{align}\label{III_eq187}
    \partial_\tau f + \partial_\eta f\cdot\partial_\tau\eta_j &=0,\\
    \partial_{r} f + \partial_\eta f\cdot\partial_{r}\eta_j &=0.\label{III_eq188}
  \end{align}
  Since $f(\sigma,\tau,r) = \det(\eta I - (D_{jk})_{j,k=0,1,d+1})$, we shall use the Jacobi's formula to calculate the derivative to the determinant of a matrix.
  Recall that
  $\alpha_1 = (\xi^2_1\M^{1/2},\M^{1/2})_{L^2}$,
   $\alpha_2 = (\xi^2_1\M^{1/2},\psi_{d+1})_{L^2}$ and applying \eqref{III_lem14145}, we have
  \begin{align}
    \partial_\tau f = \tr(\adj(\eta & I-(D_{jk})_{j,k=0,1,d+1})  (-\partial_\tau D_{jk})_{j,k=0,1,d+1})\notag\\
    \partial_\tau f(0,0,0,\eta_{0,j}) &= \tr\Bigg(\adj
    \begin{pmatrix}
        \eta_{0,j} & -\alpha_1 & 0 \\
        -\alpha_1 & \eta_{0,j} & -\alpha_2 \\
        0 & -\alpha_2 & \eta_{0,j}
    \end{pmatrix}\times\notag
    \begin{pmatrix}
      0 & i\alpha_1 & 0 \\
      i\alpha_1 & 0 & i\alpha_2 \\
      0 & i\alpha_2 & 0
  \end{pmatrix}\Bigg)\\
  &= 2i(\alpha^2_1+\alpha^2_2)\eta_{0,j}.\label{III_eq195}
  \end{align}
  If we define $\alpha_3 = ((-L+P)^{-1}\xi_1\psi_1,\xi_1\psi_1)_{L^2}$, $\alpha_4 = ((-L+P)^{-1}\xi_1\psi_{d+1},\xi_1\psi_{d+1})_{L^2}$,
  then similar to \eqref{III_eq195} and applying \eqref{III_lem14146}, we have
  \begin{align}
    \partial_r f(0,0,0,\eta_{0,j})\notag
    &= \tr\Bigg(
    \begin{pmatrix}
        \eta^2_{0,j}-\alpha_2^2 & \alpha_1\eta_{0,j} & \alpha_1\alpha_2 \\
        \alpha_1\eta_{0,j} & \eta^2_{0,j} & \alpha_2\eta_{0,j} \\
        \alpha_1\alpha_2 & \alpha_2\eta_{0,j} & \eta^2_{0,j}-\alpha_1^2
    \end{pmatrix}\times
    \begin{pmatrix}
      2\pi i\alpha_1 & 0 & 2\pi i\alpha_2 \\
      0 & 2\pi i\alpha_3 & 0\\
      2\pi i\alpha_2 & 0 & 2\pi i\alpha_4
  \end{pmatrix}\Bigg)\\
  &= 2\pi i\big(\alpha_1\eta^2_{0,j} +\alpha_3\eta^2_{0,j} + \alpha_4\eta^2_{0,j}+\alpha_1\alpha_2^2 -\alpha_1^2\alpha_4\big).\label{III_eq199}
  \end{align}
  Also one can easily get
  \begin{align}\label{III_eq1100}
    \partial_\eta f(0,0,0,\eta_{0,j}) = 3\eta_{0,j}^2 - \alpha^2_1-\alpha^2_2.
  \end{align}
  Thus from \eqref{III_eq187}, \eqref{III_eq188} and use \eqref{III_eq195}, \eqref{III_eq199}, \eqref{III_eq1100}, we can summarize:
  \begin{Thm}\label{III_eta}
    For cases $j=2,\dots,d$, we have
    \begin{align*}
      \eta_j(0,0,0) &= \partial_\tau\eta_j(0,0,0) = 0,\\
      \partial_{r}\eta_j(0,0,0) &= 2\pi i (L^{-1}\xi_1\xi_j\M^{1/2},\xi_1\xi_j\M^{1/2})_{L^2},
    \end{align*}
    where $(L^{-1}\xi_1\xi_j\M^{1/2},\xi_1\xi_j\M^{1/2})_{L^2}<0$.
  For cases $j=0,1,d+1$, we have
    \begin{align*}
      \eta_{0,0} = \sqrt{\alpha_1^2+\alpha_2^2},\quad
      \eta_{0,1} = 0,\quad
      \eta_{0,d+1} = -\sqrt{\alpha_1^2+\alpha_2^2},
    \end{align*}and
    \begin{align*}
      \partial_\tau\eta_j(0,0,0)
      &=\left\{\begin{aligned}
        &0,&&\text{ if } j=1,\\
        &-i\eta_{0,j},&&\text{  if } j=0,d+1,
      \end{aligned}\right.\\
      \partial_r\eta_j(0,0,0)
      &=\left\{\begin{aligned}
        &\frac{2\pi i\big(\alpha_1\alpha_2^2 -\alpha_1^2\alpha_4\big)}{\alpha^2_1+\alpha^2_2},&&\text{ if } j=1,\\
        &\frac{-\pi i\big((\alpha_1+\alpha_3 + \alpha_4)(\alpha_1^2+\alpha^2_2)+\alpha_1\alpha_2^2 -\alpha_1^2\alpha_4\big)}{\alpha^2_1+\alpha^2_2},&&\text{  if } j=0,d+1.
      \end{aligned}\right.
    \end{align*}
  \end{Thm}

  \subsubsection{The Eigenvalue Projection of $D$.}

  In the last section, we obtained $d+2$ smooth eigenvalues $\eta_j(\sigma,\tau,r)$ and $d+2$ smooth right eigenvectors $z_j(\sigma,\tau,r)\in\R^{d+2}$ to $(D_{jk})_{j,k=0}^{d+1}$, when $(\sigma,\tau,r)\in B(0,r_1)$.
  Notice that the dimension of Ker$L$ is $d+2$ and
  $(D_{jk})_{j,k=0,\dots,d+1}$ is the matrix representation of $D(\sigma,\tau,r,\omega)$ under the basis $\{R^T\psi_j\}$ of Ker$L$:
  \begin{align*}
    D(\sigma,\tau,r,\omega)(R^T\psi_0,\dots,R^T\psi_{d+1})
    = (R^T\psi_0,\dots,R^T\psi_{d+1})(D_{jk})_{j,k=0}^{d+1}.
  \end{align*}
  So we know that $\{\eta_j(\sigma,\tau,r)\}_{j=0,\dots,d+1}$ are the eigenvalues of $D(\sigma,\tau,r,\omega)$
  and the eigenspace of $D(\sigma,\tau,r,\omega)$ is exactly Ker$L$.
  Thus $D(\sigma,\tau,r,\omega)$ has eigenvectors:
  \begin{align*}
    \phi_j(\sigma,\tau,r) = \sum^{d+1}_{k=0}z^{(k)}_j(\sigma,\tau,r)R^T\psi_k\in C^\infty(B(0,r_1); L^2),
  \end{align*}where $(z^{(0)}_j,\dots,z^{(d+1)}_j) = z_j$.
  Define the smooth eigen-projections of $D(\sigma,\tau,r,\omega)$ on $L^2$ by
  \begin{align*}
    P_j(\sigma,\tau,r,\omega)f:&= (f, \phi_j)_{L^2}\phi_j\in  C^\infty(B(0,r_1);L(L^2_\beta)),
  \end{align*}for any $f\in L^2_\beta$, $\beta\in\R$.
  Then $\sum^{d+1}_{j=0}P_j(\sigma,\tau,r,\omega) = P$.

  \subsection{Eigenvalue Structure of $P(\lambda I -\BB)^{-1}P$.}
  Recall \eqref{III_eq28} that
  \begin{align*}
    P(\lambda I - \BB)^{-1}P &= \frac{1}{\lambda+1}\big(P-2\pi i r D(\sigma,\tau,r,\omega)\big).
  \end{align*}
  We regard $P(\lambda I - \BB)^{-1}P$ as an operator on Ker$L$,
   then we know its $j^{th}(j=0,\dots,d+1)$ eigenvalue is
   \begin{align}\label{III_eq1108}
     \mu_j(\sigma,\tau,r):=\frac{1}{\lambda+1}(1-2\pi ir\, \eta_j(\sigma,\tau,r))\in C^\infty(B(0,r_1)).
   \end{align}

  \begin{Thm}\label{III_asympotic}
  There exists $0<r_2\le r_1$ and $\sigma_j(r),\tau_j(r)\in C^\infty([-r_2,r_2])$, s.t.
  for $r\le r_2$,
  \begin{align}
    \mu_j(\sigma_j(r),\tau_j(r),r) &= 1.
  \end{align}
  Moreover,
  \begin{align}
    \sigma_j(r) &= \sigma_j^{(2)}r^2 + O(r^3),\\
    \tau_j(r) &= \tau_j^{(1)}r + O(r^3),
  \end{align}as $r\to 0$,
  where $\sigma_j^{(2)}<0$, $\tau_j^{(1)}\in\R$ with explicit expression
  \begin{align*}
    \tau_j^{(1)} = \left\{
    \begin{aligned}
      &0, &&\text{ if } j=1,\dots,d,\\
      &-2\pi \sqrt{1+2/d}, &&\text{ if } j=0,\\
      &2\pi \sqrt{1+2/d}, &&\text{ if } j=d+1,
    \end{aligned}
    \right.
  \end{align*}
  \begin{align*}
    \sigma_j^{(2)}&=\left\{\begin{aligned}
    &8\pi^2 (L^{-1}\xi_1\psi_j,\xi_1\psi_j)_{L^2}, && \text{ if } j=2,\dots,d,\\
    &\frac{8\pi^2}{1+2/d}(L^{-1}P^\perp(\xi_1\psi_{d+1}),P^\perp(\xi_1\psi_{d+1}))_{L^2},   &&   \text{ if } j=1,\\
    &4\pi^2 (L^{-1}P^\perp(\xi_1\psi_1),P^\perp(\xi_1\psi_1))_{L^2}\\
    &\qquad\qquad
    +\frac{8\pi^2}{d+2}(L^{-1}P^\perp(\xi_1\psi_{d+1}),P^\perp(\xi_1\psi_{d+1}))_{L^2},     &&\text{  if } j=0,d+1.
  \end{aligned}\right.
  \end{align*}
  \end{Thm}
  \begin{proof}
    1. Define $f = (\Re\mu_j,\Im\mu_j):B(0,r_1)\subset\R^3\to \R^2$ to be a smooth function.
    Notice
    \begin{align}
      \mu_j(0,0,0) = 1,\
      \partial_\sigma\mu_j(0,0,0) = -1,\
      \partial_\tau\mu_j(0,0,0) = -i.
    \end{align}
    Thus $f(0,0,0) = (1,0)$,
      $\det\nabla_{\sigma,\tau}f(0,0,0) = 1$.
    By implicit function theorem, there exists $r_2\in(0,r_1]$ and functions $\sigma_j(r)$, $\tau_j(r)\in C^\infty(|r|\le r_2)$ such that for $|r|\le r_2$,
    \begin{align*}
      \sigma(0) =\tau(0) =0,\
      \mu_j(\sigma_j(r), \tau_j(r),r) = 1.
    \end{align*}

  2. For $|r|\le r_2$, by \eqref{III_eq1108},
    \begin{align}
      1 = \mu_j(\sigma_j(r),\tau_j(r),r) \notag
      &=\frac{1}{\sigma_j(r)+i\tau_j(r)+1}(1-2\pi ir\eta_j(\sigma_j(r),\tau_j(r),r)),\\
      \sigma_j(r)+i\tau_j(r) &= -2\pi ir\eta_j(\sigma_j(r),\tau_j(r),r).\label{III_eq493}
    \end{align}Using \eqref{III_eq493} and applying the behavior of $\eta_j$ from \ref{III_eta}, we have
    \begin{align*}
      \sigma'_j(0)+i\tau'_j(0) &= -2\pi i\eta_j(0,0,0)
      = \left\{
      \begin{aligned}
        &0, &&\text{ if } j=1,\dots,d,\\
        &-2\pi i\sqrt{\alpha_1^2+\alpha_2^2},&&\text{ if } j=0,\\
        &2\pi i\sqrt{\alpha_1^2+\alpha_2^2},&&\text{ if } j=d+1.\\
      \end{aligned}
      \right.
    \end{align*}
    So $\sigma'_j(0)=0$ and $\tau'_j(0) = -2\pi \eta_j(0,0,0)$, then
    \begin{align*}
      \sigma''_j(0)+i\tau''_j(0) &=
      -4\pi i\big(\partial_\sigma\eta_j\cdot\sigma'_j + \partial_\tau\eta_j\cdot\tau'_j + \partial_{r}\eta_j\big)|_{\sigma=\tau=r=0}\\
      &= 8\pi^2i\,\partial_\tau\cdot\eta_j(0,0,0)\eta_j(0,0,0) -4\pi i\partial_{r}\eta_j(0,0,0).
    \end{align*}
    If $j=2,\dots,d$, then
    \begin{align*}
      \sigma''_j(0)+i\tau''_j(0)
      &=8\pi^2 (L^{-1}\xi_1\xi_j\M^{1/2},\xi_1\xi_j\M^{1/2})_{L^2} <0.
    \end{align*}
    If $j=0,1,d+1$, then
  \begin{align}
    &\sigma''_j(0)+i\tau''_j(0)\notag\\
    &=\left\{\begin{aligned}
    &  \frac{8\pi^2 \alpha_1\big(\alpha_2^2 -\alpha_1\alpha_4\big)}{\alpha^2_1+\alpha^2_2},   &&\text{ if } j=1,\\
    &4\pi^2 \big(2 (\alpha_1^2+\alpha_2^2) - \frac{\big((\alpha_1+\alpha_3 + \alpha_4)(\alpha_1^2+\alpha^2_2)+\alpha_1\alpha_2^2 -\alpha_1^2\alpha_4\big)}{\alpha^2_1+\alpha^2_2}\big), &&\text{ if } j=0,d+1.
  \end{aligned}\right.\label{III_eq88}
  \end{align}
  Using Gamma function, we can calculate that
  \begin{align*}
    \alpha_1 = (\xi_1\M^{1/2},\xi_1\M^{1/2})_{L^2} = 1,\quad
    \alpha_2 = (\xi_1\M^{1/2},\xi_1\psi_{d+1})_{L^2} = \sqrt{\frac{2}{d}},
  \end{align*}
  Also by lemma \ref{III_lem17}, we have
  \begin{align*}
    \alpha_3 = ((-L+P)^{-1}\xi_1\psi_1,\xi_1\psi_1)_{L^2}
    &= \|P(\xi_1\psi_1)\|^2_{L^2} + (-L^{-1}P^\perp(\xi_1\psi_1),P^\perp(\xi_1\psi_1))_{L^2},\\
    \alpha_4 = ((-L+P)^{-1}\xi_1\psi_{d+1},\xi_1\psi_{d+1})_{L^2}
    &= \|P(\xi_1\psi_{d+1})\|^2_{L^2} + (-L^{-1}P^\perp(\xi_1\psi_{d+1}),P^\perp(\xi_1\psi_{d+1}))_{L^2},
  \end{align*}
  and here
  \begin{align*}
    \|P(\xi_1\psi_1)\|^2_{L^2}
    = \sum^{d+1}_{j=0}|(\xi^2_1\M^{1/2},\psi_j)_{L^2}|^2
    &= |(\xi^2_1\M^{1/2},\M^{1/2})_{L^2}|^2 + |(\xi^2_1\M^{1/2},\psi_{d+1})_{L^2}|^2
    = 1 + \frac{2}{d},\\
    \|P(\xi_1\psi_{d+1})\|^2_{L^2}
    = \sum^{d+1}_{j=0}|(\xi_1\psi_{d+1}&,\psi_j)_{L^2}|^2
    = |(\xi_1\psi_{d+1},\xi_1\M^{1/2})_{L^2}|^2
    = \frac{2}{d}.
  \end{align*}Substitute these values into \eqref{III_eq88} and we will get the explicit expression of $\tau''_j(0)$ and $\sigma''_j(0)$.
  \qe
  \end{proof}

  Denote
  \begin{align}
    H(\sigma,\tau,r,\omega) = \frac{1}{\lambda+1}\big(P-2\pi i r D(\sigma,\tau,r,\omega)\big).
  \end{align}

  As an operator defined on finite dimensional space $\Ker L$, for $|r|\le r_2$, we have
  \begin{align*}
    I-H(\sigma,\tau,r,\omega)
      &= \sum^{d+1}_{j=0}\big(1-\mu_j(\sigma,\tau,r)\big)P_j(\sigma,\tau,r,\omega).
  \end{align*}
  Then we claim that on Ker$L$,
  \begin{align*}
    \big(I-H(\sigma,\tau,r,\omega)\big)^{-1} &= \sum^{d+1}_{j=0}\big(1-\mu_j(\sigma,\tau,r)\big)^{-1}P_j(\sigma,\tau,r,\omega),
  \end{align*}here the inverse is taken in Ker$L$.
  Indeed, on $\ker L$,
  \begin{align*}
    \big(I&-H(\sigma,\tau,r,\omega)\big)\sum^{d+1}_{j=0}\big(1-\mu_j(\sigma,\tau,r)\big)^{-1}P_j(\sigma,\tau,r,\omega)\\
    &= \sum^{d+1}_{j=0}\big(1-\mu_j(\sigma,\tau,r)\big)P_j(\sigma,\tau,r,\omega)
    \times\sum^{d+1}_{k=0}\big(1-\mu_k(\sigma,\tau,y)\big)^{-1}P_k(\sigma,\tau,y)\\
    &= \sum^{d+1}_{j=0}P_j(\sigma,\tau,r,\omega)=P.
  \end{align*}

  Therefore for $\Re\lambda\ge 0$, $y\in\Rd\setminus\{0\}$,
  \begin{equation}\label{III_eqtem1}
    \begin{split}
        &(\lambda I - \B)^{-1}\\ &= (\lambda I - \BB)^{-1}
        + (\lambda I - \BB)^{-1}P(I-H(\sigma,\tau,r,\omega))^{-1}P(\lambda I - \BB)^{-1}\\
        &= (\lambda I - \BB)^{-1}
        + \sum^{d+1}_{j=0}\big(1-\mu_j(\sigma,\tau,r)\big)^{-1}
        (\lambda I - \BB)^{-1}P_j(\sigma,\tau,r,\omega)(\lambda I - \BB)^{-1}.
    \end{split}
  \end{equation}
  Denote $U_j(\sigma,\tau,y)=(\lambda I - \BB)^{-1}P_j(\sigma,\tau,r,\omega)(\lambda I - \BB)^{-1}$.
  Differentiate \eqref{III_eqtem1} with respect to $\tau$, we have
  \begin{align}\label{III_eq499}
    (\lambda I - \B)^{-n-1} =  (\lambda I - \BB)^{-n-1}
    + \sum^{d+1}_{j=0}\sum^{n}_{k=0}\big(1-\mu_j(\sigma,\tau,r)\big)^{-k-1}
    U^{(n)}_{j,k}(\sigma,\tau,y),
  \end{align}
  where $U^{(n)}_{j,k}(\sigma,\tau,y)\in C^\infty(\{(\sigma,\tau,y):|(\sigma,\tau,y)|\le r_2\}; L(L^2_\beta))$ is given as a linear combination of products of $\mu_j,U_j$ and their derivatives with respect to $\tau$, with
  \begin{align*}
    \sup_{(\sigma,\tau,r)\le \overline{B}(0,r_2)}
    \|U^{(n)}_{j,k}(\sigma,\tau,y)\|_{L(L^2_\beta)}<\infty.
  \end{align*}

  \section{Estimate on the Semigroup $e^{t\B}$}
  In this section, we will give the proof of boundedness of semigroup $e^{t\B}$ and its asymptotic behavior when $t\to\infty$.

  Fisrtly we need the resolvent identities and the inversion semigroup formula.
  \begin{align}\label{V_eq514}
    (\lambda I-\B)^{-1}
    &= (I-(\lambda I-\A)^{-1}K))^{-1}(\lambda I -\A)^{-1},
  \end{align}
  and for $\sigma>0$,
    \begin{align}\label{IV_inversion_semigroup}
      e^{t\B}u = \slim_{a\to \infty}\frac{1}{2\pi i}\int^{\sigma+ia}_{\sigma-ia}e^{\lambda t}(\lambda I-\B)^{-1}u\,d\lambda,
    \end{align}for $u\in D_0$, where the limit in taken with respect to $L^2(\R^d_\xi)$ norm.

  Now we investigate the right hand side of \eqref{IV_inversion_semigroup}.
  Consider $L^2(\Rd)$ to be the whole space.
  For $y\in\Rd\setminus\{0\}$, we have $\{\Re\lambda\ge 0\}\subset \rho(\B)$ and thus
  $(\lambda I - \B)^{-1}:\{\Re\lambda\ge 0\}\to L(L^2)$ is a holomorphic operator-valued function with respect to $\lambda$ in $\{\Re\lambda\ge 0\}$. Applying Cauchy thoerem, for $u\in L^2$, $\sigma>0$,
  \begin{align*}
    \notag\frac{1}{2\pi i}\int^{\sigma+ia}_{\sigma-ia}e^{\lambda t}(\lambda I-\B)^{-1}u\,d\lambda
    &=\frac{1}{2\pi i}\Big(\int^{ia}_{-ia}+\int^{\sigma+ia}_{ia}+\int^{-ia}_{\sigma-ia}\Big)e^{\lambda t}(\lambda I-\B)^{-1}u\,d\lambda\\
    &=I_1u + I_2u + I_3u.
  \end{align*}

  \begin{Rem}
    The whole space really matters, since only in $L^2$ we can use $\{\Re\lambda\ge 0\}\subset \rho(\B)$. But later we will assume $u\in L^2_\beta\subset L^2$, with $\beta\ge 0$.
  \end{Rem}

  Firstly we consider the part $I_2$ and $I_3$.
  If futhermore $u \in L^2_{\beta+\gamma}$, then
  \begin{align*}
    I_2u&=\frac{1}{2\pi i}\int^{\sigma+ia}_{ia}e^{\lambda t}(\lambda I-\B)^{-1}u\,d\lambda\\
    &= \frac{1}{2\pi i}\int^{\sigma+ia}_{ia}(I-(\lambda I-\A)^{-1}K))^{-1}(\lambda I -\A)^{-1}u\,d\lambda.
  \end{align*}
  Notice here $y$ is fixed, then by theorem \ref{II_continuous}, $\sup_{\lambda\in\overline{C_+}}\|(I-(\lambda I-\A)^{-1}K))^{-1}\|_{L(L^2_\beta)}\le C_{y,\beta}<\infty$.
  Thus by theorem \ref{II_main_estimate},
  \begin{align*}
    \|I_2u\|_{L^2_\beta}
    &\le \frac{C_y}{2\pi}\int^{\sigma+ia}_{ia}e^{\sigma t}\|(\lambda I -\A)^{-1}u\|_{L^2_\beta}\,d\lambda\\
    &\le C_{y,\nu}e^{\sigma t}\int^{\sigma+ia}_{ia}\,d\lambda \big(\|u\|_{L^2_{\beta+\gamma}(|\xi|\ge R)}+
    C_\nu\|u\|_{L^2_{\beta+\gamma}}\frac{1}{|a|}\big)\to 0,
  \end{align*}as $R\to\infty$, and hence $a\ge 2\pi |y|R\to\infty$. The part $I_3$ is similar.
  Thus $I_2u,I_3u\to 0$ in $L^2_\beta$ norm as $a\to \infty$ if $u\in L^2_{\beta+\gamma}$.


  To deal with the part $I_1$, we need the following lemma.
  \begin{Lem}
    For any $y_1>0$, $\sigma\ge 0$, we have
    \begin{align*}
      &\int_\R \|((\sigma+i\tau)I-\A)^{-1}u\|^2_{L^2_\beta}\,d\tau
      \le C_{\nu}\|u\|^2_{L^2_{\beta+\gamma/2}}.
    \end{align*}
    Thus for $|y|\ge y_1$,
    \begin{align*}
      \int_\R \|((\sigma+i\tau)I-\B)^{-1}u\|^2_{L^2_\beta}\,d\tau
      &\le C_{\nu,y_1,\beta}\|u\|^2_{L^2_{\beta+\gamma/2}}.
    \end{align*}

  \end{Lem}

  \begin{proof}
    For $\sigma\ge 0$,
    \begin{align*}
      \int_\R \|((\sigma+i\tau)I-\A)^{-1}u\|^2_{L^2_\beta}\,d\tau
      &=\int_\Rd(1+|\xi|)^{2\beta}|u(\xi)|^2\int_\R\frac{1}{|\sigma+\nu(\xi)|^2+|\tau+2\pi y\cdot\xi|^2}\,d\tau d\xi\\
      &\le\int_\Rd(1+|\xi|)^{2\beta}|u(\xi)|^2\int_\R\frac{1}{|\nu(\xi)|^2+\tau^2}\,d\tau d\xi\\
      &\le C_{\nu}\|u\|^2_{L^2_{\beta+\gamma/2}}.
    \end{align*}
    Then using the fact that $
      \sup_{\lambda\in\overline{C_+},|y|\ge y_1}\|(I-(\lambda I-\A)^{-1}K))^{-1}\|_{L(L^2_\beta)}\le C_{y_1,\beta}<\infty$, and the resolcent identity \eqref{V_eq514}, we get the second assertion.
  \qe\end{proof}

  Now for the part $I_1$,
  \begin{align}
    I_1u&=\frac{1}{2\pi i}\int^{+a}_{-a}e^{i\tau t}(i\tau I-\B)^{-1}u\,d\tau.
  \end{align}
  Notice for $\tau\in\R$, $y\in\Rd\setminus\{0\}$, we have $i\tau \in \rho(\B)$, and
  \begin{align}
    \frac{d^k}{d\tau^k}\big((i\tau I-\B)^{-1}u(\xi)\big)
    &= i^kk!(-1)^{k}(i\tau I-\B)^{-k-1}u(\xi).
  \end{align}
  Thus using integration by parts,
  \begin{align*}
    &\int^{+a}_{-a}e^{i\tau t}(i\tau I-\B)^{-1}u(\xi)\,d\tau\\
    &=\frac{e^{i\tau t}}{it}(i\tau I-\B)^{-1}u(\xi)\Big|^{\tau=a}_{\tau=-a}
    -\int^a_{-a}\frac{d}{d\tau}\big((i\tau I-\B)^{-1}u(\xi)\big)\frac{e^{i\tau t}}{it}\,d\tau\\
    &=\cdots\\
    &=\sum^n_{k=1}\frac{(-1)^{k-1}e^{i\tau t}}{(it)^k}
    \frac{d^{k-1}}{d\tau^{k-1}}\big((i\tau I-\B)^{-1}u(\xi)\big)\Big|^{\tau=a}_{\tau=-a}\\
    &\qquad\qquad\qquad\qquad\qquad\qquad+(-1)^n\int^a_{-a}\frac{d^n}{d\tau^n}\big((i\tau I-\B)^{-1}u(\xi)\big)\frac{e^{i\tau t}}{(it)^n}\,d\tau\\
    &=\sum^n_{k=1}\frac{e^{i\tau t}}{it^k}
    (k-1)!(i\tau I-\B)^{-k}u(\xi)\Big|^{\tau=a}_{\tau=-a}+
    n!\int^a_{-a}(i\tau I-\B)^{-n-1}u(\xi)\frac{e^{i\tau t}}{t^n}\,d\tau.
  \end{align*}
  Recall theorem \ref{II_main_estimate} that
  for $R>0$, $\Re\lambda>0$ with $|\lambda|\ge 4\pi |y| R$,
    \begin{align*}
      \|(\lambda I - \A)^{-1}u\|_{L^2_\beta}
      \le C_{\nu}\|u\|_{L^2_{\beta+\gamma}(|\xi|\ge R)}+
      C_\nu\|u\|_{L^2_{\beta+\gamma}}\frac{1}{|\lambda|}.
    \end{align*}
  Without loss of generality, we can let $a\ge 1$. Notice that here $y\neq 0$ is fixed, then applying \ref{II_coro} and \ref{II_continuous}, we have
    \begin{align*}
      \|(ia I-\B)^{-k}u(\xi)\|_{L^2_\beta}
      &\le C^{k-1}_\nu\|(ia I-\B)^{-1}u(\xi)\|_{L^2_{\beta+(k-1)\gamma}}\\
      &\le C^k_{\nu}\big(\|u\|_{L^2_{\beta+k\gamma}(|\xi|\ge R)}+
      \|u\|_{L^2_{\beta+k\gamma}}\frac{1}{|a|}\big) \to 0,
    \end{align*}as $R\to \infty$ and $a\to \infty$, if $u\in L^2_{\beta+k\gamma}$. Thus it suffices to deal with following term when $a\to\infty$.
  \begin{align*}
    n!\int^a_{-a}(i\tau I-\B)^{-n-1}u(\xi)\frac{e^{i\tau t}}{t^n}\,d\tau.
  \end{align*}

  \subsection{$y$ away from origin}\label{V_subsection1}

   Notice from \ref{II_coro}, we have for $\beta\in\R$, $b>0$ that
  \begin{align*}
    C_{b}:=\sup_{\Re\lambda\ge 0,|y|\ge b}\|(\lambda I-\B)^{-1}\|_{L(L^2_{\beta+\gamma},L^2_{\beta})}
    <\infty.
  \end{align*}
  So if $|y|\ge b$, for $u\in L^2_{\beta+n\gamma}$, $v\in L^2_\beta$, we have
  \begin{align*}
    &\big|\big(n!\int^a_{-a}(i\tau I-\B)^{-n-1}u(\xi)\big)\frac{e^{i\tau t}}{t^n}\,d\tau,\, v\big)_{L^2_\beta}\big|\\
    &\le \frac{n!}{t^n}\int^a_{-a}\|(i\tau I-\B)^{-n}u(\xi)\|_{L^2_{\beta+\gamma/2}}\|(-i\tau I-\B^*)^{-1}v(\xi)\|_{L^2_{\beta-\gamma/2}}\,d\tau\\
    &\le \frac{n!}{t^n}
    \Big(\int^a_{-a}\|(i\tau I-\B)^{-n}u(\xi)\|^2_{L^2_{\beta+\gamma/2}}\,d\tau\Big)^{1/2}
    \Big(\int^a_{-a}\|(-i\tau I-\B^*)^{-1}v(\xi)\|^2_{L^2_{\beta-\gamma/2}}\,d\tau\Big)^{1/2}\\
    &\le \frac{C_{r_0}n!}{t^n}\|u\|_{L^2_{\beta+n\gamma}}\|v\|_{L^2_\beta}.
  \end{align*}
  Notice here $\B^*=2\pi i y\cdot\xi+L$ has the same boundedenss as $\B$. Thus if $u\in L^2_{\beta+n\gamma}$, $\{n!\int^a_{-a}(i\tau I-\B)^{-n-1}u(\xi)\frac{e^{i\tau t}}{t^n}\,d\tau\}^\infty_{a=1}$ is a bounded sequence in $(L^2_\beta)^*$,
  hence has a weakly $*$ convergent subsequence (denoted by $\{a_k\}$) and its weak limit $Iu$ is controlled by $\frac{n!}{t^n}\|u\|_{L^2_{\beta+n\gamma}}$.
  That is
  \begin{align}
    Iu := \weaklim_{a_k\to\infty}I_1u&=\weaklim_{a_k\to\infty}\frac{n!}{2\pi it^n}\int^{a_k}_{-a_k}e^{i\tau t}(i\tau I-\B)^{-n-1}u(\xi)\,d\tau,
  \end{align}and
  \begin{align}\label{IV_eq129}
    \|Iu\|_{L^2_\beta} = \|\weaklim_{a_k\to\infty}I_1u\|_{L^2_\beta}
    \le \frac{C_{r_0}n!}{t^n}\|u\|_{L^2_{\beta+n\gamma}}.
  \end{align}

  \begin{Rem}
    Here the integral region can be replaced by $(b,a)$, for any $b>0$.
  \end{Rem}

  \subsection{$y$ near origin}\label{subsection_nearorigin}
  For $y$ near the origin but $y\neq 0$, we can also use the method of integration by parts to obtain
  \begin{align*}
    \slim_{a\to \infty}\frac{1}{2\pi i}\int^{\sigma+ia}_{\sigma-ia}e^{\lambda t}(\lambda I-\B)^{-1}u\,d\lambda
    &=\slim_{a\to\infty}\frac{n!}{2\pi it^n}\int^{a}_{-a}e^{i\tau t}(i\tau I-\B)^{-n-1}u\,d\tau,
  \end{align*}where the limit is taken in $L^2_\beta$.
  Then for $b\in(0,a)$, we divide the integral region to be
  \begin{align*}
    \int^{a}_{-a} = \int^{a}_{b} + \int^{b}_{-b} + \int^{-b}_{-a}.
  \end{align*}
  The first and the last term has a weak limit by same argument as in section \ref{V_subsection1}, where we essentially need the uniformly boundedness when $|\Im\lambda|\ge b$ from Corollary \ref{II_coro} that
  \begin{align*}
    C_{b}:=\sup_{\Re\lambda\ge 0,|\Im\lambda|\ge b,y\in\Rd}\|(\lambda I-\B)^{-1}\|_{L(L^2_{\beta+\gamma},L^2_\beta)}
    <\infty.
  \end{align*}
  So it reamins to deal with the integral
  \begin{align}\label{IV_eq132}
    \int^{b}_{-b}e^{i\tau t}(i\tau I-\B)^{-n-1}u\,d\tau.
  \end{align}
  We will use the identity \eqref{III_eq499} from section 4. The following lemma is used for controlling the term $(1-\mu_j(\sigma,\tau,r))^{-1}$ in \eqref{III_eq499}.
  \begin{Lem}\label{IV_14}
    Let $f(x,y,z)\in C^1(\{(x,y,z)\in \R^3:|(x,y,z)|\le r\})$.
    Suppose $\nabla_{(x,y,z)}f(0,0,0)=-(1,i,a)$, with some constant $a\in\C$.
    Then for $|(x_1,y_1)|\le r_1$, $|(x_2,y_2)|\le r_1$,
    \begin{align*}
      \frac{1}{2}|(x_1-x_2,y_1-y_2)|\le |f(x_1,y_1,z)-f(x_2,y_2,z)|\le \frac{3}{2}|(x_1-x_2,y_1-y_2)|.
    \end{align*}
    \begin{proof}
      By mean value theorem,
      \begin{align*}
        f(x_1,y_1,z)-f(x_2,y_2,z) = \nabla_{(x,y,z)}f(x_\theta,y_\theta,0)\cdot(x_1-x_2,y_1-y_2,0),
      \end{align*}
      where $x_\theta = \theta x_1+ (1-\theta)x_2$, $y_\theta = \theta y_1+ (1-\theta)y_2$,
      with $\theta\in(0,1)$.
      Take $r_1\in(0,r)$ so small that for $|(x_\theta,y_\theta)|\le r_1$,
      \begin{align*}
        |\nabla_{(x,y,z)}f(x_\theta,y_\theta,0)-\nabla_{(x,y,z)}f(0,0,0)|<\frac{1}{2}.
      \end{align*}
  Then
      \begin{align*}
        f(x_1,y_1,z)-f(x_2,y_2,z) &= \nabla_{(x,y,z)}f(0,0,0)\cdot(x_1-x_2,y_1-y_2,0)\\
        &\qquad+\left(\nabla_{(x,y,z)}f(x_\theta,y_\theta,0)-\nabla_{(x,y,z)}f(0,0,0)\right)\cdot(x_1-x_2,y_1-y_2,0),\\
        \left|f(x_1,y_1,z)-f(x_2,y_2,z)\right|
        &\ge \left|(1,i,a)\cdot(x_1-x_2,y_1-y_2,0)\right| - \frac{1}{2}\left|(x_1-x_2,y_1-y_2,0)\right|\\
        &=\frac{1}{2}\left|(x_1-x_2,y_1-y_2)\right|,
      \end{align*}and the second inequality is similar.
    \qe\end{proof}
  \end{Lem}

  Write $y=r\omega$, with $r\in\R_+$, $\omega\in\S^{d-1}$. Applying theorem \ref{III_asympotic}, for $r\le r_2$,
  \begin{align*}
    (1-\mu_j(\sigma,\tau,r))^{-1} &= (\mu_j(\sigma_j(r),\tau_j(r),r)-\mu_j(\sigma,\tau,r))^{-1}.
  \end{align*}Thus by lemma \ref{IV_14},
  \begin{align*}
    \frac{2}{3}|(\sigma_j(r)-\sigma,\tau_j(r)-\tau)|\le |1-\mu_j(\sigma,\tau,r)|^{-1} \le 2|(\sigma_j(r)-\sigma,\tau_j(r)-\tau)|.
  \end{align*}
  By asymptotic behavior of $\sigma(r)$ and $\tau(r)$ in \ref{III_asympotic}, there exists $\eta_0>0$, and $r_3\in(0,r_2)$ such that for $r\le r_3$,
  \begin{align*}
    \sigma_j(r)\le -2\eta_0r^2,\quad
    |\tau_j(r)- \tau^{(1)}_jr| \le \eta_0r^2.
  \end{align*}
  Thus for $r\le r_3$, the equation \eqref{IV_eq132} becomes
  \begin{align*}
    &\int^{b}_{-b}e^{i\tau t}(i\tau I-\B)^{-n-1}u(\xi)\,d\tau\\
    &= \int^{b}_{-b}e^{i\tau t}\Big((\lambda I - \BB)^{-n-1}
    + \sum^{d+1}_{j=0}\sum^{n}_{k=0}\big(1-\mu_j(\sigma,\tau,r)\big)^{-k-1}
    U^{(n)}_{j,k}(\sigma,\tau,y)\Big)u(\xi)\,d\tau\\
    &= \int^{b}_{-b}e^{i\tau t}(\lambda I - \BB)^{-n-1}u(\xi)\,d\tau
    +\sum^{d+1}_{j=0}\sum^{n}_{k=0}\int^{b}_{-b}e^{i\tau t}\big(1-\mu_j(\sigma,\tau,r)\big)^{-k-1}
    U^{(n)}_{j,k}(\sigma,\tau,y)u(\xi)\,d\tau\\
    &=: I_0u + \sum^{d+1}_{j=0}\sum^{n}_{k=0}I^{(n)}_{j,k}u.
  \end{align*}
  Noticing that $\sup_{\Re\lambda\ge 0, y\in\Rd}\|(\lambda I - \BB)^{-1}\|_{L(L^2_{\beta+\gamma},L^2_\beta)}<\infty$, we have
  \begin{align}\label{IV_eq148}
    \|I_0u\|_{L^2_\beta} 
    \le C_{b,\beta}\|u\|_{L^2_{\beta+n\gamma}}.
  \end{align}
  On the other hand, since $U^{(n)}_{j,k}(0,\tau,y)$ is smooth in $\{(\tau, y)\in [-r_3,r_3]\times \Rd\}$, we have
  \begin{align*}
    \sup_{|\tau|\le r_3, y\in\Rd}\|U^{(n)}_{j,k}(0,\tau,y)\|_{L(L^2_\beta)}=C_{\beta}<\infty.
  \end{align*}and then if $b\in(0,r_3]$
  \begin{align*}
    \|I^{(n)}_{j,k}u\|_{L^2_\beta}
    &\le \left\|\int^{b}_{-b}e^{i\tau t}\big(1-\mu_j(0,\tau,r)\big)^{-k-1}
    U^{(n)}_{j,k}(0,\tau,y)u(\xi)\,d\tau\right\|_{L^2_\beta}\\
    &\le C_{\beta,n}\|u(\xi)\|_{L^2_{\beta}}\int^{b}_{-b}\big|1-\mu_j(0,\tau,r)\big|^{-k-1}\,d\tau\\
    &\le C_{\beta,n}\|u(\xi)\|_{L^2_{\beta}}\int^{b}_{-b}\frac{1}{\big(|\sigma_j(r)|+|\tau-\tau_j(r)|\big)^{k+1}}\,d\tau.
  \end{align*}
  Now for $r\le r_3$, we have $\sigma_j(r)\le -2\eta_0r^2$,
    $|\tau_j(r)- \tau^{(1)}_jr| \le \eta_0r^2$.
  Thus for $k\in\N$,
  \begin{align*}
    \int^{b}_{-b}\frac{1}{\big(|\sigma_j(r)|+|\tau-\tau_j(r)|\big)^{k+1}}\,d\tau
    &\le\int^{b}_{-b}\frac{1}{\big(\eta_0r^2+|\tau-\tau^{(1)}_jr|\big)^{k+1}}\,d\tau\\
    &\le\int^{b+|\tau^{(1)}_j|r}_{0}\frac{2}{\big(\eta_0r^2+\tau\big)^{k+1}}\,d\tau.
  \end{align*}
  If $k=0$, then
  \begin{align*}
    \int^{b}_{-b}\frac{1}{\big(|\sigma_j(r)|+|\tau-\tau_j(r)|\big)^{k+1}}\,d\tau
    &\le 2\log \Big(\frac{b+|\tau^{(1)}_j|r+\eta_0r^2}{\eta_0r^2}\Big)\le C\log(r^{-1}+e).
  \end{align*}
  If $k>0$, then
  \begin{align*}
    \int^{b}_{-b}\frac{1}{\big(|\sigma_j(r)|+|\tau-\tau_j(r)|\big)^{k+1}}\,d\tau
    &\le \frac{2}{-k}\Big[(\eta_0r^2+b+|\tau^{(1)}_j|r)^{-k} - (\eta_0r^2)^{-k}\Big]\le \frac{C_k}{r^{2k}}.
  \end{align*}
  Denote
  \begin{equation*}
    \rho_n(y):=\left\{\begin{aligned}
      &\log(|y|^{-1}+e),&&\text{ if }n=0,\\
      &\frac{1}{|y|^{2n}},&&\text{ if }n\ge 1.
    \end{aligned}\right.
  \end{equation*}
  Then for $k=\N$, $r\le r_3<1$,
  \begin{align}
    \|I^{(n)}_{j,k}u\|_{L^2_\beta}
    &\le C_{\beta}\rho_k(y)\|u(\xi)\|_{L^2_{\beta}},\notag\\\label{IV_eq168}
    \|\sum^{d+1}_{j=0}\sum^n_{k=0}I^{(n)}_{j,k}u\|_{L^2_\beta}
    &\le C_{\beta,d}\rho_n(y)\|u\|_{L^2_{\beta}}.
  \end{align}

  For $|y|,b\in(0,r_3]$,
  combining \eqref{IV_eq148} and \eqref{IV_eq168}, we have
  \begin{align*}
    \|\int^{b}_{-b}e^{i\tau t}(i\tau I-\B)^{-n-1}u(\xi)\,d\tau\|_{L^2_\beta}
    &\le C_{r_3,b,\beta}\|u\|_{L^2_{\beta+n\gamma}} + C_{\beta,n,d}\rho_n(y)\|u(\xi)\|_{L^2_{\beta}}.
  \end{align*}
  Thus for $u\in D_\beta$, $\beta\ge 0$,
  \begin{equation}\label{IV_173}
    \begin{split}
    \|e^{t\B}u\|_{L^2_\beta}
    &= \|\slim_{a\to\infty}\frac{n!}{2\pi it^n}
    \Big(\int^{a}_{b}+\int^{b}_{-b}+\int^{-b}_{-a}\Big)e^{i\tau t}(i\tau I-\B)^{-n-1}u(\xi)\,d\tau\|_{L^2_\beta}\\
    &\le \|\slim_{a\to\infty}\Big(\int^{a}_{b}+\int^{-b}_{-a}\Big)e^{i\tau t}(i\tau I-\B)^{-n-1}u(\xi)\,d\tau\|_{L^2_\beta}
    + \|I_0u + \sum^{d+1}_{j=0}\sum^{n}_{k=0}I^{(n)}_{j,k}u\|_{L^2_{\beta}}\\
    &\le \frac{C_{n,\beta,d,b}}{t^n}\Big(\|u\|_{L^2_{\beta+n\gamma}} + \rho_n(y)\|u\|_{L^2_{\beta}}\Big).
  \end{split}
  \end{equation}

  \subsection{Result}
  Now for $|y|\ge r_3$, we have \eqref{IV_eq129}. Together with the estimate \eqref{IV_173} for $|y|\le r_3$, we have
  \begin{align*}
    \|e^{t\B}u\|_{L^2_\beta} &\le \frac{C_{r_3,n,\beta,d}}{t^n}\Big(\|u\|_{L^2_{\beta+n\gamma}} + \rho_n(y)\chi_{|y|\le r_3}\|u\|_{L^2_{\beta}}\Big).
  \end{align*}
  But also the semigroup estimate on $e^{t\B}$ gives $\|e^{t\B}u\|_{L^2_\beta} \le e^{t\|K\|_{L(L^2_\beta)}}\|u\|_{L^2_\beta}$.
  Thus for $n\in\N$,
  \begin{align*}
    \|e^{t\B}u\|_{L^2_\beta} &\le \frac{C_{r_3,n,\beta,d}}{(1+t)^n}\Big(\|u\|_{L^2_{\beta+n\gamma}} + \rho_n(y)\chi_{|y|\le r_3}\|u\|_{L^2_{\beta}}\Big).
  \end{align*}
  But here in order to use interpolation theorem, we can only use a weaker result:
  \begin{align*}
    \|e^{t\B}u\|_{L^2_\beta} &\le \frac{C_{r_3,\beta,d}(1+\rho_n(y))\chi_{|y|\le r_3}}{(1+t)^n}\|u\|_{L^2_{\beta+n\gamma}}.
  \end{align*}
  Notice for $n\in\N$, $\theta\in(0,1)$, we have
  \begin{align*}
    &\hspace{1em}(1+\rho_n(y)\chi_{|y|\le r_3})^\theta (1+\rho_{n+1}(y)\chi_{|y|\le r_3})^{1-\theta}\\
    &\le C_{\theta}(1+\rho^\theta_n(y)\rho^{1-\theta}_{n+1}(y)\chi_{|y|\le r_3})\\
    &\le C_{\theta} + C_{\theta}\chi_{|y|\le r_3}\left\{
    \begin{aligned}
      &|y|^{-2(\theta n+(1-\theta)(n+1))}, &&\text{ if }n\ge 1, \\
      &C_\varepsilon|y|^{-2(1-\theta)}\log^\theta(|y|^{-1}+e), &&\text{ if }n= 0.
    \end{aligned}\right.
  \end{align*}
  Define
  \begin{equation*}
    \rho_\alpha(y):=\left\{
    \begin{aligned}
      &\frac{1}{|y|^{2\alpha}}\log(|y|^{-1}+e) &&\text{ if }\alpha\in[0,1),\\
      &\frac{1}{|y|^{2\alpha}} &&\text{ if }\alpha\ge 1.
    \end{aligned}\right.
  \end{equation*} Then by a interpolation theorem in appendix \ref{A_interpolation}, we have the final result:
  \begin{Thm}\label{IV_semigroup_estimate}
    Assume $\gamma\in[0,d)$, $\alpha\in[0,\infty)$, $\beta\ge 0$, then there exists $r_3\in(0,1)$ such that
    \begin{align}
      \|e^{t\B}u\|_{L^2_\beta} &\le \frac{C_{\alpha,d,r_3,\beta}(1+\rho_\alpha(y)\chi_{|y|\le r_3})}{(1+t)^\alpha}\|u\|_{L^2_{\beta+\alpha\gamma}}.
    \end{align}
  \end{Thm}

  \section{Estimate on the Semigroup $e^{tB}$ and Global Existence}
  In this section, we will give the proof on estimate on $e^{tB}$ and the proof of our main existence theorem \ref{Main_Theorem}. Define 
  \begin{align*}
  B = -\xi\cdot\nabla_x + L,\\
  A= -\xi\cdot\nabla_x -\nu,
  \end{align*}with domain $D_p(B):=\{f\in L^p_{\beta}(H^l_x):\xi\cdot\nabla_x f\in L^p_{\beta}(H^l_x)\}$. One can prove that $(B,D_p(B))$ generates a strongly continuous semigroup on $L^p_{\beta}(H^l_x)$ by regarding it as a bounded perturbation $K$ on operator $-\xi\cdot\nabla_x - \nu$, whose semigroup has explicit expression 
  \begin{align*}
  e^{t(-\xi\cdot\nabla_x-\nu)}u(x,\xi) := e^{-\nu t}u(x-t\xi,\xi).
  \end{align*}Also Fourier transform $\F_x$ on $x\in\Rd$ is an isomorphism on $L^2(\R^{d}_x)$ and $\F B=\B \F$, hence for $u\in D_p(B)$, $p\in[1,\infty]$, 
  \begin{align*}
    e^{tB}u = \F^{-1}e^{t\B}\F u,
  \end{align*}where the equality "$=$" means the almost everywhere equality on $\R^d_\xi\times\R^d_x$.
  \begin{Thm}\label{V_thm1}
    Let $\gamma\in[0,d)$, $p\in[1,2]$, $\alpha\in [0,\frac{d}{4}(\frac{2}{p}-1))$, $\beta,l\ge 0$. Suppose $u\in L^2_{\beta+\alpha\gamma}(H^l)\cap L^2_{\beta+\alpha\gamma}(L^p)$, then
    \begin{align}
      \|e^{tB}u\|_{L^\infty_\alpha(L^2_\beta(H^l))} \le
      C_{\alpha,r_3,d,\beta,p}\left(\|u\|_{L^2_{\beta+\alpha\gamma}(H^l)}
      + \|u\|_{L^2_{\beta+\alpha\gamma}(L^p)}\right).
    \end{align}
  \end{Thm}
  \begin{proof}
  For $l\ge 0$,
    \begin{align*}
      \|e^{tB}u\|_{L^2_{\beta+\alpha\gamma}(H^l)}
      &= \Big(\int_\Rd\int_\Rd(1+|\xi|)^{2\beta}(1+|y|)^{2l}|e^{t\B}\widehat{u}(y,\xi)|^2\,dyd\xi\Big)^{1/2}.
    \end{align*}Let $r_3>0$ be chosen in section \ref{subsection_nearorigin}.
    We split the integral on $y$ into two parts: $|y|\ge r_3$ and $|y|\le r_3$.
    Then on one hand,
    \begin{align*}
      &\Big(\int_{|y|\ge r_3}\int_\Rd (1+|\xi|)^{2\beta}(1+|y|)^{2l}|e^{t\B}\widehat{u}(y,\xi)|^2\,d\xi dy\Big)^{1/2}\\
      &\le C_{r_3,d,\beta}\Big(\int_{|y|\ge r_3}(1+|y|)^{2l}(1+t)^{-2\alpha} \|\widehat{u}(y,\xi)\|_{L^2_{\beta+\alpha\gamma}}^2\,d\xi dy\Big)^{1/2}\\
      &\le C_{r_3,d,\beta}(1+t)^{-\alpha} \|u(x,\xi)\|_{L^2_{\beta+\alpha\gamma}(H^l)}.
    \end{align*}
  On the other hand, let $(2q)'=p\in[1,\infty]$, by H\"older's inequality and Hausdorff–Young inequality, we have
  \begin{equation*}
  \begin{split}
    &\Big(\int_{|y|\le r_3}\int_\Rd (1+|\xi|)^{2\beta}(1+|y|)^{2l}|e^{t\B}\widehat{u}(y,\xi)|^2\,d\xi dy\Big)^{1/2}\\
    &\le \frac{C_{r_3,\beta,d}}{(1+t)^\alpha}\Big(\int_{|y|\le r_3}(1+|y|)^{2l} \Big(\|\widehat{u}\|_{L^2_{\beta+\alpha\gamma}} + |y|^{-2\alpha}\log(|y|^{-1}+e)\|\widehat{u}\|_{L^2_{\beta+\alpha\gamma}}\Big)^2\,dy\Big)^{1/2}\\
    &\le \frac{C_{r_3,\beta,d}}{(1+t)^\alpha}\|u\|_{L^2_{\beta+\alpha\gamma}(H^l)}
    + \frac{C_{r_3,\beta,d}}{(1+t)^\alpha}\||y|^{-2\alpha}\log(|y|^{-1}+e)\|_{L^{2q'}}
    \Big(\int_{|y|\le r_3}\|\widehat{u}(y,\xi)\|^{2q}_{L^2_{\beta+\alpha\gamma}}\,dy\Big)^{\frac{1}{2q}}\\
    &\le \frac{C_{r_3,\beta,d}}{(1+t)^\alpha}\|u\|_{L^2_{\beta+\alpha\gamma}(H^l)}
    + \frac{C_{r_3,\beta,d}}{(1+t)^\alpha}C_{q',\alpha}
    \|u(x,\xi)\|_{L^2_{\beta+\alpha\gamma}(L^{p})},
  \end{split}
  \end{equation*}
  provided $4\alpha q'\in[0,d)$. That is
  $\alpha\in[0,\frac{d}{4q'})=[0,\frac{d}{4}(\frac{2}{p}-1))$.
  Therefore, for any $t\ge 0$,
  \begin{align*}
    (1+t)^{\alpha}\|e^{tB}u\|_{L^2_\beta(H^l)} \le
    C_{\alpha,r_3,d,\beta,p}\big( \|u\|_{L^2_{\beta+\alpha\gamma}(H^l)}
    +\|u\|_{L^2_{\beta+\alpha\gamma}(L^p)}\big).
  \end{align*}
  \qe\end{proof}

  Futhermore, we need a estimate on the semigroup generated by $A$ on $L^p_\beta$.
  \begin{Lem}For $\alpha,l\ge 0$, $\beta\in\R$,
    \begin{align*}
      \|e^{tA}u\|_{L^p_\beta(H^l)} &\le C_{\nu_0}(1+t)^{-\alpha}\|u\|_{L^p_{\beta+\alpha\gamma}(H^l)}.
    \end{align*}
  \end{Lem}
  \begin{proof}
  The semigroup generated by $A$ with domain $\{f\in L^p_{\beta}(L^2_x):\xi\cdot\nabla_x f\in L^p_{\beta}(L^2_x)\}$ is $e^{tA}u(x,\xi) = e^{-t\nu(\xi)}u(x-t\xi,\xi)$.
    For $u\in D(A)\subset L^p_\beta({H^l})$, we have 
    \begin{align*}
      (1+|\xi|)^{-\alpha\gamma}|\F_xe^{tA}u(x,\xi)|
      &\le e^{-t\nu_0(1+|\xi|)^{-\gamma}}(1+|\xi|)^{-\alpha\gamma}|\F_xu(x-t\xi,\xi)| \\
      &\le C_{\nu_0}(1+t)^{-\alpha}|e^{-2\pi i y\cdot\xi t}u(y,\xi)|,
    \end{align*}since for $t,\alpha\ge 0$, $\sup_{x\ge 0} (xt)^{\alpha}e^{-\nu_0tx}\le C_{\nu_0}$.
    Thus
    \begin{align*}
      \|e^{tA}u\|_{L^p_\beta(H^l)} 
      &= \|(\int_\Rd(1+|y|)^{2l}|\F_x e^{tA}u(x,\cdot)|^2\,dy)^{1/2}\|_{L^p_\beta}\\
      &\le C_{\nu_0}(1+t)^{-\alpha}\|u\|_{L^p_{\beta+\alpha\gamma}(H^l)}.
    \end{align*}
  \qe\end{proof}
  \begin{Rem}
    This lemma shows that the weighted normed space $L^2_{\beta+\alpha\gamma}$ is essential for our analysis.
  \end{Rem}

  \begin{Lem}
    For $0\le\alpha<1<\alpha_0$,
    \begin{align}
      \int^t_0 \frac{1}{(1+t-s)^\alpha(1+s)^{\alpha_0}}\,ds
      &\le C_{\alpha,\alpha_0}\frac{1}{(1+t)^\alpha}.
    \end{align}
  \end{Lem}
  \begin{proof}
    \begin{align*}
      &\int^t_0 \frac{1}{(1+t-s)^\alpha(1+s)^{\alpha_0}}\,ds\\
      &\le \frac{1}{(1+t/2)^\alpha}\int^{t/2}_0\frac{1}{(1+s)^{\alpha_0}}\,ds
      +\frac{1}{(1+t/2)^{\alpha_0}}\int^t_{t/2}\frac{1}{(1+t-s)^\alpha}\,ds\\
      &\le C_{\alpha,\alpha_0}\Big(\frac{1}{(1+t)^\alpha}+\frac{1}{(1+t)^{\alpha_0+\alpha-1}}\Big) \le C_{\alpha,\alpha_0}\frac{1}{(1+t)^\alpha}.
    \end{align*}
  \qe\end{proof}

  Recall that $e^{tB}$ can be viewed as a semigroup on $L^p_\beta(H^l)$, for $p\in[1,2]$. Then by the Duhamel principle, we can have another boundedness of semigroup $e^{tB}$.

  \begin{Thm}
    Let $\gamma\in[0,d)$, $l\ge 0$, $p\in[1,2]$, $\alpha\in [0,\frac{d}{4}(\frac{2}{p}-1))$, $\beta>\frac{d}{2}$.
    Suppose $u\in  L^2_{\beta+\alpha\gamma}(H^l)\cap L^2_{\beta+\alpha\gamma}(L^p)$, then
    \begin{align}\label{estimate_semi}
      \|e^{tB}u\|_{L^\infty_\alpha(L^\infty_\beta(H^l))}
      &\le C_{\nu,\gamma,\alpha,r_3,d,\beta,p}\Big(\|u\|_{L^\infty_{\beta+\alpha\gamma}(H^l)}
      + \|u\|_{L^2_{\alpha\gamma}(L^p)}\Big).
    \end{align}
  \end{Thm}
  \begin{proof}
    Let $u\in L^2_{\beta+\alpha\gamma}(H^l)\cap L^2_{\beta+\alpha\gamma}(L^p)\subset L^2(\R^{2d})$, and $v=e^{tB}u$.
    Recall that $B=A+K$, by bounded perturbation, we have for $t\ge 0$,
    \begin{align}\label{V_Duhamel}
      v(t) = e^{tA}u + \int^t_0e^{(t-s)A}Ke^{sB}u\,ds.
    \end{align}
    Thus for $p\in [1,\infty]$, $\alpha_0>1$,
    \begin{align*}
      \|v(t)\|_{L^p_\beta(H^l)}
      &\le \|e^{tA}u\|_{L^p_\beta(H^l)} + \int^t_0\|e^{(t-s)A}Ke^{sB}u\|_{L^p_\beta(H^l)}\,ds\\
      &\le C_{\nu_0}(1+t)^{-\alpha}\|u\|_{L^p_{\beta+\alpha\gamma}(H^l)}
      + \int^t_0(1+t-s)^{-\alpha_0}\|Kv(s)\|_{L^p_{\beta+\alpha_0\gamma}(H^l)}\,ds\\
      &\le C_{\nu_0}(1+t)^{-\alpha}\|u\|_{L^p_{\beta+\alpha\gamma}(H^l)}
      + C_{\alpha,\alpha_0}(1+t)^{-\alpha}\, \|Kv\|_{L^\infty_\alpha(L^p_{\beta+\alpha_0\gamma}(H^l))}.
    \end{align*}
  Thus when $p=\infty$,
  \begin{align}\label{V_Linfty}
    \|v\|_{L^\infty_\alpha(L^\infty_\beta(H^l))}
    &\le C_{\nu_0,\gamma,\alpha,\alpha_0}\Big(\|u\|_{L^\infty_{\beta+\alpha\gamma}(H^l)}
    + \|v\|_{L^\infty_\alpha(L^\infty_{\beta+\alpha_0\gamma-\gamma-2}(H^l))}\Big).
  \end{align}
  Pick $\alpha_0\in (1, \frac{\gamma+2}{\gamma})$, then $\alpha_0\gamma-\gamma-2<0$. Using equation \eqref{V_Linfty} inductively, we have
  \begin{align*}
    \|v\|_{L^\infty_\alpha(L^\infty_\beta(H^l))}
    &\le C_{\nu_0,\gamma,\alpha,\alpha_0}\Big(\|u\|_{L^\infty_{\beta+\alpha\gamma}(H^l)}
    + \|v\|_{L^\infty_\alpha(L^\infty_0(H^l))}\Big).
  \end{align*}

   Recall the important property of $K$ from \ref{I_proerties_K} that for $p>\max(\frac{d}{d-\gamma},\frac{d}{2})$,
   $\theta\in(0,1)$,
  \begin{align}\label{V_eq138}
    \|Kf\|_{L^\infty_{\beta+\gamma+1}(H^l)}&\le C_{\gamma,d,q}\|f\|_{L^p_\beta(H^l)},\\
    \|Kf\|_{L^{q_\theta}_{\beta+\gamma+1}}&\le C_{\gamma,d,q,p,\theta}\|f\|_{L^{p_\theta}_\beta},\label{V_eq138_2}
  \end{align}
  with $\frac{1}{q_\theta} = \frac{\theta}{\infty} + \frac{1-\theta}{1},
    \frac{1}{p_\theta} = \frac{\theta}{p_0} + \frac{1-\theta}{1}$,
  where $p_0 = \frac{d}{d-\gamma}+\frac{d}{2}$.
  Pick a sequence $\{p_j\}^n_{j=1}\in(1,\infty)$ by letting
  \begin{align}\label{V_eq142}
    p_1=2,\ p_{j+1}=p_j + \frac{p_j(p_j-1)}{p_0-p_j},\ (j=1,2,\dots).
  \end{align}
  Then $\{p_j\}$ satisfies
  \begin{align}
    \frac{1}{p_{j+1}} = \frac{\theta}{\infty} + \frac{1-\theta}{1},\quad
    \frac{1}{p_j} = \frac{\theta}{p_0} + \frac{1-\theta}{1},\label{V_eq143}
  \end{align}where $\theta = \frac{1-1/p_j}{1-1/p_0}$.
  One can observe from \eqref{V_eq142} that $p_{j+1}-p_j\ge \frac{2}{p_0-2}$.
  Thus there exists a finite $n\in\N$ such that $p_{n-1}\le\max\left(\frac{d}{d-\gamma},\frac{d}{2}\right)<p_n$,
  so we can apply \eqref{V_eq138} to $p_n$. (Be careful that we can't use $p_{n+1}$, since $p_n$ may be larger than $p_0$ and we won't have \eqref{V_eq143} with $\theta>0$).
  Thus using Duhamel's formula \eqref{V_Duhamel} and the boundedness of $K$ \eqref{V_eq138}\eqref{V_eq138_2} inductively,
  \begin{align*}
    \|v\|_{L^\infty_\alpha(L^\infty_\beta(H^l))}
    &\le C_{\nu_0,\gamma,\alpha,\alpha_0,n}\Big(\|u\|_{L^\infty_{\beta+\alpha\gamma}(H^l)}
    +\sum^n_{j=0}\|u\|_{L^{p_j}_{\alpha\gamma}(H^l)}
    + \|v\|_{L^\infty_\alpha(L^2_0(H^l))}\Big).
  \end{align*}
  Pick $\beta>\frac{d}{2}$ to get that  $\|u\|_{L^{p_j}_{\alpha\gamma}(H^l)}\le\|u\|_{L^\infty_{\beta+\alpha\gamma}(H^l)}$, for $j=0,1,\dots,n$.
  Then by theorem \ref{V_thm1}, we have
  \begin{align*}
    \|v\|_{L^\infty_\alpha(L^\infty_\beta(H^l))}
    &\le C_{\nu_0,\gamma,\alpha,\alpha_0,n,r_3,d,\beta,p}\Big(\|u\|_{L^\infty_{\beta+\alpha\gamma}(H^l)}
    + \|u\|_{L^2_{\alpha\gamma}(L^p)}\Big).
  \end{align*}
  \qe\end{proof}

  The following lemma is well-studied in \cite{Ukai1982} and I will put the proof in appendix. 
  \begin{Lem}\label{V_lemma}Assume $\gamma\in[0,d)$, $\alpha\ge 0$.

    (1).
  For $l>\frac{d}{2}$, $\beta\in\R$,
  \begin{align}
    \|\Gamma(f,g)\|_{L^\infty_{\beta+\gamma}(H^l)}\le C_{\nu}\|f\|_{L^\infty_\beta(H^l)}\|g\|_{L^\infty_\beta(H^l)}.
  \end{align}

  (2).
  For $\beta>\frac{d}{2}-\gamma+\alpha\gamma$, $l\ge 0$,
  \begin{align}
    \|\Gamma(f,g)\|_{L^2_{\alpha\gamma}(L^1)}\le C_{\nu,\beta_0}\|f\|_{L^\infty_{\beta+\alpha\gamma}(H^l)}\|g\|_{L^\infty_{\beta+\alpha\gamma}(H^l)}.
  \end{align}
  \end{Lem}

  \begin{Thm} Assume the cross-section $q$ satisfies the angular cut-off assumption \eqref{I_cutoffassumption}.
  Assume $d\ge 3$, $\gamma\in[0,d)$, $l>\frac{d}{2}$, $\beta>\frac{d}{2}$, $p\in[1,2)$ such that $\frac{d}{4}(\frac{2}{p}-1)>1/2$. Let $\alpha\in [\frac{1}{2},\min(\frac{d}{4}(\frac{2}{p}-1),1))$. 
  There exists constants $A_0<1, A_1$ such that if the initial data $f_0\in L^\infty_{\beta+\alpha\gamma}(H^l)\cap L^2_{\alpha\gamma}(L^p)$ satisfies
  \begin{align}
    \|f_0\|_{L^\infty_{\beta+\alpha\gamma}(H^l)} + \|f_0\|_{L^2_{\alpha\gamma}(L^p)}\le A_0.
  \end{align}
  Denote $X=\{f\in L^\infty_{\alpha}(L^\infty_{\beta}(H^l)): \|f\|_{L^\infty_{\alpha}(L^\infty_{\beta}(H^l))}\le A_1\}$.
  Then the Cauchy problem to Boltzmann equation
  \begin{equation}\left\{
    \begin{aligned}
      f_t + \xi\cdot\nabla_x f &= Lf + \Gamma(f,f),\\
      f|_{t=0} &= f_0.
    \end{aligned}\right.
  \end{equation}
  posseses a unique solution $f=f(t)\in X \cap BC^0([0,\infty);L^\infty_{\beta}(H^l))\cap BC^1([0,\infty);L^\infty_{\beta-1}(H^{l-1}))$ and
  \begin{align}
    \|f\|_{L^\infty_{\alpha}(L^\infty_{\beta}(H^l))} + \|\partial_tf\|_{L^\infty_{\alpha}(L^\infty_{\beta-1}(H^{l-1}))}
    &\le C_{\nu,\gamma,\alpha,\beta,d}\Big(\|f_0\|_{L^\infty_{\beta+\alpha\gamma}(H^l)} + \|f_0\|_{L^2_{\alpha\gamma}(L^p)}\Big).
  \end{align}The uniqueness is taken in the sense that
  $f\in X$.
  \end{Thm}
  \begin{proof}
  By semigroup theory, it suffices to find the fixed point of
    \begin{align}
      \Phi[f] := e^{tB}f_0 + \int^t_0e^{(t-s)B}\Gamma(f(s),f(s))\,ds.
    \end{align}

    Now pick $\alpha\in (0,\frac{d}{4})\cap (0,1)$,
    $\beta>\frac{d}{2}$, then
    \begin{align}
      \|\Phi[f]\|_{L^\infty_\beta(H^l)}\label{V_esti_Phi}
      &\le \|e^{tB}f_0\|_{L^\infty_\beta(H^l)} +\int^t_0\|e^{(t-s)B}\Gamma(f(s),f(s))\|_{L^\infty_\beta(H^l)}\,ds.
    \end{align}
    For the first term, using \eqref{estimate_semi}, we have 
    \begin{align*}
      \|e^{tB}f_0\|_{L^\infty_\beta(H^l)}\le  \frac{C_{\nu,\gamma,\alpha,\beta,d}}{(1+t)^\alpha}\Big(\|f_0\|_{L^\infty_{\beta+\alpha\gamma}(H^l)} + \|f_0\|_{L^2_{\alpha\gamma}(L^p)}\Big).
    \end{align*}
    Pick $\alpha\in(0,\frac{d}{4})\cap (\frac{1}{2},1)$, which is non-empty since $d\ge 3$, then the second term in \eqref{V_esti_Phi} becomes
    \begin{align*}
      &\int^t_0\|e^{(t-s)B}\Gamma(f(s),g(s))\|_{L^\infty_\beta(H^l)}\,ds\\
      &\le C_{\nu,\gamma,\alpha,\beta,d} \int^t_0 \frac{1}{(1+t-s)^{\alpha}}\Big(\|\Gamma(f(s),g(s))\|_{L^\infty_{\beta+\alpha\gamma}(H^l)} + \|\Gamma(f(s),g(s))\|_{L^2_{\alpha\gamma}(L^1)}\Big)\,ds\\
      &\le C_{\nu,\gamma,\alpha,\beta,d} \int^t_0 \frac{1}{(1+t-s)^{\alpha}}
      \|f(s)\|_{L^\infty_{\beta}(H^l)}\|g(s)\|_{L^\infty_{\beta}(H^l)}\,ds\\
      &\le C_{\nu,\gamma,\alpha,\beta,d} \int^t_0 \frac{1}{(1+t-s)^{\alpha}(1+s)^{2\alpha}}\,ds\
      \|f\|_{L^\infty_{\alpha}(L^\infty_{\beta}(H^l))}\|g\|_{L^\infty_{\alpha}(L^\infty_{\beta}(H^l))}\\
      &\le C_{\nu,\gamma,\alpha,\beta,d} (1+t)^{-\alpha}
      \|f\|_{L^\infty_{\alpha}(L^\infty_{\beta}(H^l))}\|g\|_{L^\infty_{\alpha}(L^\infty_{\beta}(H^l))}.
    \end{align*}
    where the first inequality follows from \eqref{estimate_semi}. Thus
    \begin{align*}
      \|\Phi[f]\|_{L^\infty_\alpha(L^\infty_\beta(H^l))}
      &\le C_{\nu,\gamma,\alpha,\beta,d}\Big(\|f_0\|_{L^\infty_{\beta+\alpha\gamma}(H^l)} + \|f_0\|_{L^2_{\alpha\gamma}(L^p)} + \|f\|^2_{L^\infty_{\alpha}(L^\infty_{\beta}(H^l))}\Big)\\
      &=: C_1\Big(\|f_0\|_{L^\infty_{\beta+\alpha\gamma}(H^l)} + \|f_0\|_{L^2_{\alpha\gamma}(L^p)}\Big)+ C_2\|f\|^2_{L^\infty_{\alpha}(L^\infty_{\beta}(H^l))}.
    \end{align*}

    On the other hand, noticing that $
      \Gamma(f,f)-\Gamma(g,g) = \Gamma(f+g,f-g)$, we have
    \begin{align*}
      \|\Phi[f]-\Phi[g]\|_{L^\infty_\beta(H^l)}\notag
      &\le\int^t_0\|e^{(t-s)B}\Gamma((f+g)(s),(f-g)(s))\|_{L^\infty_\beta(H^l)}\,ds\\
      &\le C_{\nu,\gamma,\alpha,\beta,d} (1+t)^{-\alpha}
      \|f+g\|_{L^\infty_{\alpha}(L^\infty_{\beta}(H^l))}\|f-g\|_{L^\infty_{\alpha}(L^\infty_{\beta}(H^l))},
    \end{align*}and
    \begin{align*}
      \|\Phi[f]-\Phi[g]\|_{L^\infty_\alpha(L^\infty_\beta(H^l))}
      &\le C_2\|f+g\|_{L^\infty_{\alpha}(L^\infty_{\beta}(H^l))}\|f-g\|_{L^\infty_{\alpha}(L^\infty_{\beta}(H^l))},
    \end{align*}by letting $C_2>\frac{1}{C_1}$ large enough.
  Define $C_3 := \frac{1}{4C_1C_2}$, then if $\|f_0\|_{L^\infty_{\beta+\alpha\gamma}(H^l)} + \|f_0\|_{L^2_{\alpha\gamma}(L^p)}<C_3$, we can let
  \begin{align*}
    C_4:= 1-4C_1C_2\Big(\|f_0\|_{L^\infty_{\beta+\alpha\gamma}(H^l)} + \|f_0\|_{L^2_{\alpha\gamma}(L^p)}\Big)>0.
  \end{align*}
  Choose $C_5:=\frac{1}{2C_2}(1-\sqrt{C_4})$,
  then $
    C_2C_5^2-C_5+C_1(\|f_0\|_{L^\infty_{\beta+\alpha\gamma}(H^l)} + \|f_0\|_{L^2_{\alpha\gamma}(L^p)})=0$.

  Finally we pick a normed space
  \begin{align*}
    X:= \{f\in L^\infty_{\alpha}(L^\infty_{\beta}(H^l)):\|f\|_{L^\infty_{\alpha}(L^\infty_{\beta}(H^l))}\le C_5\}.
  \end{align*}
  Then $X$ is a complete with norm $\|\cdot\|_{L^\infty_{\alpha}(L^\infty_{\beta}(H^l))}$,
   and for $f\in X$,
  \begin{align*}
    \|\Phi[f]\|_{L^\infty_{\alpha}(L^\infty_{\beta}(H^l))}
    &\le C_1\Big(\|f_0\|_{L^\infty_{\beta+\alpha\gamma}(H^l)} + \|f_0\|_{L^2_{\alpha\gamma}(L^p)}\Big)+ C_2C_5^2
    = C_5,\\
    \|\Phi[f]-\Phi[g]\|_{L^\infty_\alpha(L^\infty_\beta(H^l))}
    &\le 2C_2C_5\|f-g\|_{L^\infty_{\alpha}(L^\infty_{\beta}(H^l))}
    = (1-\sqrt{C_5})\|f-g\|_{L^\infty_{\alpha}(L^\infty_{\beta}(H^l))}.
  \end{align*}
  This proves that $\Phi$ is a contraction map on $X$. Thus there exists a unique fixed point $f\in X$ to $\Phi$,
  which is the solution of
  \begin{align*}
    f := e^{tB}f_0 + \int^t_0e^{(t-s)B}\Gamma(f(s),f(s))\,ds.
  \end{align*}with 
  \begin{align*}
  \|f\|_{L^\infty_{\alpha}(L^\infty_{\beta}(H^l))}
  &\le C_5 = \frac{1}{2C_2}(1-\sqrt{C_4}) = \frac{2C_1}{1+\sqrt{C_4}}\Big(\|f_0\|_{L^\infty_{\beta+\alpha\gamma}(H^l)} + \|f_0\|_{L^2_{\alpha\gamma}(L^p)}\Big).
  \end{align*}
  For the continuity, we have
  \begin{align*}
    \|f(t+h)-f(t)\|_{L^\infty_\beta(H^l)} &\le \|(e^{hB}-I)e^{tB}f_0\|_{L^\infty_\beta(H^l)}\\
    &\qquad\qquad + \int^{t+h}_t\|e^{(t+h-s)B}\Gamma(f(s),f(s))\|_{L^\infty_\beta(H^l)}\,ds\\
    &\qquad\qquad + \|(e^{hB}-I)\int^t_0e^{(t-s)B}\Gamma(f(s),f(s))\,ds\|_{L^\infty_\beta(H^l)}.
  \end{align*} Then we can obtain from the continuity with respect to $t$ of semigroup $e^{tB}$ on $L^\infty_\beta(H^l)$ that $f\in C^0([0,\infty);L^\infty_\beta(H^l))$.
  Also $f_0\in L^\infty_{\beta+\alpha\gamma}(H^l) \subset D_\infty(B)$, so by the time decay estimate on $e^{tB}$, together with $\|Bu\|_{L^\infty_{\beta-1}(H^{l-1})}\le C\|u\|_{L^\infty_{\beta}(H^{l})}$, we have
  \begin{align*}
  \partial_t f &= Be^{tB}f_0 + \Gamma(f(s),f(s)) + \int^t_0Be^{(t-s)B}\Gamma(f(s),f(s))\,ds,\\
  \|\partial_t f\|_{L^\infty_{\beta-1}(H^{l-1})} &\le \frac{C}{(1+t)^\alpha}\Big(\|f_0\|_{L^\infty_{\beta+\alpha\gamma}(H^l)} + \|f_0\|_{L^2_{\alpha\gamma}(L^p)} + \|f\|^2_{L^\infty_\alpha(L^\infty_{\beta-1-\gamma}(H^{l-1}))} + \|f\|^2_{L^\infty_\alpha(L^\infty_{\beta}(H^{l}))}\Big).
  \end{align*} Thus $\partial_tf\in BC^0([0,\infty);L^\infty_{\beta-1}(H^{l-1}))$. Since $\|f_0\|_{L^\infty_{\beta+\alpha\gamma}(H^l)} + \|f_0\|_{L^2_{\alpha\gamma}(L^p)}<1$, we get the energy estimate on $\partial_tf$ which has time decay $(1+t)^{-\alpha}$. Hence by theorem \ref{semigroup_solution}, $f|_{t=0}=f_0$ and
  \begin{align*}
  \partial_t f + \xi\cdot\nabla_x f &= Lf + \Gamma(f,f).
  \end{align*}
  
  \qe\end{proof}

  \section{Appendix}

	\subsection{Semigroup Theory}
	Here we give some useful semigroup theory, one may refer to \cite{Engel1999,Gohberg1990} for more details.
	\begin{Def}\label{adjoint}
		Let $(A,D(A))$ be a linear unbounded densely defined operator from Hilbert space $H_1$ into Hilbert space $H_2$ with domain $D(A)$. Define 
		\begin{align*}
		D(A^*) = \Big\{y\in H_2\big| \sup_{0\neq x \in D(A)}\frac{\<Ax,y\>}{\|x\|} \Big\}.
		\end{align*}
		Take $y\in D(A^*)$, since $\overline{D(A)}=H_1$, the functional $F_y(x):=\<Ax,y\>$ has a unique bounded extension on $H_1$. Hence Riesz representation theorem ensures the existence of a unique $z\in H_1$ such that $F_y(x) = \<x,z\>$. Define $A^*y:=z$, then 
		\begin{align*}
		\<Ax,y\> = \<x,A^*y\>.
		\end{align*}The operator $A^*(H_2\to H_1)$ is linear and is called the adjoint of $A$.
	\end{Def}
	\begin{Thm}\label{semigroup_solution}
		Let $A(X\to X)$ be the generator of the strongly
		continuous semigroup $T(\cdot)$. Suppose that $f: [0,\infty) \to X$ is continuously differentiable on $[0,\infty)$. Then for each $x\in D(A)$, there exists a unique solution to
		\begin{equation*}\left\{
		\begin{aligned}
		u'(t) &= Au(t) + f(t),\quad 0<t<\infty,\\
		u(0) &= x. 
		\end{aligned}\right.
		\end{equation*}
		This solution is continuously differentiable and is given by
		\begin{align*}
		u(t) = T(t)x + \int^t_0T(t-s)f(s)ds,\quad t > 0.
		\end{align*} 
	\end{Thm}

	\begin{Thm}(Hille-Yosida Theorem)
		For a linear operator $(A,D(A))$ on a Banach space $X$, the following properties are equivalent.
		
		(a) $(A,D(A))$ generates a strongly continuous contraction semigroup.
		
		(b) $(A,D(A))$ is closed, densely defined, and for every $\lambda\in\C$ with $\Re\lambda>0$ one has $\lambda\in\rho(A)$ and
	$\|(\lambda I-A)^{-1}\|\le \frac{1}{\Re\lambda}$.
	\end{Thm} 
\begin{Def}
	A linear operator $(A,D(A))$ on a Banach space $X$ is called dissipative if
	$\|(\lambda I - A)x\|\ge \lambda \|x\|$
	for all $\lambda > 0$ and $x\in D(A)$.
\end{Def}
\begin{Prop}
	An operator $(A,D(A))$ is dissipative if and only if for every $x \in D(A)$ there exists $j(x) \in \{x'\in X':\<x,x'\>=\|x\|^2=\|x'\|^2\}$ such that
	\begin{align}
	\Re\<Ax,j(x)\>\le 0.
	\end{align}
\end{Prop}
\begin{Thm}\label{semigroup_dissipative}
	Let $(A,D(A))$ be a densely defined linear operator on a Banach
	space $X$. If both $A$ and its adjoint $A^*$ are dissipative, then the closure $\overline{A}$ of $A$ generates a contraction semigroup on $X$.
\end{Thm}
\begin{Thm}\label{semigroup_boundedpertur}
	Let $(A,D(A))$ be the generator of a strongly continuous semigroup
	$(T(t))_{t\ge 0}$ on a Banach space $X$
	satisfying $\|T(t)\|\le Me^{\omega t}$ for all $t \ge 0$
	and some $\omega\in\R$, $M\ge 1$. If $B\in L(X)$, then
	$C := A + B$ with $D(C) := D(A)$
	generates a strongly continuous semigroup $(S(t))_{t\ge 0}$ satisfying
	$\|S(t)\|\le Me^{(\omega+M\|B\|)t}$ and 
	\begin{align*}
	S(t)x = T(t)x + \int^t_0T(t-s)BS(s)x\,ds,
	\end{align*} for all $t \ge 0$, $x\in X$.
\end{Thm}

  \subsection{Hilbert-Schmidt Operator}
  Let $H_1$ be a seperable Hilbert space, $H_2$ be a Hilbert space.
  \begin{Def}
    $T\in L(H_1,H_2)$ is called a Hilbert-Schmidt operator if there exists an orthonomral basis $\{e_n\}^\infty_{n=1}$ of $H_1$ such that $
      \sum^\infty_{n=1}\|Te_n\|^2_{H_2}<\infty$.
  \end{Def}
  \begin{Thm}\label{A_Hilbert_Schmidt}
    If $T\in L(H_1,H_2)$ is a Hilbert-Schmidt operator, then $T$ is compact.
  \end{Thm}
  \begin{proof}
    Let $\{e_n\}^\infty_{n=0}$ be an orthonormal basis of $H_1$ such that
    $\sum^\infty_{n=1}\|Te_n\|^2_{H_2}<\infty$.
    Then for $x\in H_1$,
    \begin{align*}
      Tx = \sum^\infty_{n=0}(x,e_n)Te_n,
    \end{align*}since the right hand side is absolutely convergent.
    Define $T_n:H_1\to H_2$ by $
      T_n(x) := \sum^n_{k=1}(x,e_k)Te_k$.
    Then $T_n$ has finite rank and hence is compact.

  On the other hand, for $x\in H_1$ such that $\|x\|_{H_1}\le 1$, we have
  \begin{align*}
    \|Tx-T_nx\|^2_{H_2}&\le \|\sum^\infty_{k=n+1}(x,e_k)Te_k\|^2_{H_2}
    \le \|x\|_{H_1}^2\cdot \sum^\infty_{k=n+1}\|Te_k\|^2_{H_2},\\
    &\|T-T_n\|\le \left(\sum^\infty_{k=n+1}\|Te_k\|^2_{H_2}\right)^\frac{1}{2}\to 0,
  \end{align*}as $n\to\infty$.
  Thus $T$ is compact since it's the limit of a sequence of compact operators.
  \qe\end{proof}

  \begin{Thm}
    Suppose $L^2(X,\mu)$ and $L^2(Y,\lambda)$ are two seperable Hilbert space.
    If $k\in L^2(X\times Y,\mu\otimes\lambda)$, then
    \begin{align}
      Kf := \int_Y k(x,y)f(y)\,d\lambda(y) : L^2(Y,\lambda)\to L^2(X,\mu)
    \end{align}
    is a Hilbert-Schmidt operator.
  \end{Thm}
  \begin{proof}
    Let $\{e_n\}^\infty_{n=0}$ be an orthonormal basis of $L^2(Y,\lambda)$, then so is $\{\overline{e_n}\}^\infty_{n=0}$.
  Since $k\in L^2(X\times Y,\mu\otimes\lambda)$, we have $k(x,\cdot)\in L^2(Y,\lambda)$, for almost all $x\in X$. Thus $Ke_n$ is well-defined for almost all $x\in X$ and
  \begin{align*}
    Ke_n(x) = \int_Y k(x,y)e_n(y)\,d\lambda(y)&\\
    \sum^\infty_{n=0} \|Ke_n\|^2_{L^2(X,\mu)}
    =\int_X\sum^\infty_{n=0}\left|(k(x,\cdot),\overline{e_n})_{L^2(Y)}\right|^2\,d\mu(x)
    &=\|k\|^2_{L^2(X\times Y)}<\infty.
  \end{align*}
  \qe\end{proof}

  \subsection{Interpolation}
  Define $\|f\|_{L^p_\beta} = \|(1+|\xi|)^\beta f\|_{L^p(\Rd)}$, $p\in[1,\infty]$.
  \begin{Thm}\label{A_interpolation}
    Let $\beta\in\R$, $\gamma\in\R\setminus\{0\}$, $p\in[1,\infty]$. Suppose $T$ is a linear operator defined on $L^p_{\beta+n\cdot\gamma}$ such that
    \begin{align}
      \|Tf\|_{L^p_\beta} \le A_n\|f\|_{L^p_{\beta+n\cdot\gamma}},
    \end{align}for some constants $A_n>0$, for all $n\in\{0,1,2,\dots\}$.
    Let $\theta\in(0,1)$, $n,m\in\{0,1,2,\dots\}$, pick $\alpha=n\cdot\theta+m\cdot(1-\theta)$.
    Then $T$ is bounded linear operator from $L^p_{\beta+\alpha\cdot\gamma}$ to $L^p_\beta$, with
    \begin{align}
      \|Tf\|_{L^p_\beta} \le A^\theta_nA^{1-\theta}_m\|f\|_{L^p_{\beta+\alpha\cdot\gamma}}.
    \end{align}
  \end{Thm}
  To prove this thoerem, we will need the Hadamard's three lines thoerem.
  \begin{Thm}\label{A_Three}
    Let $F$ be an analytic function in the open strip $S=\{z\in\C:0<\Re\lambda <1\}$. Suppose $F$ is continuous and bounded on $\overline{S}$ with
    \begin{align*}
      |F(z)|\le \left\{
      \begin{aligned}
        A_0, \text{  if }\Re\lambda=0,\\
        A_1, \text{  if }\Re\lambda=1.
      \end{aligned}\right.
    \end{align*}for some positive constants $A_0,A_1$. Then for any $\theta\in[0,1]$, if $\Re z=\theta$, we have
    \begin{align*}
      |F(z)|\le A_0^{1-\theta}A_1^\theta.
    \end{align*}
  \end{Thm}

  \begin{proof}[Proof of Theorem \ref{A_interpolation}]
    1.
    With loss of generality, assume $m>n$. Fix $\theta\in(0,1)$ and Let
    \begin{align*}
      \alpha = n\cdot\theta+m\cdot(1-\theta).
    \end{align*}
    Notice $T$ is well-defined on $L^p_{\beta+\alpha\gamma}$, since $L^p_{\beta+\alpha\gamma}\subset L^p_{\beta+m\gamma}$.

    2. Let $f(\xi)=\sum^K_{k=1}a_ke^{i\alpha_k}\chi_{A_k}(\xi)$ be any simple complex function on $\Rd$, with $a_k>0$, $\alpha_k\in\R$, $\{A_k\}$ are pairwise disjoint bounded measurable subsets of $\Rd$. We would like to control
    \begin{align*}
      \|Tf\|_{L^p_\beta} = \sup_{\|g\|_{L^{p'}}\le 1, g\text{ is simple}}
      \left|\int_\Rd Tf(\xi)g(\xi)(1+|\xi|)^\beta\,d\xi\right|.
    \end{align*}
    Write $g=\sum^J_{j=1}b_je^{i\beta_j}\chi_{B_j}(\xi)$, with $b_j>0$, $\beta_j\in\R$, $\{B_j\}$ are pairwise disjoint bounded measurable subsets of $\Rd$.
    Suppose $z\in \{z\in\C:0\le\Re z\le 1\}$. Define
    \begin{align*}
      f_z(\xi):=\sum^K_{k=1}a_ke^{i\alpha_k}\chi_{A_k}(\xi)\cdot(1+|\xi|)^{(\alpha-nz-m(1-z))\gamma}.
    \end{align*}
    Then $f_\theta(\xi) = f(\xi)$, $f_0(\xi) = f(\xi)(1+|\xi|)^{(\alpha-m)\gamma}$, $f_1(\xi) = f(\xi)(1+|\xi|)^{(\alpha-n)\gamma}$. Define
    \begin{align*}
      F(z) :&= \int_\Rd T(f_z)g(\xi)(1+|\xi|)^\beta\,d\xi\\
      &=\sum^K_{k=1}\sum^J_{j=1}a_kb_je^{i\alpha_k}e^{i\beta_j}
      \int_\Rd T(\chi_{A_k}(\xi)(1+|\xi|)^{(\alpha-nz-m(1-z))\gamma})\chi_{B_j}(\xi)(1+|\xi|)^\beta\,d\xi.
    \end{align*}
  Now we need to check $F$ satisfies the assumptions in three-lines theorem \ref{A_Three}.

  (i). Claim: $F(z)$ is continuous and bounded in $\{0\le\Re z\le 1\}$.

  Indeed,
  \begin{align*}
    |F(z)|&\le \sum^K_{k=1}\sum^J_{j=1}a_kb_j\|T(\chi_{A_k}(1+|\cdot|)^{(\alpha-nz-m(1-z))\gamma})\|_{L^p_\beta}\|\chi_{B_j}(\xi)\|_{L^{p'}}\\
    &\le C\sum^K_{k=1}\|\chi_{A_k}\|_{L^p_{\beta+(\alpha-nz+mz)\gamma}}<\infty.
  \end{align*}
   Notice $A_k$ is bounded, there exists a open ball $B(0,R)$ with radius $R>0$ such that $A_k\subset B(0,R)$, for all $1\le k\le K$. Then for $z_1,z_2\in\{0\le\Re z\le 1\}$, by H\"older's inequality and the boundedness of $T$, we have
  \begin{align*}
    &|F(z_1)-F(z_2)|\\ &\le \sum^K_{k=1}\sum^J_{j=1}a_kb_j\left\|\chi_{A_k}(1+|\xi|)^{\beta+\alpha\gamma}\left|(1+|\xi|)^{(m-n)z_1\gamma}-(1+|\xi|)^{(m-n)z_2\gamma}\right|\right\|_{L^p}\|\chi_{B_j}(\xi)\|_{L^{p'}}\\
    &\le C_{\gamma,m,n,g}\sum^K_{k=1}\left\|\chi_{A_k}(1+|\xi|)^{\beta+\alpha\gamma}
    \min\{1,(1+|\xi|)^{(m-n)\gamma}\}\ln(1+|\xi|)\right\|_{L^p}|z_1-z_2|\\
    &\le C_{\gamma,m,n,R}|z_1-z_2|\to 0,
  \end{align*}
  as $|z_1-z_2|\to 0$.
  This proves the claim.

  (ii). Claim: $F(z)$ is analytic in $\{0<\Re z< 1\}$.

  Indeed, similarly, for $z,z_0\in\{0<\Re z< 1\}$, by H\"older's inequality and the boundedness of $T$ and noticing that $\{A_k\}^K_{k=1}$ is uniformly bounded, we have
  \begin{align*}
    &\bigg|\frac{F(z)-F(z_0)}{z-z_0} - \sum^K_{k=1}\sum^J_{j=1}a_kb_je^{i\alpha_k}e^{i\beta_j}\\
    &\qquad\qquad\qquad\qquad\qquad
    \int_\Rd T\Big(\chi_{A_k}\partial_z\Big[(1+|\cdot|)^{(\alpha-nz-m(1-z))\gamma}\Big]\Big|_{z=z_0}\Big)\chi_{B_j}(\xi)(1+|\xi|)^\beta\,d\xi\Big|\\
    &\le C_{\alpha,\beta,\gamma,K,J} \sum^K_{k=1}\Big\|\frac{(1+|\xi|)^{(m-n)z\gamma}-(1+|\xi|)^{(m-n)z_0\gamma}}{z-z_0}\\
    &\qquad\qquad\qquad\qquad\qquad\qquad\qquad\qquad-(m-n)\gamma(1+|\xi|)^{(m-n)z_0\gamma}\ln(1+|\xi|)\Big\|_{L^p(A_k)}\\
    &\le C_{\alpha,\beta,\gamma,K,J} \sum^K_{k=1}(m-n)\Big\|\Big((1+|\xi|)^{(m-n)(z_0+t(z-z_0))\gamma}-(1+|\xi|)^{(m-n)z_0\gamma}\Big)\ln(1+|\xi|)\Big\|_{L^p(A_k)},\\
    &\le C_{\alpha,\beta,\gamma,K,J} \sum^K_{k=1}(m-n)^2\Big\|(1+|\xi|)^{(m-n)(z_0+st(z-z_0))\gamma}(\ln(1+|\xi|))^2\Big\|_{L^p(A_k)}t|z-z_0|,
  \end{align*}
  for some $s,t\in(0,1)$.
  Notice $z\to z_0$ implies $t,s\to 0$ and $A_k$ are uniformly bounded, thus the limit $\lim_{z\to z_0}\frac{F(z)-F(z_0)}{z-z_0}$ exists and hence $F(z)$ is analytic.

  (iii). If $\Re z = 0$, by using H{\"o}lder's inequality and the boundedness of $T$, we have
  \begin{align*}
    |F(z)|\le \|Tf_z\|_{L^p_\beta}\|g\|_{L^{p'}}
    \le A_n\|f_z\|_{L^p_{\beta+n\gamma}}\|g\|_{L^{p'}}
    = A_n\|f\|_{L^p_{\beta+\alpha\gamma}}\|g\|_{L^{p'}}.
  \end{align*}
  If $\Re z =1$, similarly we have
  \begin{align*}
    |F(z)|\le \|Tf_z\|_{L^p_\beta}\|g\|_{L^{p'}}
    \le A_m\|f_z\|_{L^p_{\beta+m\gamma}}\|g\|_{L^{p'}}
    = A_m\|f\|_{L^p_{\beta+\alpha\gamma}}\|g\|_{L^{p'}}.
  \end{align*}

  Therefore we can apply the Hadamard's three-lines theorem. When $\Re z=\theta$,
  \begin{align*}
    |F(z)|\le A_n^\theta A_m^{1-\theta}\|f\|_{L^p_{\beta+\alpha\gamma}}\|g\|_{L^{p'}}.
  \end{align*}
  Thus for any simple function $f$ in $L^p_{\beta+\alpha\gamma}$,
  \begin{align*}
    \|Tf\|_{L^p_\beta} = \sup_{\|g\|_{L^{p'}}\le 1, g\text{ is simple}}|F(\theta)| \le A_n^\theta A_m^{1-\theta}\|f\|_{L^p_{\beta+\alpha\gamma}}.
  \end{align*}
  Since simple functions is dense in $L^p_{\beta+\alpha\gamma}$, we proved the theorem.
  \qe\end{proof}

  \subsection{Preliminary Lemmas and Properties of L}

  \begin{Lem}\label{I_lemma1}
    For $A_1,A_2>0$, $d\ge 3$. Denote $ b=\left(\frac{\xi+\xi_*}{2}\cdot\frac{\xi_*-\xi}{|\xi_*-\xi|}\right)\frac{\xi_*-\xi}{|\xi_*-\xi|}$.
    Then

    (1). For $\alpha\in(-\infty,d)$,
    \begin{align}\label{I_lemma11}
      \int_\Rd \frac{1}{|\xi_*-\xi|^\alpha}e^{-A_1|\xi_*-\xi|^2-A_2|b|^2}\,d\xi_*
      \le C_{\alpha,A_1,A_2,d}\frac{1}{1+|\xi|}.
    \end{align}

    (2). For $\alpha\in[0,d)$,
    \begin{align}\label{I_lemma12}
      \int_\Rd \frac{1}{|\xi_*-\xi|^\alpha}e^{-A_1|\xi_*|^2}\,d\xi_*
      =\int_\Rd \frac{1}{|\xi_*|^\alpha}e^{-A_1|\xi_*-\xi|^2}\,d\xi_*
      \le C_{\alpha,A_1,A_2,d}\frac{1}{(1+|\xi|)^\alpha}.
    \end{align}

    (3). For $\alpha\in[0,d)$, $\beta\in\R$,
    \begin{align}\label{I_lemma13}
      \int_\Rd\frac{1}{|\xi_*-\xi|^\alpha(1+|\xi_*|)^\beta}e^{-A_1|\xi_*-\xi|^2-A_2|b|^2}\,d\xi_*
      \le C_{\alpha,A_1,A_2,d}\frac{1}{(1+|\xi|)^{\beta+1}}
    \end{align}
  \end{Lem}
  \begin{proof}
    1. If $\alpha<d$,
    \begin{align*}
      &\int_\Rd \frac{1}{|\xi_*-\xi|^\alpha}\exp(-A_1|\xi_*-\xi|^2-A_2|b|^2)\,d\xi_*\\
      &= \int_\Rd \frac{1}{|\xi_*|^\alpha} \exp(-A_1|\xi_*|^2-\frac{A_2(|\xi_*|^2+2\xi_*\cdot\xi)^2}{4|\xi_*|^2})\,d\xi_*\\
      &= \int^\infty_0 \frac{1}{r^{\alpha-d+1}}\exp(-A_1r^2)\ \int_{\S^{d-1}} \exp(-\frac{A_2(r+2\xi_*\cdot\xi)^2}{4})\,d\sigma(\xi_*)dr.
    \end{align*}
    By integration over sphere, 
    \begin{align*}
      \int_{\S^{d-1}} \exp(-\frac{A_2(r+2\xi_*\cdot\xi)^2}{4})\,d\sigma(\xi_*)
      &= \frac{2\pi^{\frac{d-1}{2}}}{\Gamma(\frac{d-1}{2})}
      \int^1_{-1}\exp(-\frac{A_2(r+2s|\xi|)^2}{4})(1-s^2)^{\frac{d-3}{2}} ds\\
      &\le C_d \int^1_{-1}\exp(-\frac{A_2(r+2s|\xi|)^2}{4}) ds.
    \end{align*}If $|\xi|\ge 1$, we have
    \begin{align*}
      \int^1_{-1}\exp(-\frac{A_2(r+2s|\xi|)^2}{4}) ds &= \frac{1}{2|\xi|}\int^{r+2|\xi|}_{r-2|\xi|}\exp(-\frac{A_2s^2}{4})\,ds \le C_{d,A_2}\frac{1}{1+|\xi|},
    \end{align*}and if $|\xi|\le 1$, we have
    \begin{align*}
      \int^1_{-1}\exp(-\frac{A_2(r-2s|\xi|)^2}{4}) ds &\le 2\le \frac{C}{1+|\xi|}.
    \end{align*}
    These estimate are independent of $r$ and we obtain \eqref{I_lemma11}.

    2. Fix $0\le\alpha<d$. If $|\xi|\le 1$, we have 
    \begin{align*}
      &\int_\Rd \frac{1}{|\xi_*-\xi|^\alpha}\exp(-A_1|\xi_*|^2)\,d\xi_*\\
      &=\int_{|\xi_*-\xi|>2}\frac{1}{|\xi_*-\xi|^\alpha}\exp(-A_1|\xi_*|^2)\,d\xi_*
      +\int_{|\xi_*-\xi|\le2}\frac{1}{|\xi_*-\xi|^\alpha}\exp(-A_1|\xi_*|^2)\,d\xi_*\\
      &\le \int_{|\xi_*|>1}\frac{1}{2^\alpha}\exp(-A_1|\xi_*|^2)\,d\xi_*
      +\int_{|\xi_*-\xi|\le2}\frac{1}{|\xi_*-\xi|^\alpha}\,d\xi_*\\
      &\le C_{A_1,d,\alpha}.
    \end{align*}
    If $|\xi|\ge 1$, notice $|\xi_*-\xi|\le\frac{|\xi|}{2}$ implies $|\xi_*|\ge|\xi|/2\ge|\xi_*-\xi|$, we have
    \begin{align*}
      &\int_\Rd \frac{1}{|\xi_*-\xi|^\alpha}\exp(-A_1|\xi_*|^2)\,d\xi_*\\
      &\le \int_{|\xi_*-\xi|>\frac{|\xi|}{2}}\frac{1}{|\xi_*-\xi|^\alpha}\exp(-A_1|\xi_*|^2)\,d\xi_*
      +\int_{|\xi_*-\xi|\le\frac{|\xi|}{2}}\frac{1}{|\xi_*-\xi|^\alpha}\exp(-A_1|\xi_*|^2)\,d\xi_*\\
      &\le \int_{|\xi_*-\xi|>\frac{|\xi|}{2}}\frac{1}{(\frac{|\xi|}{2})^\alpha}\exp(-A_1|\xi_*|^2)\,d\xi_*\\
      &\qquad\qquad\qquad\qquad\qquad
      +\int_{|\xi_*-\xi|\le\frac{|\xi|}{2}}\frac{1}{|\xi_*-\xi|^\alpha}\exp(-\frac{A_1|\xi_*-\xi|^2}{2})\,d\xi_*\ \exp(-\frac{A_1|\xi|^2}{8})\\
      &\le C_{A_1,\alpha,d}\left(\frac{1}{|\xi|^\alpha}
      + \exp(-\frac{A_1|\xi|^2}{8})\right)\\
      &\le \frac{C_{A_1,\alpha,d}}{(1+|\xi|)^\alpha}.
    \end{align*}
  This proves \eqref{I_lemma12}.

  3. Notice that $|\xi_*|\le |\xi|/2$ implies $|\xi_*-\xi|\ge |\xi|/2$. Thus by \eqref{I_lemma11},
  \begin{align*}
    &\int_\Rd\frac{1}{|\xi_*-\xi|^\alpha(1+|\xi_*|)^\beta}\exp(-A_1|\xi_*-\xi|^2-A_2|b|^2)\,d\xi_*\\
    &\le \left(\int_{|\xi_*|>\frac{|\xi|}{2}} + \int_{|\xi_*|\le\frac{|\xi|}{2}}\right) \frac{1}{|\xi_*-\xi|^\alpha(1+|\xi_*|)^\beta}\exp(-A_1|\xi_*-\xi|^2-A_2|b|^2)\,d\xi_*\\
    &\le \frac{1}{(1+|\xi|/2)^\beta}\int_{|\xi_*|>\frac{|\xi|}{2}}\frac{1}{|\xi_*-\xi|^\alpha}\exp(-A_1|\xi_*-\xi|^2-A_2|b|^2)\,d\xi_*\\
    &\qquad+ \int_{|\xi_*|\le\frac{|\xi|}{2}} \frac{1}{|\xi_*-\xi|^\alpha}\exp(-\frac{A_1|\xi_*-\xi|^2}{2}-A_2|b|^2)\,d\xi_*\ \exp(-\frac{A_1|\xi|^2}{8})\\
    &\le \frac{C_{\beta,\alpha,A_1,A_2}}{(1+|\xi|)^\beta}\left(\frac{1}{1+|\xi|}+\exp(-\frac{A_1|\xi|^2}{8})\right).
  \end{align*}
  This gives \eqref{I_lemma13}.
  \qe\end{proof}

  \begin{proof}[Proof of Theorem \ref{I_ThmL1}]
  1. Firstly using $\M\M_* = \M'\M'_*$, we have
  \begin{align*}
    Lf &= \int_\Rd\int_{\S^{d-1}}\M^{\frac{1}{2}}_*\left(f'_*\left(\M^{\frac{1}{2}}\right)'+f'\left(\M^{\frac{1}{2}}\right)'\right)q(\xi-\xi_*,\theta)\,d\omega d\xi_*\\
    &\qquad - \M^{\frac{1}{2}}\int_\Rd\int_{\S^{d-1}}\M^{\frac{1}{2}}_*q(\xi-\xi_*,\theta)\,dnd\xi_*
    - f\int_\Rd\int_{\S^{d-1}}\M_*q(\xi-\xi_*,\theta)\,dnd\xi_*.
  \end{align*}
  It suffices to deal with the first term, which denoted by $I$. Consider the unit vector $m\in \text{Span}\{\xi-\xi_*,\omega\}$ such that $m\perp \omega$. Then performing a changing variable from $\omega$ to $m$, (where one need to use polar coordinate on $\S^{d-1}$ to do the change of variable,) we have
  \begin{align*}
    \int_\Rd\int_{\S^{d-1}}\M^{\frac{1}{2}}_*\left(\M^{\frac{1}{2}}\right)'f'_*q(\xi-\xi_*,\theta)\,d\omega d\xi_*
    = \int_\Rd\int_{\S^{d-1}}\M^{\frac{1}{2}}_*\left(\M^{\frac{1}{2}}\right)'_*f'q(\xi-\xi_*,\theta)\,d\omega d\xi_*.
  \end{align*}
  Thus by Fubini's theorem,
  \begin{align*}
    I &= 2\int_{\S^{d-1}}\int_\Rd\M^{\frac{1}{2}}_*\left(\M^{\frac{1}{2}}\right)'_*f'q(\xi-\xi_*,\theta)\, d\xi_*d\omega\\
    &=2\int_{\S^{d-1}}\int_\Rd\M^{\frac{1}{2}}(\xi_*+\xi)\M^{\frac{1}{2}}(\xi_*+\xi-(\xi_*\cdot\omega)\omega)f(\xi+(\xi_*\cdot\omega)\omega)q(\xi_*,\theta)\, d\xi_*d\omega\\
    &=4\int_{\S^{d-1}}\int_{\R_+}\int_{\P_\omega}\M^{\frac{1}{2}}(x+r\omega+\xi)\M^{\frac{1}{2}}(x+\xi)f(\xi+r\omega)q(x+r\omega,\theta)\, dxd\sigma(\xi_*)d\omega,
  \end{align*}
  where $\P_\omega$ is the hyperplane in $\Rd$ that is orthogonal to $\omega$ and contains the origin.
  Here we remark that one actually needs the boundedness on $K$ to make sure that Fubini's theorem can be applied.
  Then by change of variable $\xi_*\mapsto \xi_*-\xi$, we have 
  \begin{align*}
    I
    &=4\int_{\Rd}\int_{\P_{\xi_*}}\M^{\frac{1}{2}}(x+\xi_*+\xi)\M^{\frac{1}{2}}(x+\xi)f(\xi+\xi_*)\frac{q(x+\xi_*,\theta)}{|\xi_*|^{d-1}}\,dxd\xi_*\\
    &=4\int_{\Rd}\int_{\P_{\xi_*-\xi}}\frac{1}{(2\pi)^{-\frac{d}{2}}}
    \exp\left(-\frac{|x+a|^2}{2}-\frac{|b|^2}{2}-\frac{|\xi_*-\xi|^2}{8}\right)\frac{f(\xi_*)q(x+\xi_*-\xi,\theta)}{|\xi_*-\xi|^{d-1}}\,dxd\xi_*\\
    &=4\int_{\Rd}\int_{\P_{\xi_*-\xi}}\frac{e^{-\frac{|x|^2}{2}}}{(2\pi)^{-\frac{d}{2}}}q(x-a+\xi_*-\xi,\theta)\,dx\,
    \exp\left(-\frac{|b|^2}{2}-\frac{|\xi_*-\xi|^2}{8}\right)\frac{f(\xi_*)}{|\xi_*-\xi|^{d-1}}\,d\xi_*,
  \end{align*}
  where $\theta$ is angle between $x-a+\xi_*-\xi$ and $\xi_*-\xi$. This gives the expression of $k_1$.

  2. $\nu$ is defined by $
    \nu(\xi) = \int_\Rd\int_{\S^{d-1}}\M^{\frac{1}{2}}_*q(\xi-\xi_*,\theta)\,d\omega d\xi_*$.
  Recall the cut-off assumption \eqref{I_cutoffassumption}, then we have
  \begin{align*}
    \nu(\xi)
    &= q_0(2\pi)^{d/4}\int_\Rd \exp(-\frac{|\xi_*|^2}{4})|\xi-\xi_*|^{-\gamma}d\xi_*.
  \end{align*}
  On one hand, by \eqref{I_lemma12} in lemma \ref{I_lemma1}, we have
  \begin{align*}
    \nu(\xi)
    &= q_0(2\pi)^{d/4}\int_\Rd \exp(-\frac{|\xi_*|^2}{4})|\xi-\xi_*|^{-\gamma}\,d\xi_*
    \le C_{d,\gamma,q_0}(1+|\xi|)^{-\gamma}.
  \end{align*}
  On the other hand, noticing $|\xi_*-\xi|\le 1+|\xi_*|+|\xi|\le (1+|\xi_*|)(1+|\xi|)$, we have
  \begin{align*}
    \nu(\xi)
    &\ge q_0(2\pi)^{d/4}\int_\Rd \exp(-\frac{|\xi_*|^2}{4})(1+|\xi_*|)^{-\gamma}\,d\xi_*\,(1+|\xi|)^{-\gamma}\ge C_{d,\gamma,q_0}(1+|\xi|)^{-\gamma}.
  \end{align*}
  This proves statement (i).

  3. For the part $k_1$, we firstly estimate $J := \int_{\P_{\xi_*-\xi}}e^{-\frac{|x|^2}{2}} q(x-a+\xi_*-\xi,\theta)\,dx$. Notice $a\in\P_{\xi_*-\xi}$ and integral
  is taken in $\P_{\xi_*-\xi}$, we have
  \begin{align*}
    J &\le \int_{\P_{\xi_*-\xi}}e^{-\frac{|x|^2}{2}} |x-a+\xi_*-\xi|^{-\gamma} |\cos\theta|\,dx\\
    &\le \int_{\P_{\xi_*-\xi}}e^{-\frac{|x|^2}{2}} |x-a+\xi_*-\xi|^{-\gamma} \frac{|(x-a+\xi_*-\xi)\cdot(\xi_*-\xi)|}{|x-a+\xi_*-\xi|\,|\xi_*-\xi|}\,dx\\
    &= \int_{\P_{\xi_*-\xi}}e^{-\frac{|x|^2}{2}} \frac{|\xi_*-\xi|}{(|x-a|^2+|\xi_*-\xi|^2)^{(\gamma+1)/2}}\,dx.
  \end{align*}
  If $|\xi_*-\xi|\le 1$, by \eqref{I_lemma12} in lemma \ref{I_lemma1},
  \begin{align*}
    \frac{J}{|\xi_*-\xi|}
    &\le \int_{\P_{\xi_*-\xi}}e^{-\frac{|x|^2}{2}}\frac{1}{|x-a|^{\gamma+1}}\,dx
    \le C_{d,\gamma}\frac{1}{(1+|a|)^{\gamma+1}}.
  \end{align*}
  If $|\xi_*-\xi|> 1$, noticing $|x-a|\le\frac{|a|}{2}$ implies $|x|>\frac{|a|}{2}$, then
  \begin{align*}
    \frac{J}{|\xi_*-\xi|} &\le \left(\int_{\substack{\P_{\xi_*-\xi}\\|x-a|\le\frac{|a|}{2}}}+\int_{\substack{\P_{\xi_*-\xi}\\|x-a|>\frac{|a|}{2}}}\right)
    e^{-\frac{|x|^2}{2}}  \frac{1}{(|x-a|^2+|\xi_*-\xi|^2)^{(\gamma+1)/2}}\,dx\\
    &\le \int_{\substack{\P_{\xi_*-\xi}\\|x-a|\le\frac{|a|}{2}}}e^{-\frac{|x|^2}{4}-\frac{|a|^2}{16}} \frac{1}{|\xi_*-\xi|^{\gamma+1}}\,dx
    +\int_{\substack{\P_{\xi_*-\xi}\\|x-a|>\frac{|a|}{2}}}
    e^{-\frac{|x|^2}{2}}  \frac{1}{(|a|^2/4+|\xi_*-\xi|^2)^{(\gamma+1)/2}}\,dx\\
    &\le C_{d,\gamma}\left(e^{-\frac{|a|^2}{16}} \frac{1}{|\xi_*-\xi|^{\gamma+1}} +\frac{1}{(|a|+|\xi_*-\xi|)^{\gamma+1}}\right)\\
    &\le C_{d,\gamma}\frac{1}{(|a|+|\xi_*-\xi|)^{\gamma+1}},
  \end{align*}by using the fact that
  \begin{equation*}
    \sup_{a\in\Rd,|\xi_*-\xi|\ge 1}e^{-\frac{|a|^2}{16}} \frac{(|a|+|\xi_*-\xi|)^{\gamma+1}}{|\xi_*-\xi|^{\gamma+1}}
    \le
    \begin{cases}
      e^{-\frac{|a|^2}{16}}(2|a|)^{\gamma+1}, &\text{   if }|a|>|\xi_*-\xi|,\\
      2^{\gamma+1}, &\text{   if }|a|\le|\xi_*-\xi|.
    \end{cases}
  \end{equation*}
  Thus for any $\varepsilon\in(0,1)$,
  \begin{align*}
    |k_1(\xi,\xi_*)|
    &\le \frac{C_{d,\gamma}}{|\xi_*-\xi|^{d-2}}\frac{1}{(1+|a|+|\xi_*-\xi|)^{\gamma+1}}\ \exp(-\frac{|b|^2}{2}-\frac{|\xi_*-\xi|^2}{8}),\\
    &\le \frac{C_{d,\gamma}}{|\xi_*-\xi|^{d-2}}\frac{1}{(1+|a|+|b|+|\xi_*-\xi|)^{\gamma+1}}\ \exp(-(1-\varepsilon)(\frac{|b|^2}{2}-\frac{|\xi_*-\xi|^2}{8})),\\
    &\le \frac{C_{d,\gamma,\varepsilon}}{|\xi_*-\xi|^{d-2}}\frac{1}{(1+|\xi|+|\xi_*|)^{\gamma+1}}\ \exp(-(1-\varepsilon)(\frac{|b|^2}{2}-\frac{|\xi_*-\xi|^2}{8})),
  \end{align*}where we use the fact that $|a|^2+|b|^2 = \frac{|\xi_*+\xi|^2}{4}$.
  \qe\end{proof}

  \begin{proof}[Proof of Theorem \ref{I_ThmL2}]
    1. Notice $\nu$ is a bounded positive function and $K$ is linear continuous integral operator with symmetric kernel $k(\xi,\xi_*)$, so $L=\nu+K$ is self-adjoint.

    2. For the non-positiveness, we can use a well-known fact that
    \begin{align*}
      (Q(f,g),\psi)_{L^2} = \frac{1}{4}\int_\Rd\int_{\S^{d-1}}(f'_*g'+f'g'_*-f_*g-fg_*)q(\xi-\xi_*,\theta)(\psi+\psi_*-\psi'-\psi'_*)\,d\omega d\xi_*d\xi,
    \end{align*}which is valid whenever the integral is absolutely convergent.
    Then for $f\in L^2(\Rd)$, since $K$ is bounded on $L^2$ and $\nu\in L^\infty$, we can apply this identity to get
    \begin{align}\label{A_eq186}
      (Lf,f)_{L^2}
      = -\frac{1}{4}\int_\Rd\int_\Rd\int_{\S^{d-1}}
      \left|f'_*\left(\M^{\frac{1}{2}}\right)'+f'\left(\M^{\frac{1}{2}}\right)'_*-f_*\M^{\frac{1}{2}}-f\M^{\frac{1}{2}}_*\right|^2q(\xi-\xi_*,\theta)\,d\omega d\xi_*d\xi \le 0.
    \end{align}

  3. Let $f\in L^2$ such that $Lf=0$. Then $(Lf,f)_{L^2}=0$. Using \eqref{A_eq186} and the fact $\M\M_*=\M'\M'_*$, we have for a.e. $\xi,\xi_*\in\Rd$, $\omega\in\S^{d-1}$ that
  \begin{align*}
    f'_*\left(\M^{-\frac{1}{2}}\right)'_*+f'\left(\M^{-\frac{1}{2}}\right)' = f_*\M^{-\frac{1}{2}}_*+f\M^{-\frac{1}{2}}.
  \end{align*}
  Thus by the theory of collision invariant, $f\in$ Span$\{\M^{\frac{1}{2}}, \xi_1\M^{\frac{1}{2}},  \dots,  \xi_d\M^{\frac{1}{2}}, |\xi|^2\M^{\frac{1}{2}}\}$.
  \qe\end{proof}

\begin{proof}[Proof of lemma \ref{V_lemma}]
	1. Firstly we notice that 
	\begin{align*}
	1+|\xi|&\le (1+|\xi|)(1+|\xi_*|),\\
	1+|\xi|&\le 1+(|\xi'|^2+|\xi'_*|^2)^{1/2}\le 1+|\xi'|+|\xi'_*|\le (1+|\xi'|)(1+|\xi'_*|).
	\end{align*}
	Then for $f,g\in L^\infty_\beta(\Rd)$,
	\begin{align*}
	|\Gamma(f,g)|
	&= \M^{-1/2}\int_\Rd\int_{S^{d-1}}
	\Big|\M'^{1/2}f'\left(\M^{\frac{1}{2}}\right)'g'_* + \left(\M^{\frac{1}{2}}\right)'f'_*\M'^{1/2}g'\\
	& \qquad\qquad\qquad\qquad - \M^{1/2}f\M_*^{1/2}g_* - \M_*^{1/2}f_*\M^{1/2}g \Big|q(\xi-\xi_*,\theta)\,d\omega d\xi_*\\
	&= \int_\Rd\int_{S^{d-1}}
	\M_*^{1/2}\big(|f'g'_*| + |f'_*g'|+ |fg_*|+|f_*g|\big)q(\xi-\xi_*,\theta)\,d\omega d\xi_*\\
	&=: \Gamma_1(f,g)+\Gamma_2(f,g)+\Gamma_3(f,g)+\Gamma_4(f,g).
	\end{align*}
	Thus
	\begin{align*}
	&\|\Gamma(f,g)\|_{L^\infty_{\beta+\gamma}(\Rd)}\\
	&\le \sup_{\xi\in\Rd}\int_\Rd\int_{S^{d-1}}
	\M_*^{1/2}|f'g'_*| + |f'_*g'|+ |fg_*| +|f_*g|) (1+|\xi|)^{\beta+\gamma}q(\xi-\xi_*,\theta)\,d\omega d\xi_*\\
	&\le \sup_{\xi\in\Rd}(1+|\xi|)^\gamma \int_\Rd\int_{S^{d-1}}
	\M_*^{1/2}\Big(|(1+|\xi'|)^{\beta}f'(1+|\xi'_*|)^{\beta}g'_*| + |(1+|\xi'_*|)^{\beta}f'_*(1+|\xi'|)^{\beta}g'|\\
	&\qquad\qquad + |(1+|\xi|)^{\beta}f(1+|\xi_*|)^{\beta}g_*| +|(1+|\xi_*|)^{\beta}f_*(1+|\xi|)^{\beta}g|\Big)q(\xi-\xi_*,\theta)\,d\omega d\xi_*\\
	&\le 4\|f\|_{L^\infty_\beta}\|g\|_{L^\infty_\beta}
	\sup_{\xi\in\Rd}(1+|\xi|)^\gamma \int_\Rd\int_{S^{d-1}}
	\M_*^{1/2}q(\xi-\xi_*,\theta)\,d\omega d\xi_*\\
	&\le 4\nu_1\|f\|_{L^\infty_\beta}\|g\|_{L^\infty_\beta}.
	\end{align*}
	On the other hand, $l>\frac{d}{2}$ implies that the Sobolev space $H^l$ is a Banach algebra. Thus for $f,g\in L^\infty_\beta(H^l)$, $\|\Gamma_j(f,g)\|_{H^l}\le \Gamma_j(\|f\|_{H^l},\|g\|_{H^l})$. This proves assertion (1).
	
	2. For $\beta_0>\frac{d}{2}$, we have
	\begin{align*}
	\|f\|_{L^2_{\alpha\gamma}}&\le\|f\|_{L^\infty_{\beta_0+\alpha\gamma}}\|(1+|\xi|)^{-2\beta_0}\|_{L^2}\le C_{\beta_0}\|f\|_{L^\infty_{\beta_0+\alpha\gamma}}.
	\end{align*}
	For $f,g\in L^2_{\alpha\gamma}(L^1)$, $\beta_0>\frac{d}{2}$,
	\begin{align*}
	&\|\Gamma(f,g)\|_{L^2_{\alpha\gamma}(L^1)}\\
	&\le \|\int_\Rd\int_{S^{d-1}}
	\M_*^{1/2}\big(|f'g'_*| + |f'_*g'|+ |fg_*| +|f_*g|\big)
	q(\xi-\xi_*,\theta)\,d\omega d\xi_*\|_{L^2_{\alpha\gamma}(L^1)}\\
	&\le C_{\beta_0}\|\int_\Rd\int_{S^{d-1}}
	\M_*^{1/2}\big(\|f'g'_*\|_{L^1_x} + \|f'_*g'\|_{L^1_x}+ \|fg_*\|_{L^1_x} +\|f_*g\|_{L^1_x}\big) q(\xi-\xi_*,\theta)\,d\omega d\xi_*\|_{L^\infty_{\beta_0+\alpha\gamma}}\\
	&\le C_{\beta_0}\|\Gamma(\|f\|_{L^2_x},\|g\|_{L^2_x})\|_{L^\infty_{\beta_0+\alpha\gamma}}\\
	&\le C_{\beta_0,\nu_1}
	\|f\|_{L^\infty_{\beta_0-\gamma+\alpha\gamma}(L^2_x)}\|g\|_{L^\infty_{\beta_0-\gamma+\alpha\gamma}(L^2_x)}.
	\end{align*}
	If $\beta>\frac{d}{2}-\gamma+\alpha\gamma$, then we can find $\beta_0\in(\frac{d}{2},\beta+\gamma-\alpha\gamma)$ to make the above inequality valid. Then
	\begin{align*}
	\|\Gamma(f,g)\|_{L^2_{\alpha\gamma}(L^1)}&\le C_{\beta_0,\nu_1}\|f\|_{L^\infty_{\beta}(L^2_x)}\|g\|_{L^\infty_{\beta}(L^2_x)}
	\le C_{\beta_0,\nu_1}\|f\|_{L^\infty_{\beta}(H^l)}\|g\|_{L^\infty_{\beta}(H^l)}.
	\end{align*}
	\qe\end{proof}

{\bf Acknowledgment} The author would like to thank Professor Yang Tong and the
colleagues and reviewers for their valuable comments to this paper.


%
%
%

\bibliographystyle{plain}
\bibliography{1.bib}

\begin{thebibliography}{10}

\bibitem{Caflisch1980}
Russel~E. Caflisch.
\newblock {The Boltzmann equation with a soft potential I}.
\newblock {\em Communications in Mathematical Physics}, 74(1):71--95, feb 1980.

\bibitem{Caflisch1980a}
Russel~E. Caflisch.
\newblock {The Boltzmann equation with a soft potential II}.
\newblock {\em Communications in Mathematical Physics}, 74(2):97--109, jun
  1980.

\bibitem{Cercignani1994}
Carlo Cercignani, Reinhard Illner, and Mario Pulvirenti.
\newblock {\em {The Mathematical Theory of Dilute Gases}}, volume 106 of {\em
  Applied Mathematical Sciences}.
\newblock Springer Science+Business Media New York, 1994.

\bibitem{Ellis1975}
Richard~S. Ellis and Mark~A. Pinsky.
\newblock {The First and Second Fluid Approximations to the Linearized
  Boltzmann Equation}.
\newblock {\em Journal de Math$\acute{e}$matiques Pures et
  Appliqu$\acute{e}$es}, 54:125--156, 1975.

\bibitem{Engel1999}
Klaus-Jochen Engel, Rainer Nagel, Rainer Nagel, M.~Campiti, and T.~Hahn.
\newblock {\em {One-Parameter Semigroups for Linear Evolution Equations}}.
\newblock Springer New York, 1999.

\bibitem{Evans2010}
Lawrence~C. Evans.
\newblock {\em {Partial Differential Equations: Second Edition}}.
\newblock Graduate Studies in Mathematics. American Mathematical Society, 2010.

\bibitem{Gohberg1990}
Israel Gohberg, Seymor Goldberg, and Marinus Kaashoek.
\newblock {\em {Classes of Linear Operators Vol. I (Operator Theory: Advances
  and Applications) (v. 1)}}.
\newblock Birkhäuser, 1990.

\bibitem{Guo2003}
Yan Guo.
\newblock {Classical Solutions to the Boltzmann Equation for Molecules with an
  Angular Cutoff}.
\newblock {\em Archive for Rational Mechanics and Analysis}, 169(4):305--353,
  sep 2003.

\bibitem{Liu2011}
Tai-Ping Liu and Shih-Hsien Y.
\newblock {Solving Boltzmann equation I, Green function}.
\newblock {\em Bulletin of the Institute of Mathematics Academia Sinica(New
  Series)}, pages 115--243, 2011.

\bibitem{Seeley1964}
R.~T. Seeley.
\newblock {Extension of $C^\infty$ functions defined in a half space}.
\newblock {\em Proceedings of the American Mathematical Society}, pages
  625--625, 1964.

\bibitem{Strain2007}
Robert~M. Strain and Yan Guo.
\newblock {Exponential Decay for Soft Potentials near Maxwellian}.
\newblock {\em Archive for Rational Mechanics and Analysis}, 187(2):287--339,
  dec 2007.

\bibitem{Ukai1982}
Seiji Ukai and Kiyoshi Asano.
\newblock {On the Cauchy problem of the Boltzmann equation with a soft
  potential}.
\newblock {\em Publications of the Research Institute for Mathematical
  Sciences}, 18(2):477--519, 1982.

\bibitem{Ukai}
Seiji Ukai and Tong Yang.
\newblock {Mathematical Theory of Boltzmann Equation}.
\newblock Lecture Notes.

\bibitem{Yang2016}
Tong Yang and Hongjun Yu.
\newblock {Spectrum Analysis of Some Kinetic Equations}.
\newblock {\em Archive for Rational Mechanics and Analysis}, 222(2):731--768,
  may 2016.

\end{thebibliography}

\end{document}